\documentclass[12pt]{amsart}

\usepackage{tikz,amsthm,amsmath,amstext,amssymb,amscd,epsfig,euscript, mathrsfs, dsfont,pspicture,multicol,graphpap,graphics,graphicx,times,enumerate,subfig,sidecap,wrapfig,color,pict2e}
\usepackage[colorlinks,citecolor=cyan,pagebackref,hypertexnames=false]{hyperref}

\usepackage{enumerate}

 \numberwithin{equation}{section}

\def\bB{{\mathbb{B}}}
\def\bC{{\mathbb{C}}}
\def\bD{{\mathbb{D}}}
\def\bR{{\mathbb{R}}}
\def\R{{\mathbb{R}}}

\def\bZ{{\mathbb{Z}}}

\def\bN{{\mathbb{N}}}

\def\cB{{\mathscr{B}}}
\def\cC{{\mathscr{C}}}
\def\cD{{\mathscr{D}}}

\def\cF{{\mathscr{F}}}
\def\cG{{\mathscr{G}}}
\def\cH{{\mathscr{H}}}

\def\cS{{\mathscr{S}}}
\def\cT{{\mathscr{T}}}

\def\cW{{\mathscr{W}}}

\def\one{\mathds{1}}

\def\ve{\varepsilon}

\renewcommand{\d}{{\partial}}
\def\lec{\lesssim}
\def\gec{\gtrsim}

\DeclareMathOperator{\diam}{diam}
\def\BMO{\mathop\mathrm{BMO}} 					%BMO
\def\dist{\mathop\mathrm{dist}} 						%distance
\def\loc{\mathop\mathrm{loc}}						%locally
\newcommand{\ps}[1]{\left( #1 \right)}

\newcommand{\ck}[1]{\left\{#1 \right\}}
\newcommand{\av}[1]{\left| #1 \right|}

\newcommand{\cnj}[1]{\overline{#1}}

\newcommand{\nrm}[1]{\left|\left| #1 \right|\right|}

\def\warrow{\rightharpoonup}

\def\Xint#1{\mathchoice
{\XXint\displaystyle\textstyle{#1}}%
{\XXint\textstyle\scriptstyle{#1}}%
{\XXint\scriptstyle\scriptscriptstyle{#1}}%
{\XXint\scriptscriptstyle\scriptscriptstyle{#1}}%
\!\int}
\def\XXint#1#2#3{{\setbox0=\hbox{$#1{#2#3}{\int}$ }
\vcenter{\hbox{$#2#3$ }}\kern-.58\wd0}}

\def\avint{\Xint-}
\def\grad{\nabla}

\theoremstyle{plain}

\newtheorem{theorem}{Theorem}

\newtheorem{corollary}[theorem]{Corollary}

\newtheorem{lemma}[theorem]{Lemma}

\theoremstyle{definition}

\newtheorem{definition}[theorem]{Definition}
\newtheorem{remark}[theorem]{Remark}

\numberwithin{equation}{section}
\numberwithin{theorem}{section}

\newtheorem{main}{Theorem}

  \DeclareFontFamily{U}{mathb}{\hyphenchar\font45} 
\DeclareFontShape{U}{mathb}{m}{n}{
      <5> <6> <7> <8> <9> <10> gen * mathb
      <10.95> mathb10 <12> <14.4> <17.28> <20.74> <24.88> mathb12
      }{}
\DeclareSymbolFont{mathb}{U}{mathb}{m}{n}

\DeclareMathSymbol{\toitself}{3}{mathb}{"FD}  %*

\makeatletter
\def\@tocline#1#2#3#4#5#6#7{\relax
  \ifnum #1>\c@tocdepth % then omit
  \else
    \par \addpenalty\@secpenalty\addvspace{#2}%
    \begingroup \hyphenpenalty\@M
    \@ifempty{#4}{%
      \@tempdima\csname r@tocindent\number#1\endcsname\relax
    }{%
      \@tempdima#4\relax
    }%
    \parindent\z@ \leftskip#3\relax \advance\leftskip\@tempdima\relax
    \rightskip\@pnumwidth plus4em \parfillskip-\@pnumwidth
    #5\leavevmode\hskip-\@tempdima
      \ifcase #1
       \or\or \hskip 1em \or \hskip 2em \else \hskip 3em \fi%
      #6\nobreak\relax
    \dotfill\hbox to\@pnumwidth{\@tocpagenum{#7}}\par
    \nobreak
    \endgroup
  \fi}
\makeatother

\begin{document}

\def\Top{{\rm Top}}
\def\Tree{{\rm Tree}}
\def\Next{{\rm Next}}
\def\HD{{\rm HD}}
\def\LD{{\rm LD}}
\def\BLWG{{\rm BLWG}}
\def\Floor{{\rm Floor}}
\def\Bad{{\rm Bad}}
\def\Stop{{\rm Stop}}
\def\ND{{\rm ND}}

\def\WHSA{{\rm WHSA}}
\def\BAUP{{\rm BAUP}}
\def\CDHM{{\rm CDHM}}

\title{Harmonic Measure and the Analyst's Traveling Salesman Theorem}

\author[Azzam]{Jonas Azzam}

\address{Jonas Azzam\\
School of Mathematics \\ University of Edinburgh \\ JCMB, Kings Buildings \\
Mayfield Road, Edinburgh,
EH9 3JZ, Scotland.}
\email{j.azzam "at" ed.ac.uk}

\begin{abstract}
We study how the multiscale-geometric structure of the boundary of a domain $\Omega\subseteq\mathbb{R}^{d+1}$ relates quantitatively to the behavior of its harmonic measure $\omega_{\Omega}$. This has been well-studied in the case that the domain has boundary is Ahlfors regular and is uniformly rectifiable, a property that assumes scale-invariant estimates on the multi-scale flatness of the boundary, measured by the so-called Jones $\beta$-numbers. In this note we approach the same problem but without uniform estimates on either the measure of the boundary or on the $\beta$-numbers. Firstly, we generalize a result of Garnett, Mourgoglou and Tolsa by showing that domains in $\R^{d+1}$ whose boundaries are just lower $d$-content regular admit Corona decompositions for harmonic measure if and only if the square sum of the generalized Jones $\beta$-numbers is finite. Secondly, for semi-uniform domains with Ahlfors regular boundaries, it is known that uniform rectifiability implies harmonic measure is $A_{\infty}$ for semi-uniform domains, but now we give more explicit dependencies on the $A_{\infty}$-constant in terms of the uniform rectifiability constant. This follows from a more general estimate that does not assume the boundary to be uniformly rectifiable that relates a log integral of the Poisson kernel to the square sum of $\beta$-numbers. For general semi-uniform domains, we also show how to bound the harmonic measure of a subset in terms of that sets Hausdorff measure and the square sum of $\beta$-numbers on that set. 

Using these results, we give estimates on the fluctuation of Green's function in a uniform domain in terms of the $\beta$-numbers. As a corollary, for bounded NTA domains , if $B_{\Omega}=B(x_{\Omega},c\diam \Omega)$ is so that $2B_{\Omega}\subseteq \Omega$, we obtain that 
\[
(\diam \d\Omega)^{d} + \int_{\Omega\backslash B_{\Omega}} \av{\frac{\grad^2 G_{\Omega}(x_{\Omega},x)}{G_{\Omega}(x_{\Omega},x)}}^{2} \dist(x,\Omega^c)^{3} dx \sim  \cH^{d}(\d\Omega).
\]

\end{abstract}
\dedicatory{Dedicated to John Garnett on the occasion of his retirement}

\maketitle

\tableofcontents
\section{Introduction}

\def\hm{\omega_{\Omega}}

\subsection{Background}
Let $\Omega\subseteq \R^{d+1}$ be a domain and $\omega_{\Omega}$ denote its harmonic measure. This paper continues a long trend of trying to understand the quantitative relationship between the behavior of $\hm$ and the geometry of the boundary. 

One geometric feature of the boundary which has a well-established connection with $\hm$ is rectifiability. We will say that a measure $\mu$ is {\it  $d$-rectifiable} if it may be covered up to $\mu$ measure zero by $d$-dimensional Lipschitz graphs, and a set $E\subseteq \R^{n}$ is {\it $d$-rectifiable} if $\cH^{d}|_{E}$ is a $d$-rectifiable measure, where $\cH^{d}$ denotes $d$-dimensional Hausdorff measure. In one direction, the most general qualitative result says that for $\Omega\subseteq \R^{d+1}$ and $E\subseteq \d\Omega$, then $\hm|_{E}\ll \cH^{d}|_{E}$ implies $\hm|_{E}$ is $d$-rectifiable, and in fact there is $E'\subseteq E$ that is $d$-rectifiable and $\hm(E\backslash E')=0$ \cite{AHMMMTV16}, which was previously only known for simply connected planar domains \cite{Pom86}. In the reverse direction, there isn't a result quite as general: for rectifiability to imply absolute continuity, some fatness condition on the boundary is required. 

\begin{definition}\label{d:LBB}
A domain $\Omega\subset \bR^{d+1}$ is said to have {\it large complement} if there is $c>0$ so that 
\begin{equation}\label{e:Omdcontent}
\cH^{d}_{\infty} ( B\backslash \Omega)\geq c r_{B}^{d} \mbox{ for all $B$ centered on $\d\Omega$ with $0<r_{B}<\diam \d\Omega$}.
\end{equation}
We will say that $E\subseteq \R^{n}$ is {\it lower $d$-content regular} if there is $c>0$ so that 
\begin{equation}\label{e:Edcontent}
\cH^{d}_{\infty} ( B\backslash E)\geq c r_{B}^{d} \mbox{ for all $B$ centered on $E$ with $0<r_{B}<\diam E$}.
\end{equation}
\end{definition} 

A converse to the aforementioned theorem holds in this setting: if $\Omega\subseteq \R^{d+1}$ has large complement and $\hm|_{E}\subseteq \d\Omega$ is $d$-rectifiable for some $E\subseteq \d\Omega$, then $\hm|_{E}\ll \cH^{d}$ \cite{AAM19} (see also \cite{Wu86}). We must caution here that the definition of rectifiability of measures we are using here in describing these results is not standard: {\it Federer} rectifiability says a measure $\mu$ is covered up to $\mu$-measure zero by $d$-dimensional Lipschitz {\it images} of $\R^{d}$, as opposed to Lipschitz graphs. This is a really subtle point: When $\mu=\cH^{d}|_{E}$, then these two notions are equivalent, but this is not so for general measures, even for quite well-behaved measures like doubling measures \cite{GKS10}. It is not true that Federer rectifiability implies $\hm$ is absolutely continuous, although it does hold for simply connected planar domains \cite{BJ90}. To guarantee that $\cH^{d}|_{\d\Omega} \ll \hm|$, it is sufficient that $\d\Omega$ is rectifiable and to assume that the interior is not collapsing near $\d\Omega$, i.e. $\limsup_{r\rightarrow 0} |\Omega\cap B(x,r)|/|B(x,r)|>0$ for $\cH^{d}$-a.e. $x\in \d\Omega$. \cite{ABHM16}.

Our interest will be on the {\it quantitative} relationship between $\hm$ and the geometry of $\Omega$, and in fact many of the results above are deduced using quantitative methods. We will give a short synopsis of these results below, but for a good survey on the state-of-the-art concerning quantitative absolute continuity of harmonic measure, see \cite{Hof19}.

For quantitative results, it has been natural to work in the Ahlfors regular setting, since then many of the classical harmonic analytic techniques in Euclidean space can be repeated in this setting. 

\def\AR{{\rm AR}}
\begin{definition}
We say $E\subseteq \R^{d+1}$ is {\it $C$-Ahlfors $d$-regular} (or  $C$-AR) if 
\begin{equation}
\label{e:ar}
C^{-1} r^{d} \leq \cH^{d}(E\cap B(x,r))\leq Cr^{d} \;\; \mbox{ for all }x\in E, \;\; 0<r<\diam E.
\end{equation}

\end{definition}

The quantitative analogue of absolute continuity in this setting are the  $A_{\infty}$ and weak-$A_{\infty}$ conditions. 

\begin{definition}
If $\Omega\subseteq \R^{d+1}$ and $\d\Omega$ is AR, we will say $\hm\in A_{\infty}$ (resp. weak-$A_{\infty}$) if for every $\ve>0$ there is $\delta>0$ so that whenever $B$ is a ball centered on $\d\Omega$ with $0<r_{B}<\diam \d\Omega$ and $F\subseteq \d\Omega\cap B$, then $\cH^{d}(F)<\delta r_{B}^{d}$ implies $\hm^{x}(F)<\ve \hm^{x}(B)$ (resp. $\hm^{x}(F)<\ve \hm^{x}(2B)$) whenever $x\in \Omega\backslash 4B$.
\end{definition}

The quantitative analogue of rectifiability is uniform rectifiability. 

\begin{definition}
\label{d:UR}
A set $E\subset\R^{n}$ is  {\it uniformly  $d$-rectifiable} (UR) if it is $d$-AR and there are constants $\theta, M >0$ such that for all $x \in E$ and all $0<r\leq \diam E$ there is an $M$-Lipschitz mapping $g:B_d(0,r)\subseteq \R^{d}\rightarrow \R^{n}$ such that
$$
\cH^{d} (E\cap B(x,r)\cap g(B_d(0,r)))\geq \theta r^{n}.$$
\end{definition}

UR sets were  introduced by David and Semmes in connection to singular integrals on Ahlfors regular sets (see \cite{DS} and \cite{of-and-on}), however they appear very naturally in the study of harmonic measure as we shall see below.

It also turns out that, for quantitative results about harmonic measure, the connectivity of the domain plays an important role. We review some various kinds of connectivity. 

\begin{definition}
Let $\Omega\subseteq \bR^{d+1}$ be an open set. 
\begin{enumerate}
	\item For $x,y\in \cnj{\Omega}$, we say a curve $\gamma\subseteq \cnj{\Omega}$ is a {\it $C$-cigar curve} from $x$ to $y$ if $\min\{\ell(x,z),\ell(y,z)\}\leq C \dist(z,\Omega^{c})$ for all $z\in \gamma$, where $\ell(a,b)$ denotes the length of the sub-arc in $\gamma$ between $a$ and $b$.  We will also say it has {\it bounded turning} if  $\ell(\gamma)\leq C |x-y|$. 
\item If there is $x\in \Omega$ such that every $y\in \Omega$ is connected to $x$ by a curve$\gamma$ so that $\ell(y,z)\leq C \dist(z,\Omega^{c})$ for all $z\in \Gamma$,  we say $\Omega$ is {\it $C$-John}.
\item If every pair $x\in \Omega$ and $\xi\in \d\Omega$ are connected by a $C$-cigar with bounded turning, then we say $\Omega$ is {\it $C$-semi-uniform (SU)}.
\item One says that that $\Omega$ satisfies the {\it weak local John condition} (WLJC) with parameters $\lambda,\theta,\Lambda$ if there are constants $\lambda,\theta\in (0,1)$ and $\Lambda \geq 2$ such that for every $x\in\Omega$ there is a Borel subset $F\subset B(x,\Lambda \delta_\Omega(x))\cap \partial \Omega)$ with $\cH^{d}(F)\geq \theta\,\cH^{d}(B(x,\Lambda \delta_\Omega(x))\cap \partial \Omega)$ such that every $y\in F$ can be joined to $x$ by a $\lambda$-cigar curve.
\item If every $x,y\in \Omega$ are connected by a $C$-cigar of bounded turning, we say $\Omega$ is {\it uniform}. 
\item For a ball $B$ of radius $r_{B}$ centered on $\d\Omega$, we say $x\in B$ is an {\it interior/exterior $c$-corkscrew point}  or that $B(x,cr_{B})$ is an  {\it interior/exterior $c$-corkscrew ball} for $\Omega\cap B$if $B(x,2cr_{B})\subseteq B\cap \Omega$ (or $B(x,2cr_{B})\subseteq B\backslash \Omega$) . We say $\Omega$ satisfies the interior {\it $c$-Corkscrew condition} if every ball $B$ on $\d\Omega$ has a interior (or exterior) $c$-corkscrew point.
\item A uniform domain with exterior corkscrews is {\it nontangentially accessible} (NTA). 
\item An NTA with AR boundary is a {\it chord-arc domain} (CAD).
\end{enumerate}	
\end{definition}

The notion of an NTA domain was introduced by Jerison and Kenig in \cite{JK82}; there they codified many scale invariant properties for harmonic measure that had been known for Lipschitz domains, but their crucial observation was that it was not the Lipschitz structure but the nontangential connectedness that was guaranteeing these properties. Later, Aikawa and Hirata also observed that many of these properties {\it implied} good connectivity of the domain as well \cite{Aik06,Aik08,AH08}.

Various sufficient \cite{Dah77,DJ90, Sem90} and necessary \cite{HM14,HMU14, HM15,MT15,AHMNT17,HLMN17} conditions have been given for the $A_{\infty}$ and weak-$A_{\infty}$ to hold, and we will discuss more of these below. Particular mention should go to \cite{Dah77}, who showed that $\hm\in A_{\infty}$ when $\Omega$ is a Lipschitz domain, by which we mean $\hm = kd\sigma$ where $\sigma = \cH^{d}|_{\d\Omega}$ and
\begin{equation}
\label{e:expafin}
[\omega]_{A_{\infty}}:=\sup_{B} \exp\ps{\avint_{B}\log \frac{1}{k}d\sigma }\avint_{B} k d\sigma <\infty
\end{equation}
where the supremum is over all balls $B$ centered on $\d\Omega$ with $0<r_{B}<\diam \d\Omega$. This is equivalent to the more familiar definition of $A_{\infty}$ by \cite{Hru84} (see also \cite[V.6.6.3]{Big-Stein}). In fact, Dahlberg showed the stronger reverse H\"older inequality
\[
\ps{\avint_{B} k^2 d\sigma}^{\frac{1}{2}}\lec \avint_{B} kd\sigma.
\]
The works of David and Jerison \cite{DJ90}  and Semmes \cite{Sem90} who proved $\hm\in A_{\infty}$ when $\Omega$ is a CAD. In their proofs, they actually reduce things to Dahlberg's theorem by approximating a CAD domain quantitatively from within by Lipschitz domains.

Very recently, however, building on the techniques in these papers, a complete characterization has been obtained. In  \cite{Azz17}, we showed $\omega\in A_{\infty}$ if and only if $\Omega$ is SU and $\d\Omega$ is UR, although shortly after the definitive characterization of weak-$A_{\infty}$ was obtained (and implies the SU case). 

\begin{theorem}\label{t:AMT-HM} \cite{AHMMT19}
Let $\Omega\subset\R^{d+1}$, $d\geq2$, be an open set with AR boundary and interior corkscrew condition. Then $\d\Omega$ is UR and $\Omega$ satisfies the WLJC if and only if $\omega\in {\rm weak}-A_{\infty}$. 
\end{theorem}

Here, weak-$A_{\infty}$ means that there is $q>0$ so that if $\omega_{\Omega}^{x} = kd\cH^{d}|_{\d\Omega}$ and $x\in \Omega\backslash 4B$ where $B$ is centered on $\d\Omega$, and $\sigma = \cH^{d}|_{\d\Omega}$, then 
\[
\ps{\avint_{B} k^{1+q}d\sigma }^{\frac{1}{q}}\lec \avint_{2B} kd\sigma.
\]

\def\VMO{{\rm VMO}}
\def\BMO{{\rm BMO}}

With better information about harmonic measure comes better information about the geometry of the boundary. In particular, there is a class of ``small constant" results that effectively say that if harmonic measure is ``very" much like surface measure, then the boundary must be ``very" flat. In \cite{KT97} and \cite[Main Theorem]{KT03}, for example (and after reviewing the discussion after \cite[Definition 4.1.9]{HMT10}, which explains how some definitions in these results are equivalent) Kenig and Toro show that if a domain $\Omega\subseteq \R^{d+1}$ whose boundary is Ahlfors regular and sufficiently flat\footnote{The flatness condition originally stipulated that the boundary was Reifenberg flat and that the unit normal had sufficiently small BMO norm, although Bortz and Engelstein showed that this latter property implies the former \cite{BE17}} and $\omega = kd\cH^{d}|_{\d\Omega}$, then $\log k\in \VMO$ if and only if the unit normal vector on $\d\Omega$ is in $\VMO$. This in turn implies that the boundary has {\it very} big pieces of Lipschitz graphs with small Lipschitz constant \cite[Theorem 4.2.4]{HMT10}. In fact, in \cite{KT97}, Kenig and Toro also show that for all $\delta>0$, if $\d\Omega$ is Reifenberg flat and $\sigma(B)\leq (1+\ve)(2r_{B})^{d}$, then for $\ve>0$ small enough, there is $q>0$ so that 
\[
\ps{\avint_{B}k^{1+q}}^{\frac{1}{1+q}}\leq (1+\delta) \avint_{B} k.
\]

Moreover, with additional smoothness on $\log k$ comes additional smoothness of the boundary, see for example \cite{Eng16}.

In total, the (weak-)$A_{\infty}$ and VMO conditions on harmonic measure are well-studied, and UR plays a crucial role. We mention one last way of describing the quantitative relationship between the geometry and harmonic measure when it is not $A_{\infty}$. It may seem a bit ad hoc but is a very convenient form of quantitative absolute continuity.  For this result below, we will refer to Christ-David cubes, so see Theorem \ref{t:Christ} below if you are not familiar with these. For a measure $\mu$ and a set $A$, define
\[
\Theta_{\mu}^{d}(A) = \frac{\mu(A)}{(\diam A)^{d}}.
\]

\begin{definition}[Corona Decomposition for Harmonic Measure (CDHM)]
\label{d:CDHM}
Let $\Omega\subseteq \R^{d+1}$ be a domain and $E=\d\Omega\subseteq \R^{d+1}$ be lower $d$-content regular, $A> 1>\tau>0$, $\lambda\geq 1$, and let $Q_0\in \cD$  Suppose there are cubes $\Top$ in $Q_0$ and a partition $\{\Tree(R):R\in \Top\}$ of the cubes in $Q_0$ into stopping-time regions so that for each $R\in \Top$, there is a (interior) corkscrew ball $B(x_R,c\ell(R))\subseteq B_{R}\cap \Omega$ so that for all $Q\in \Tree(R)$,
\[
\tau \Theta_{\omega}^{x_{R}}(\lambda B_R) \leq \Theta_{\omega^{x_{R}}}^{d}(\lambda B_{Q})\leq A\Theta_{\omega}^{x_{R}}(\lambda B_R).
\]
For $Q_0\in \cD$, we let 
\[
\CDHM(Q_0,\lambda, A,\tau ) =\inf  \sum_{R\in \Top}\ell(R)^{d}.
\]
where the infimum is over all possible decompositions $\{\Tree(R):R\in \Top\}$ satisfying the conditions above.
\end{definition}

In \cite[Theorem 1.3]{GMT18}, Garnett, Mourgoglou and Tolsa showed that if $E=\d\Omega$ is Ahlfors regular and $\Omega$ has the interior corkscrew property, then for all $\lambda>1$ there are $A,\tau$ so that $E$ is UR if and only if 
\[
\CDHM(Q_0,\lambda, A,\tau )\lec \ell(Q_0)^{d} \;\; \mbox{ for all }Q_0\in \cD.
\]
In other words, while harmonic measure may not be $A_{\infty}$ in this theorem, the $\CDHM$ is the strongest statement one can make in the AR setting about how the density of $\hm$ behaves if the boundary is UR in the absence of any assumptions about connectivity.\\

The objective of this paper is to try and further quantify the behavior of harmonic measure, both in AR and (more importantly) non-AR settings. In particular, the results mentioned above usually assume some {\it uniform} control on harmonic measure or the boundary: the boundary is {\it uniformly} rectifiable, or the surface measure satisfies \eqref{e:ar} {\it uniformly} over all balls. The (weak)-$A_{\infty}$ condition on harmonic measure and the BMO/VMO conditions on $\log k$ are also statements that hold uniformly over all balls. We would instead like to study harmonic measure when either there the surface measure, the rectifiable structure, or our estimates on harmonic measure is allowed to vary between balls.

 As a motivating example, note that if $\Omega$ is a CAD so that $B(0,c)\subseteq \Omega$ and $\d\Omega\subseteq B(0,C)$ for some constants $c,C>0$, say, then the results of David, Jerison and Semmes (and using the scale invariant estimate Lemma \ref{l:w/w} below) imply that $\omega_{\Omega}^{0}$ is $A_{\infty}$. Now if $\Omega$ is simply uniform with Ahlfors regular boundary but not a CAD, then this can fail since $\hm$ may not be $A_{\infty}$. For instance, recall the {\it Garnett example} or {\it 4-corner Cantor set}: For $k=0,1,2,3$, let  $f_{k}(x)=(x/4+e^{i\pi/4}\frac{1}{2\sqrt{2}}) e^{ik\pi/2}$. Let $K_0=[-1/2,1/2]^{2}$ and for $j>0$ set
\[
K_j = \bigcup_{i=0}^{3} f_i(K_{j-1}). 
\]
Then $K:= \bigcap_{j=0}^{\infty}K_j$ is an Ahlfors $1$-regular set, and $\Omega= K^{c}$ is a uniform domain so that $\omega_{\Omega}\perp \cH^{1}$. See Figure \ref{f:cantor}.

\begin{figure}
\includegraphics[width=300pt]{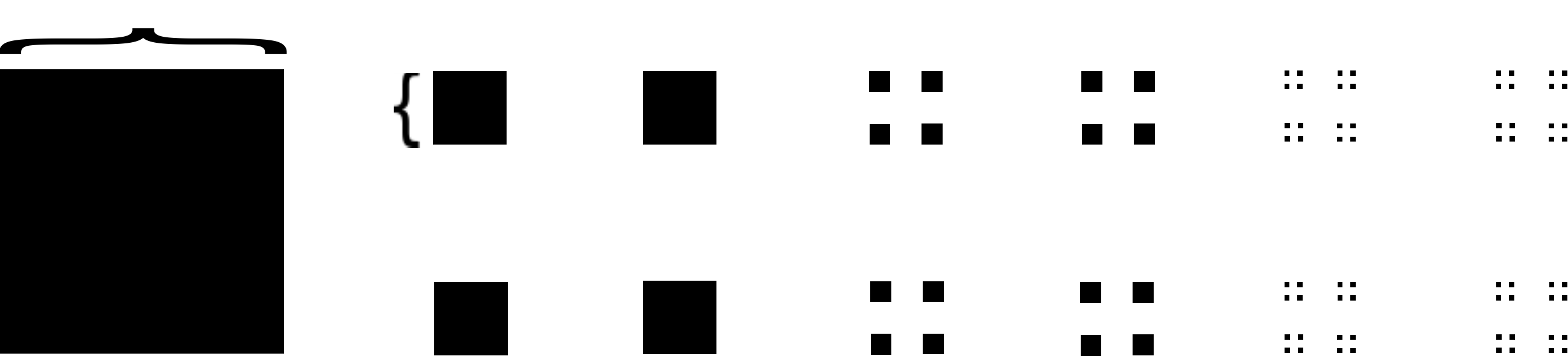}
\begin{picture}(0,0)(300,0)
\put(25,-10){$K_{0}$}
\put(105,-10){$K_{1}$}
\put(190,-10){$K_{2}$}
\put(250,-10){$K=\lim K_{j}$}
\put(22,65){$1$}
\put(65,45){$\frac{1}{4}$}
\end{picture}
\caption{The Garnett example, positioned so that 0 is at the center.}
\label{f:cantor}
\end{figure}

Note however that each of the $K_j^{c}$ are CADs that are uniformly AR and uniform, but the exterior corkscrew constant is  worsening with $j$. In particular, if $\omega_{j}=\omega_{K_{j}^{c}}^{0}$ are the respective harmonic measures for these domains with Radon-Nikodym derivatives $k_j$, then 
\begin{equation}
\label{e:cantor-int}
\exp \ps{ \avint_{\d K_{j}^{c} } \log \frac{1}{k_j} dx} \avint_{\d K_{j}^{c} } k_jdx
\end{equation}
must be going to infinity since the limit of these measures is $\omega_{K^{c}}\not\in A_{\infty}$. How fast should it go to infinity? If we took our domain to instead be the complement of a finite union of squares covering $K$ of different sizes instead of all being of size $4^{-j}$, how does this affect the integral? 

Another natural set of questions is what happens in the non-AR or non-uniform settings? The integrals above may not be useful anymore, but is there another way of quantitatively describing the behavior of harmonic measure? There are plenty of qualitative results in this setting about when $\hm$ is absolutely continuous or not, but less about the quantitative behavior. 

A first question may be what is the analogue of UR for sets that aren't AR? Or what geometric quantity or property of the boundary should we use in studying harmonic measure? What we found to be appropriate were $\beta$-numbers, which we now describe.

We first recall Jones' $\beta$-numbers: If $I$ is a cube in $\R^{d}$ and $E\subseteq \R^{d}$ is compact, define
\[
\beta_{E}(I)=\inf_{L}\sup_{x\in E\cap I}\dist(x,L)/\ell(I)
\]
where the infimum is over all lines. This measures in a scale invariant way how close $E$ is to being contained in a line in $I$. Jones showed in \cite{Jon90} for $d=2$ (and Okikiolu for general $d$ \cite{Oki92}) that the quantity
\[
\diam E+\sum_{I\cap E\neq\emptyset}\beta_{E}(3I)^{2} \ell(I)
\]
is comparable to the length of the shortest curve containing $E$ (where the sum is over all dyadic cubes $I$). This is called the {\it analyst's traveling salesman theorem}. There isn't a perfect analogue of this result for higher dimensional sets, first of all because it is not clear what the analogue of a curve should be, though there are some results that do generalize this in some sense. For technical reasons, a more suitable $\beta$-number is required (see the introduction to \cite{AS18} for a further discussion about why one is needed).

For arbitrary sets $E$ and $B$, $p>0$, and a $d$-dimensional plane $L$, define

\[
\beta_{E}^{d,p}(B,L)= \ps{\frac{1}{r_{B}^{d}}\int_{0}^{1}\cH^{d}_{\infty}(\{x\in B\cap E: \dist(x,L)>t r_{B}\})t^{p-1}dt}^{\frac{1}{p}}\]
where $2r_B=\diam B$, and set 
\[
\beta_{E}^{d,p}(B)=\inf\{ \beta_{E}^{d,p}(B,L): L\mbox{ is a $d$-dimensional plane in $\bR^{n}$}\}.\]

Note that if $E$ is Ahlfors regular and $B$ is a ball, then $\cH^{d}_{\infty}\sim \cH^{d}$, and the above integral is comparable to 
\[
\beta_{E}^{d}(B,L)\sim \inf_{L} \ps{\avint_{B}\ps{\frac{\dist(x,L)}{r_{B}}}^{2} d\cH^{d}|_{E}(x)}^{\frac{1}{2}},
\]
that is, the $L^{2}$-average distance to a $d$-dimensional plane. 

These numbers are variants of $\beta$-numbers defined by David and Semmes in \cite{DS}, where they showed that, if $E$ is Ahlfors $d$-regular, then $E$ is UR if and only if, for every $R\in \cD$,
\[
\beta_{E}(R):=\ell(R)^{d}+\sum_{Q\subseteq R} \beta_{E}^{d,2}(3B_Q)^{2}\ell(Q)^{d} \lec \ell(R)^{d}.
\]

We will call the sum the {\it linear deviation} of $E$ in $R$, since it measures how non-flat the set $E$ is at all scales and locations in $R$. Thus, UR is equivalent to having uniform control on the linear deviation inside every cube $R$. However, even if a set $E$ is not UR, this quantity can still be finite. For example, 
\begin{equation}
\label{e:nta-case}
\cH^{d}(Q_0)\sim \beta_{\d\Omega}^{d}(Q_0) \;\; \mbox{ if $\Omega$ is NTA and $Q_0\subseteq \d\Omega$}
\end{equation}
This follows from \cite[Corollary III]{AS18}. Moreover, in \cite{Vil19}, Villa extends this to more general surfaces than boundaries of NTA domains. Thus, this more resembles the original traveling salesman theorem of Jones mentioned earlier. 

The quantity $\beta_{E}(Q_0)$ still has some significance, even if $E$ is a set where $\beta_{E}(Q_0)$ and $\cH^{d}(Q_0)$ are not comparable. In fact, in \cite{AV19}, the author and Villa show that the linear deviation is comparable to many other quantities measuring how much $E$ deviates from satisfying certain geometric properties over all scales and locations. The main results below can be seen as a companion to this paper, where instead we now show how the linear deviation compares to different quantities involving harmonic measure and Green's function, and it seems to be a natural way of extending some results about harmonic measure to the non-AR setting. So instead of dealing with UR sets below, we will work with the class of sets $E$ where $\beta_{E}(R)$ is just finite for each cube $R$.

\subsection{Main Results}

Our first result is a version of Garnett, Mourgoglou and Tolsa's theorem for lower regular sets.

\begin{main}
\label{t:CDHM}
Let $\Omega\subseteq \R^{d+1}$ be a interior $c$-corkscrew domain with lower $d$-content regular boundary and let $\cD$ be the Christ-David cubes for $\d\Omega$. Then for all $\lambda \geq 1$ and for all $A,\tau^{-1}$ sufficiently large depending on $\lambda$, and for all $Q_0\in \cD$,
\begin{equation}
\label{e:CDHM}
\beta_{\d\Omega}(Q_0)\sim_{\lambda,A,\tau,c}  \CDHM(Q_0,\lambda, A,\tau )
\end{equation}
\end{main}

This has some nicer looking consequences when we know more about our domain. For example, we have the following:

\begin{main}
\label{t:SU-version}
Let $\Omega\subseteq \R^{d+1}$ be a SU domain with lower $d$-content regular boundary and let $\cD$ be the Christ-David cubes for $\d\Omega$. Then for $\lambda\geq 1$, for $M,A,\tau^{-1}$ large enough, and $Q_{0}\in \cD$, if  $\omega=\omega_{\Omega}^{x_{{0}}}$ where $x_{0}\in \Omega\backslash M\lambda B_{Q_{0}}$, there is a partition of the cubes in $Q_0$ into stopping-time regions $\{\Tree(R); R\in \Top\}$ so that 
\begin{enumerate}
\item $\tau \Theta_{\omega}^{d}(\lambda B_R)\leq \Theta_{\omega}^{d}(\lambda B_Q) \leq A \Theta_{\omega}^{d}(\lambda B_R)$ for all $Q\in \Tree(R)$.
\item We have 
\begin{equation}
\label{e:SU-top-sum}
\sum_{R\in \Top} \ell(R)^{d}\sim \beta_{\d\Omega}(Q_0).
\end{equation}
\end{enumerate}
That is, if the domain is SU, then we can improve over Theorem \ref{t:CDHM} by keeping the pole for $\omega$ in our trees the same.

If $\beta_{\d\Omega}(Q_0)<\infty$, then we can find trees so that the following holds: 
\begin{enumerate}
\setcounter{enumi}{2}
\item If $\Stop(R)$ denotes the minimal cubes of $\Tree(R)$ and $Q\in \Stop(R)$, then either $\Theta_{\omega}^{d}(Q)\sim \tau \Theta_{\omega}^{d}(\lambda B_R)$ or $\Theta_{\omega}^{d}(\lambda B_Q)\sim A \Theta_{\omega}^{d}(\lambda B_R)$ (with implied constants independent of $\tau,\lambda$ and $A$).
\end{enumerate}

\end{main}

We can use this and Theorem \ref{t:SU-version} to show how the $\beta$-numbers give estimates on the Poisson kernel for SU domains with AR boundary in the following theorem, which can be seen as an integral form of Theorem \ref{t:SU-version}.

\begin{main}
\label{t:Afin}
Let $\Omega\subseteq \R^{d+1}$ be a semi-uniform domain with Ahlfors regular boundary and let $\cD$ be the Christ-David cubes. There is $M>0$ depending on the semi-uniformity and Ahlfors regularity constants so that the following holds. For $Q_0\in \cD$, let $x_0\in  \Omega\backslash MB_{Q_{0}}$ and $\omega=\omega_{\Omega}^{x_0}$, and suppose $\omega|_{Q_{0}} \ll \cH^{d}$. Let $k$ be the Radon-Nikodym derivative of $\omega$ in $Q_0$. Then 
\begin{equation}
\label{e:log1/k}
1+ \avint_{Q_{0}}  \log \frac{1}{k}  d\cH^{d} +\log\ps{\avint_{Q_{0}} k d\cH^{d}} \sim \frac{\beta_{\d\Omega}(Q_0)}{\ell(Q_{0})^{d}}
\end{equation}
Above, the constants only depend on the semi-uniformity and Ahlfors regularity.
\end{main}

We prove a slightly more general statement than Theorem \ref{t:Afin} in Theorem \ref{t:Afin2} below that allows for the scenario when $\omega|_{Q_{0}}\not\ll \cH^{d}$.

Indeed, a word of caution: the above theorems do {\it not} say that, if $\beta_{\d\Omega}(Q_0)<\infty$, then $\hm\ll \cH^{d}$. In \cite{AMT15}, we showed with Mourgoglou and Tolsa that there was an NTA domain $\Omega\subseteq \R^{d+1}$ with $\cH^{d}(\d\Omega)<\infty$ so that $\hm\not\ll \cH^{d}$, meaning there is $E\subseteq \d\Omega$ with $\hm(E)>0=\cH^{d}(E)$. Since $ \beta_{\d\Omega}(R)\sim \cH^{d}(\d\Omega)<\infty$ by \cite[Corollary III]{AS18}, this means we can still find trees satisfying the conclusions of the above theorems, but the densities $\Theta_{\omega}^{d}(\lambda B_{Q})$ can still diverge on a set of positive harmonic measure. Thus, Theorems \ref{t:CDHM} and \ref{t:SU-version} cannot be improved to imply absolute continuity of harmonic measure. This counterexample, however, isn't relevant to Theorem \ref{t:Afin} since its boundary is not Ahlfors regular. Nonetheless, we show that we still cannot conclude absolute continuity in this setting due to the following example:

\begin{main}
\label{t:example}
There is a domain $\Omega\subseteq \bC$ so that $\d\Omega$ is $1$-Ahlfors regular, $\beta_{\d\Omega}(\d\Omega)<\infty$, and there is $E\subseteq \d\Omega$ so that $\omega(E)>0=\cH^{1}(E)$. In particular, $\Omega$ is a uniform domain with $1$-rectifiable and $1$-Ahlfors regular boundary whose harmonic measure has a singular set. 
\end{main}

However, we do have that, under the assumptions of the Theorem \ref{t:CDHM} (or any of the theorems above), if $\beta_{\d\Omega}(R)<\infty$ for all $R\in \cD$, then $  \cH^{d}|_{\d\Omega} \ll \omega_{\Omega}$. This follows by the main result of \cite{ABHM16} and the fact that $\d\Omega$ is rectifiable by \cite[Theorem II]{AS18}.

Theorem \ref{t:Afin} implies that, if $A=e^{-1}$, there are constants  $C_1,C_2>0$ so that 
\[
A\exp\ps{C_1\frac{\beta_{\d\Omega}(Q_0)}{\ell(Q_{0})^{d}}}\leq \exp\ps{\avint_{Q_{0}}\log\frac{1}{k}  d\cH^{d} }\avint_{Q_{0}} k d\cH^{d}
\leq A\exp\ps{C_2\frac{\beta_{\d\Omega}(Q_0)}{\ell(Q_{0})^{d}}}
\]
and this bounds the familiar term that formerly characterized the $A_{\infty}$ condition. In particular, if $\omega\ll \cH^{d}$ and $\beta_{\d\Omega}(Q_0)\leq C_0 \cH^{d}(Q_0)$ for all $Q_0\in \cD$, then this shows $\omega\in A_{\infty}$ in a much longer way, but now it is more transparent to see how the $A_{\infty}$-constant in \eqref{e:expafin} depends on the Carleson constant $C_0$: we see $[\omega]_{A_{\infty}}\lec  \exp(C_2C_0)$ where $C_2$ depends only on the Ahlfors regularity of $\d\Omega$ and the SU constants.

Even if the boundary is not UR, the $\beta$-numbers now allow us to estimate how badly harmonic measure fails to be $A_{\infty}$: recalling Garnett's example, one can show $\beta_{\d (K_j^{c})}(K_{j})\sim j$, and so the above theorem implies that \eqref{e:cantor-int} must be growing exponentially in $j$.

%In fact, even in the setting of Theorem \ref{t:SU-version}, one can have a nontrivial singular part. To see this, let $K$ be the Garnett example again and let $\omega=\omega_{K^{c}}^{0}$. Then $\omega\perp\cH^{1}$, so we can find $E\subseteq K$ compact with $\omega(E)>0=\cH^{1}(E)$. Now let $\mathscr{K}$ be maximal squares $I$ of the form $f_{i_{1}}\circ\cdots \circ f_{i_{n_{I}}}$ where $i_{j}\in \{0,1,2,3\}$ and $I\subseteq E^{c}$. Let 
%\[
%\Omega = \ps{E\cup \bigcup_{I\in \mathscr{K}}I}^{c}.
%\]
%Then $\Omega$ is uniform with Ahlfors $1$-regular boundary. 
%
We also obtain a local version of Theorem \ref{t:Afin}, which allows us to estimate the harmonic measure of a subset of the boundary of a SU domain directly in terms of that set's Hausdorff content and the $\beta$-numbers around that set:

\begin{main}
\label{t:hruscev}
Let $\Omega$ be a SU domain with LCR boundary. Let $Q_0\in \cD$ and $E\subseteq Q_0$. Then for $M$ large enough, $x_0\in \Omega\backslash MB_{Q_{0}}$, and $\omega=\omega_{\Omega}^{x}$
\[
\frac{\omega(E)}{\omega(Q_{0})}
\geq \exp\ps{ -C\frac{\ell(Q_{0})^{d}}{\cH_{\infty}^{d}(E)} \exp\ps{C\frac{\sum_{Q\subseteq Q_0\atop Q\cap E\neq\emptyset}\beta_{\d\Omega}(MB_{Q})\ell(Q)^{d}}{\cH_{\infty}^{d}(E)}}}.
\]
\end{main}

Finally, we give a continuous version of the above results for uniform domains which relates the linear deviation of the boundary of a uniform domain to the {\it affine} deviation of Green's function in the domain

\begin{main}
\label{t:green}
Let $\Omega\subseteq \R^{d+1}$ be a bounded uniform domain with lower $d$-content regular boundary. Let $B_{\Omega} = B(x_{\Omega},c\diam \d\Omega)$ be so that $2B_{\Omega}\subseteq \Omega$ and $g=G_{\Omega}(x_{\Omega},\cdot)$. Then
\begin{multline}
(\diam \d\Omega)^{d} + \int_{\Omega\backslash B_{\Omega} } \av{\frac{\grad^2 g(x)}{g(x)}}^{2} \dist(x,\Omega^c)^{3} dx \\
\sim (\diam \d\Omega)^{d} + \sum_{Q\subseteq \d\Omega} \beta_{\d\Omega}^{2}(3B_{Q})^{2}\ell(Q)^{d}.
\end{multline}
In particular, if $\Omega$ is an NTA domain, then
\begin{equation}
(\diam \d\Omega)^{d} + \int_{\Omega\backslash B_{\Omega}} \av{\frac{\grad^2 g(x)}{g(x)}}^{2} \dist(x,\Omega^c)^{3} dx \sim  \cH^{d}(\d\Omega)
\end{equation}

\end{main}

The left side measures affine deviation as $\av{\frac{\grad^2 g}{g\dist(x,\Omega^c)^2}} $ measures in a scale and dilation invariant way how close $g$ is to being linear around $x$. This has an analogy with a result of Bishop and Jones in the complex plane: For a conformal mapping $\phi:\bD\rightarrow \Omega$, its Schwarzian derivative is defined as 
\[
\mathcal{S}\phi  = \ps{\frac{\phi''}{\phi'}}^{'} -\frac{1}{2}\ps{\frac{\phi''}{\phi'}}^{2}\\
 = g''-\frac{1}{2} (g')^{2} \;\;\; \mbox{where} \;\;\;
g=\log \phi'. 
\]

Much like how $\grad^{2}g/g$ measures how affine $g$ is (in the sense that $g$ is affine if this is zero), the Schwarzian derivative of a conformal map $\phi$ measures how close it is to being a M\"obius transformation (that is, $\phi$ is a M\"obius transformation if $S\phi$ is identically zero). 

The following result of Bishop and Jones (see also \cite[Chapter X, Lemma 6.1 and Theorem 6.2]{Harmonic-Measure} relates this quantity measuring deviation from being a M\"obius map to the size of the boundary. 

\begin{theorem}
\label{t:BJ94}
 \cite{BJ94} Let $\Omega\subseteq \bC$ be a simply connected planar domain and $\phi:\bD\rightarrow \Omega$ be the Riemann mapping. If $\d\Omega$ is rectifiable, then
\[
\diam \Omega+\int_{\Omega}|\phi'(z)| |\mathcal{S}\phi(z)|^{2} \dist(z,\Omega^{c})^{3}dz \lec \cH^{1}(\d\Omega). 
\]
The opposite inequality holds if $\Omega$ is uniform. 
\end{theorem}

In fact, the papers \cite{BJ90,BJ94} are the first to explore the relationship between $\beta$-numbers and harmonic measure (and conformal mappings). 

Also see \cite{GEM18}, where they also study the fluctuation of a certain smoothed form of the distance to the boundary and relate that to the (uniform) rectifiable structure of the boundary for domains whose boundaries have higher codimension. \\

The motivation for comparing all these quantities to the $\beta$-numbers as opposed to some other quantity is twofold. Firstly, the $\beta$-numbers are more precise to estimate, and they give the most useful information (for example, one can use them directly to construct bi-Lipschitz David-Toro parametrizations, see Lemma \ref{l:DT-dyadic} below). Secondly, in \cite{AV19},  we show that the linear deviation is comparable to a trove of other quantities that measure the multiscale geometry of a set. For example, below we will use one such result that the linear deviation is comparable to the sum of cubes where the BAUP condition fails that were originally developed and studied by David and Semmes for uniformly rectifiable sets \cite{DS,of-and-on}. Thus, one can more easily estimate harmonic measure or Green's function using the above theorems if one knows something about the multiscale geometry. We already saw how this allowed us to compute the integral in \eqref{e:log1/k} in the case of Garnett's example. More generally, due to the BAUP estimate in Theorem \ref{t:BAUP} helow and Theorem \ref{t:Afin} above, for every cube $Q\in \cD$ around which $\d\Omega$ does not resemble a finite union of planes, that contributes approximately $\ell(Q)^{d}$ to the size of the integral in \eqref{e:log1/k}.

\subsection{Outline}

Theorem \ref{t:CDHM} and \ref{t:SU-version} will be proven in Sections \ref{s:I-II-Part-I} and \ref{s:I-II-Part-II} below. In fact, Theorem \ref{t:SU-version} can be seen as a corollary of the proof of \ref{t:CDHM} and will be proven simultaneously. Showing the $\gec$ estimate in Theorem \ref{t:CDHM} is the focus of Section \ref{s:I-II-Part-I}. Observe that in the proof of the analogus result in \cite{GMT18} that UR implies the CDHM, they instead prove that  $\ve$-approximability implies the CDHM, and then use the result of Hofmann, Martell and Mayboroda \cite{HMM14} that UR implies $\ve$-approximability (see these papers for the definition of $\ve$-approximability). Instead, our main idea is similar to the proofs \cite{Sem90} and \cite{DJ90}: we carefully construct some chord-arc subdomains where we know the $A_{\infty}$ property holds for harmonic measure and apply the maximum principle to get estimates on how the densities behave. However, some care is needed since unlike \cite{DJ90}, we will require infinitely many such subdomains, and we don't have nice corona decompositions by interior Lipschitz domains as in \cite{Sem90} or \cite{HMM14} to work with, for example, since we aren't assuming the boundary is UR. To overcome this, we use the David-Toro parametrizations from \cite{DT12} to build them from scratch. Ultimately, the set-up for the stopping-time procedure is similar to that in \cite{GMT18}, but we need to add in a few more stopping-time conditions. 

The $\lec$ estimate in Theorem \ref{t:CDHM} is based on the proof of the main result in \cite{HLMN17}, however some care is needed since in our setting we do not assume Ahlfors regularity. We will need to use a few new results from \cite{AV19}: the corona decomposition of lower regular sets by Ahlfors regular sets (Lemma \ref{l:corona} below) and a generalization of David and Semmes {\it bilateral approximation by planes} estimate (Theorem \ref{t:BAUP} below). The former result allows us to effectively pretend that our setting is Ahlfors regular (or at least partition the surface cubes into trees where we can pretend).

The proof of Theorem \ref{t:Afin} is given in Section \ref{s:III}. To prove this, we use Theorem \ref{t:SU-version} and perform a  martingale-type decompositions similar to those that appear \cite{FKP91} when Fefferman, Kenig, and Pipher study dyadic $A_{\infty}$-weights. We prove Theorem \ref{t:hruscev} in the following section and it has a similar proof, but care is needed to account for the lack of Ahlfors regularity. We then move on to Theorem \ref{t:example} in Section \ref{s:example}  by adapting the techniques of Batakis in \cite{Bat96}.

Finally, we prove Theorem \ref{t:green} in Section \ref{s:IV}. The $\lec$ is mostly the same as the proof of the second part of Theorem \ref{t:CDHM}, although things are simpler since we assume uniformity). The $\gec$ estimate is shown in Section \ref{s:IV-Part-II}, and this requires more work. We actually prove a result that holds for more general functions than Green functions, see Theorem \ref{t:affinedeviation} below.\\

\noindent {\bf Acknowledgements.} We would like to thank Mihalis Mourgoglou, Raanan Schul, Xavier Tolsa, and Michele Villa for their useful comments and advice, and John Garnett for his encouragement and explaining to him Schwarzian derivatives.

\section{Preliminaries}

\subsection{Notation}
We will write $a\lesssim b $ if there is a constant $C>0$ so that $a\leq C b$ and $a\lesssim_{t} b$ if the constant depends on the parameter $t$. As usual we write $a\sim b$ and $a\sim_{t} b$ to mean $a\lesssim b \lesssim a$ and 
$a\lesssim_{t} b \lesssim_{t} a$ respectively. We will assume all implied constants depend on $d$ and hence write $\sim$ instead of $\sim_{d}$.

Whenever $A,B\subset\mathbb{R}^{d+1}$ we define
\[
\mbox{dist}(A,B)=\inf\{|x-y|;\, x\in A, \, y\in B\}, \, \mbox{and}\, \, \mbox{dist}(x,A)=\mbox{dist}(\{x\}, A). 
\]
Let $\diam A$ denote the diameter of $A$ defined as
\[
\diam A=\sup\{|x-y|;\, x,y\in A\}.
\]

For a domain $\Omega$ and $x\in \Omega$, we will write
\[
\delta_{\Omega}(x) =\dist(x,\d\Omega).
\]

We let $B(x,r)$ denote the open ball centered at $x$ of radius $r$. For a ball $B$, we will denote its radius by $r_{B}$. 

Given  two closed sets $E$ and $F$, and $B$ a set we denote
\[
d_{B}(E,F)=\frac{2}{\diam B}\max\left\{\sup_{y\in E\cap B}\dist(y,F), \sup_{y\in F\cap  B}\dist(y,E)\right\}\]

\subsection{Christ-David Cubes}

\begin{theorem}
Let $X$ be a doubling metric space. Let $X_{k}$ be a nested sequence of maximal $\rho^{k}$-nets for $X$ where $\rho<1/1000$ and let $c_{0}=1/500$. For each $n\in\bZ$ there is a collection $\cD_{k}$ of ``cubes,'' which are Borel subsets of $X$ such that the following hold.
\begin{enumerate}
\item For every integer $k$, $X=\bigcup_{Q\in \cD_{k}}Q$.
\item If $Q,Q'\in \cD=\bigcup \cD_{k}$ and $Q\cap Q'\neq\emptyset$, then $Q\subseteq Q'$ or $Q'\subseteq Q$.
\item For $Q\in \cD$, let $k(Q)$ be the unique integer so that $Q\in \cD_{k}$ and set $\ell(Q)=5\rho^{k(Q)}$. Then there is $\zeta_{Q}\in X_{k}$ so that
\begin{equation}\label{e:containment}
B_{X}(\zeta_{Q},c_{0}\ell(Q) )\subseteq Q\subseteq B_{X}(\zeta_{Q},\ell(Q))
\end{equation}
and
\[ X_{k}=\{\zeta_{Q}: Q\in \cD_{k}\}.\]
\end{enumerate}
\label{t:Christ}
\end{theorem}
\def\Child{{\rm Child}}
If $Q\in \cD_{k}$, we let 
\[
\Child(Q)=\{R\in \cD_{k+1}:R\subseteq Q\}.
\]

We recall some facts about stopping-time regions, which can be found in \cite{of-and-on}.

\begin{definition}
A {\it tree} or {\it stopping-time region} is a subcollection $S\subseteq \cT$ with a maximal cube $Q(S)\in S$ so that if $Q\in S$ and $Q\subseteq T\subseteq Q(S)$, then $T\in S$, and if whenever a child of $Q\in S$ is not in $S$, then no children of $Q$ are in $S$.
\end{definition}

We will usually construct stopping-time regions as follows. 

\begin{lemma}
\label{l:ds-cover}
Let $R$ be a cube, $\cC$ a (possibly empty) collection of subcubes properly contained in $R$, let $\Stop(R)$ be the maximal cubes in $R$ that contain a child in $\cC$, and let $\cT$ be those cubes in $\cT$ that are not properly contained in a cube from $\Stop(R)$. Then $\cT$ is a stopping-time region. 
\end{lemma}

\begin{proof}
The first two properties of being a stopping-time are immediate, so we just verify the last one. Let $Q\in \cT$, then there is $S\in \Stop(R)$ with $S\subset Q$. Suppose $Q'$ was a child of $Q$ that was also in $\cT$. Then $Q'$ is not properly contained in a cube from $\Stop(R)$, but then neither can any of its siblings, so all of its siblings are in $\cT$. 
\end{proof}

The last property in the definition may seem odd, but it is to guarantee the following property about minimal cubes. Recall that $Q\in \cT$ is a {\it minimal cube} for $\cT$ if it does not properly contain any cubes from $\cT$. Let $z(\cT)$ be those points in $T$ not contained in any minimal cube. In particular, for $\cT$ defined as in the previous lemma, the minimal cubes are exactly $\Stop(R)$.

\begin{lemma}
\label{l:minimal-partition}
Let $\cT$ be a stopping-time region with top cube $T$. Then for all $x\in T$, there is a smallest cube in $\cT$ containing $x$ or there are infinitely many cubes from $T$ containing $x$, and $z(\cT)$ along with the set of all such minimal cubes partition $T$.  
\end{lemma}

See \cite[Page 56]{of-and-on}.

\subsection{Quantitative Rectifiability}

In this section, we recall a few preliminaries about quantitative rectifaibility. We first recall the Analyst's traveling salesman theorem proven in \cite{AS18} (however see \cite{AV19} for this statement):

\begin{theorem}
Let $1\leq d<n$ and $E\subseteq \bR^{n}$ be lower $(c,d)$-lower content regular and let $\cD$ denote the Christ-David cubes for $E$. For a ball $B$ centered on $E$ and a $d$-dimensional plane $P$, let 

\[
b\beta_{E}^{d}(B,P) = d_{B}(E,P), \;\; b\beta_{E}^{d}(B) = \inf_{P}b\beta(B,P)
\]
where the infimum is over all $d$-dimensional planes $P$. Let ${\rm BLWG}(\ve,C_0)=\{Q: b\beta_{E}^{d}(C_0B_Q)\geq \ve)$. For $R\in \cD$,
    \[
  {\rm BLWG}(R)=  {\rm BLWG}(R,\ve,C_0)=\sum_{ Q\in {\rm BLWG}(\ve,C_0) \atop Q\subseteq R } \ell(Q)^{d}. 
    \]
    and for $M\geq 3$,
        \[
    \beta_{E,M}(R) :=\ell(R)^{d}+\sum_{Q\subseteq R} \beta_{E}^{d,p}(MB_Q)^{2}\ell(Q)^{d}.
    \]

    Then for $R\in \cD$,
\begin{equation}
\label{e:tst}
\cH^{d}(R)+  {\rm BLWG}(R,\ve,C_0)
 \sim_{n,c,M,C_0 , \ve}    \beta_{E}(R).
 \end{equation}
\label{t:TST}
\end{theorem}

Note that as these values are comparable for all $M$, we will denote
\[
    \beta_{E}(R) = \beta_{E,3}(R)\sim_{M} \beta_{E,M}(R) \;\; \mbox{ for all }M\geq 3.
    \]

This is a version of the original traveling salesman theorem of Jones \cite{Jon90}, which instead had an $L^{\infty}$-$\beta$-number, and their square sum was comparable to the shortest curve containing $E$. This was originally shown in the plane, but was subsequently generalized to Euclidean space \cite{Oki92} and Hilbert space \cite{Sch07-TST}. 

\begin{remark}
To avoid some confusion with notation, the reader will find it helpful to remember that, given a set $S$, $\beta(S)$ will denote the usual $\beta$-number if $S$ is a ball and will denote a sum of cubes in $S$ as $\beta_{E,M}(R)$ above if $S$ is a cube.
\end{remark}

The following is the main lemma from \cite{AV19}.

\begin{lemma}
\label{l:corona}
Let $k_0>0$, $\vartheta>0$, $d>0$ and $E$ be a closed set that is lower $d$-content $c$-regular with $\diam E\sim 1$. Let $\cD_k$ denote the Christ-David cubes on $E$ of scale $k$ and $\cD=\bigcup_{k\in\bZ} \cD_{k}$. Let $Q_{0}\in \cD_{0}$ and $\cD(k_0)=\bigcup_{k=0}^{k_0}\{Q\in \cD_{k}:Q\subseteq Q_0\}$. Then we may partition $\cD(k_{0})$ into stopping-time regions $Tree(R)$ for $R$ from some collection $Top(k_{0})\subseteq \cD(k_{0})$ with the following properties: 
\begin{enumerate}
\item We have 
\begin{equation}
\label{e:ADR-packing}
\sum_{R \in Top(k_{0})} \ell(R)^{d} \lec_{c,d} \cH^{d}(Q_0).
\end{equation}

\item Given $R\in \Top(k_{0})$ and a stopping-time region $\cT\subseteq \Tree(R)$ with maximal cube $T$, let  $\cF$ denote the minimal cubes of $\cT$ and 
\[
d_{\cF}(x)=\inf_{Q\in \cF} (\ell(Q)+\dist(x,Q)).
\]
For $C_{0}>4$ and $0<\vartheta\ll C_{0}^{-1}$, there is a collection  $\cC$ of disjoint dyadic cubes covering $C_{0}B_{T}\cap E$ so that 
if 
\[
E(\cT)=\bigcup_{I\in \cC} \d_{d} I,\]
where $\d_{d}I$ denotes the $d$-dimensional skeleton of $I$, then the following hold:
\begin{enumerate}[(a)]
\item $E(\cT)$ is Ahlfors regular, that is,
\begin{equation}
\label{e:ETregular}
\cH^{d}(B(x,r)\cap E(\cT))\sim_{C_{0},\vartheta,d,c} r^{d} \;\; \mbox{ for all }x\in E(\cT), \;\; 0<r<\diam E(\cT).
\end{equation}
\item We have the containment
\begin{equation}
\label{e:contains}
C_{0}B_{T}\cap E \subseteq \bigcup_{I\in \cC} I\subseteq 2C_{0}B_{T}.
\end{equation}

%
%Let $\cC_{\cF}$ denote the maximal dyadic cubes $I$ for which $\ell(I)<\tau \inf_{x\in I}d_{\cF}(x)$, let $C_{0}>4$, and set 
%\[E_{\cT}=\bigcup_{I\in \cC_{\cF} \atop I\cap C_{0}B_{T}\neq\emptyset} \d_{d} I\] 
%where $\d_{d}I$ denotes the $d$-dimensional skeleton of $I$. Then $E(\cT)$ is Ahlfors $d$-regular,
\item $E$ is close to $E(\cT)$ in $C_{0}B_{T}$ in the sense that
\begin{equation}
\label{e:adr-corona}
\dist(x,E(\cT))\lec  \vartheta d_{\cF}(x) \;\; \mbox{ for all }x\in E\cap C_{0}B_{T}.
\end{equation}
\item The cubes in $\cC$ satisfy
\begin{equation}
\label{e:whitney-like}
\ell(I)\sim \vartheta \inf_{x\in I} d_{\cF}(x) \mbox{ for all }I\in \cC.
\end{equation}
\end{enumerate}
\end{enumerate}
\end{lemma}

Finally, we recall the David-Toro parametrization theorem \cite{DT12}. We state only a consequence of their result, since their full result is more general. There they used planes associated to balls in their statements, but we would like to use planes associated to cubes. Converting between the two has been done in several papers \cite{AS18,ATT18,AT15}, but here we state a converted version for cubes, hopefully so that it doesn't have to be re-converted in the future. We prove this reformulation in the appendix:

\begin{lemma}
\label{l:DT-dyadic}
Let $E\subseteq \R^{n}$ be some set with Christ cubes $\cD$ in some set $E$. Declare $R\sim Q$ if $C_{2}^{-1}\ell(Q)\leq \ell(R)\leq C_{2}\ell(Q)$ and $\dist(Q,R)\leq C_{2}\min \{\ell(Q),\ell(R)\}$. For $\ve^{-1}\gg C_{1}, C_{2}$, the following holds. Let $S$ be a stopping-time region with top cube $Q(S)$ so that $\zeta_{Q(S)}=0$ and for all $Q\in S$, there is a $d$-plane $P_{Q}$ such that 
\begin{equation}
\label{e:distQtoP_Q}
\dist(\zeta_{Q},P_{Q})< \ve \ell(Q) \mbox{ for all } Q\in S.
\end{equation}
Moreover, if 
\[
\ve(Q) = \max_{R\sim Q} \ell(Q)^{-1}\ps{\sup_{x\in P_{Q}\cap C_{1} B_{Q}}\dist(x,P_{R}) + \sup_{x\in P_{R}\cap C_{1} B_{Q}}\dist(x,P_{Q})}
\]
and 
\begin{equation}
\label{e:epsilonsum}
\sum_{Q\subseteq R\subseteq Q(S)} \ve(R)^{2} <\ve^2. 
\end{equation}
Then for $\ve>0$ small enough, there is $g:\R^{n}\rightarrow \R^{n}$ that is $C$-bi-Lipschitz on $\R^{n}$ and $(1+C\ve^{2})$-bi-Lipschitz when restricted to $P_{Q(S)}$ and 
\begin{equation}
\label{e:gz-z}
|g(z)-z|\lec \ve \ell(R) \;\; \mbox{ for all }z\in \R^{n}.
\end{equation}
 The surface $g(P_{R})=:\Sigma_{S}$ is $C\ve$-Reifenberg flat so that 
\begin{equation}
\label{e:dtjones}
\dist(Q,\Sigma_{Q})\lec \ve \ell(Q) \;\; \mbox{ for all }Q\in S.
\end{equation}
If 
\begin{equation}
\label{e:C1beta}
\sup_{x\in 2C_{1}B_{Q}\cap E} \dist(x,P_{Q})<\ve \ell(Q) \mbox{ for all }Q\in S,
\end{equation}
then
\begin{equation}
\label{e:close-to-P_Q}
\sup_{z\in C_{1}B_{Q}\cap E} \dist(z,\Sigma_{R})\lec \ve \ell(Q)\mbox{ for all }Q\in S
\end{equation}
If 
\begin{equation}
\label{e:C2beta}
\sup_{x\in 2C_{1}B_{Q}\cap P_Q}\dist(x,E)<\ve {\ell(Q)} \mbox{ for all }Q\in S,
\end{equation}
then
\begin{equation}
\label{e:close-to-E}
\sup_{z\in C_{1}B_{Q}\cap \Sigma_{R}} \dist(z,E)\lec \ve \ell(Q)\mbox{ for all }Q\in S
\end{equation}
\end{lemma}

\subsection{Harmonic Measure}

For background on harmonic measure and Green's function, we refer the reader to \cite{AG}. 
\begin{definition}
%A domain $\Omega \subset \R^{d+1}$ is called {\it regular} if every point of $\d_\infty\Omega$ is regular (where $\d_{\infty}\Omega$ means we include the point at infinity if $\Omega$ is unbounded), that is, if $f$ is a continuous function on $\d_{\infty} \Omega$, then the harmonic . 
For $K\subset \d\Omega$, we say that $\Omega$ has the {\it capacity density condition (CDC) in $K$} if $ \textup{cap}({B}(x,r) \cap \Omega^c, B(x,2r)) \gtrsim r^{d-1}$, for every $x \in K$ and $r<\diam K$, and that $\Omega$ has the {\it capacity density condition} if it has the CDC in $K=\d\Omega$.  Here, cap$(\cdot, \cdot)$ stands for the variational $2$--capacity of the condenser $(\cdot, \cdot)$ (see \cite[p. 27]{HKM} for the definition).
\end{definition}

\begin{lemma}[{\cite[Lemma 11.21]{HKM}}]\label{l:bourgain}
Let $\Omega\subset \bR^{d+1}$ be any domain satisfying the CDC condition,  $B$ a ball centered on $\d\Omega$ so that $\Omega\backslash 2B\neq\emptyset$. Then 
\begin{equation}\label{e:bourgain}
\omega_{\Omega}^{x}(2B)\geq c >0 \;\; \mbox{ for all }x\in \Omega\cap B.
\end{equation}
where $c$ depends on $d$ and the constant in the CDC. 
\end{lemma}

Using the previous lemma and iterating, it is possible to obtain the following lemma.

\begin{lemma}\label{l:holder} \cite[Lemma 2.3]{AM18}
Let $\Omega\subseteq \bR^{d+1}$ be a domain with the CDC, $\xi\in \d\Omega$ and $0<r<\diam \d\Omega/2$. Suppose $u$ is a non-negative function that is harmonic in $B(\xi,r)\cap \Omega$ and vanishes continuously on $\d\Omega\cap B(\xi,r)$. Then
\begin{equation}\label{e:holder}
u(x) \lec  \ps{\sup_{y\in B(\xi,r)\cap \Omega} u} \ps{\frac{|x-\xi|}{r}}^{\alpha}
\end{equation}
where $\alpha>0$ depends on the CDC constant and $d$. 
\end{lemma}

%
%\begin{lemma}[\cite{AHMMMTV16}, Lemma 3.1] \label{greeneverywhere}
%Let $\Omega$ be a Greenian domain and let $y\in\Omega$. For $m$-almost all $x\in\Omega^c$ we have
%\begin{equation}\label{greeneverywhereeq}
%\mathcal{E}(x-y) - \int_{\partial\Omega} \mathcal{E}(x-z)\,d\omega^y(z)=0.
%\end{equation}
%\end{lemma}

%We won't define Greenian domain here, but just mention that CDC domains are Greenian. 

There are two key facts we will use about Green's function.

\begin{lemma}
\cite[Lemma 1]{Aik08}
For $x\in \Omega\subseteq \bR^{d+1}$ and $\phi\in C_{c}^{\infty}(\bR^{d+1})$, 
\begin{equation}
\label{e:ibp}
\int \phi d\omega_{\Omega}^{x} = \int_{\Omega} \triangle \phi(y) G_{\Omega}(x,y)dy + \phi(x).
\end{equation}
\end{lemma}

\begin{lemma}
Let $\Omega\subset\bR^{d+1}$ be a CDC domain. Let $B$ be a ball centered on $\d\Omega$ and $0<r_{B}<\diam \d\Omega $. Then,
 \begin{equation}\label{e:w>G}
 \omega^{x}(4B)\gtrsim r_{B}^{d-1}\, G_{\Omega}(x,y)\quad\mbox{
 for all $x\in \Omega\backslash  2B$ and $y\in B\cap\Omega$,}
 \end{equation}
 The opposite inequality holds if $\Omega$ is also uniform.
\end{lemma}

This follows quickly from the maximum principle, Lemma \ref{l:bourgain}, and the fact that, for $x\in \d 2B\cap \Omega$ and $y\in B$, $r_{B}^{d-1}G_{\Omega}(x,y)\lec 1$. For proofs, see \cite[Lemma 3.5]{AH08} or \cite[Lemma 3.3]{AHMMMTV16}.

The following lemma was first shown by Aikawa and Hirata for John domains with the CDC \cite{AH08}; with a minor adjustment, the John condition can be removed \cite[Theorem I]{Azz17}

\begin{lemma}\label{l:doubling}
Let $\Omega\subseteq \bR^{d+1}$ be a CDC domain. Then the following are equivalent:
\begin{enumerate}
\item $\omega_{\Omega}$ is {\it doubling}, meaning there is a constant $A\geq 2$ and a function $C:(0,\infty)\rightarrow (1,\infty)$ so that, for any ball $B$ centered on $\d\Omega$  and $\alpha>0$, 
\begin{equation}\label{doubling}
\omega_{\Omega}^{x}(2B)\leq C(\alpha) \omega_{\Omega}^{x}(B) \mbox{ for all $x$ such that $\dist(x, AB\cap \d\Omega)\geq \alpha |x-x_{B}|$}.
\end{equation}
\item $\Omega$ is semi-uniform.
\end{enumerate}
\end{lemma}

\begin{lemma} \cite[Theorem II]{Azz17}
\label{l:w/w}
Let $\Omega$ be a semi-uniform CDC domain, $B$ a ball centered on $\d\Omega$, and $E\subseteq B\cap \d\Omega$. Then there is $M>0$ depending on the CDC and semi-uniformity constants and corkscrew points $x_{1}$ and $x_{2}$ in $B\cap \Omega$ so that 
\[
\omega_{\Omega}^{x_{1}}(E)
\lec \frac{\omega_{\Omega}^{x}(E)}{\omega_{\Omega}^{x}(B)}
\lec \omega_{\Omega}^{x_{2}}(E)
\mbox{ for all }x\in \Omega\backslash MB.
\]
If $\Omega$ is uniform, then we can take $x_1=x_2$ to be any corkscrew point in $B$. 
\end{lemma}
The last line of the lemma is due to Jerison and Kenig \cite{JK82} for NTA domains, and for general uniform domains this follows from the work of Aikaha and Hirata \cite{AH08}. 

%
%
%\begin{definition}
%For a domain $\Omega\subseteq \bR^{d+1}$, we say that points $y_{1},...,y_{n}$ are reference points for a ball $B$ centered on $\d\Omega$ if
%\begin{equation}
%\label{e:pseudo}
%\min_{i=1,...,n}k_{\Omega}(x,y_{i})\lec \log\frac{r_{B}}{\delta_{\Omega}(x)}+1 \;\; \mbox{ for all }x\in B. 
%\end{equation}
%where $k_{\Omega}(x,y)$ denotes the quasihyperbolic distance between $x$ and $y$. As observed in \cite[p. 434]{AH08}, if $N_{\Omega}(x,y)$ denotes the length of the shortest Harnack chain between $x$ and $y$, then
%\[
%N_{\Omega}(x,y)\sim k_{\Omega}(x,y)+1.
%\]
% 
%\end{definition}
%
%If $\Omega$ is SU, one can always find $N$ reference points for each ball centered on the boundary with $N$ depending on the SU constants, see \cite[Remark 4.2]{Azz19}. 
%
%
%
%\begin{lemma} \cite[Lemma 3.6]{AH08}. 
%\label{l:AHlemma}
%Let $\Omega\subseteq \R^{d+1}$ be a CDC domain, $B$ a ball centered on $\d\Omega$ with $r_{B}<\diam \d\Omega$, and $y_{1},...,y_{n}\in \Omega$ be reference points for $2B$. Then
%\begin{equation}
%\label{e:AHlemma}
%\omega_{\Omega}^{x}(B)\lec r_{B}^{d-1} \sum_{i=1}^{n}G_{\Omega}(x,y_{i}) \;\; \mbox{ for }x\in \Omega\backslash 2B.
%\end{equation}
%The implied constant depends on the CDC constant and reference point constants. 
%\end{lemma}
%
%The statement above is slightly different from the original, but the proof is exactly the same. 
%
%

\section{Proof of Theorems \ref{t:CDHM} and \ref{t:SU-version}, Part I}
\label{s:I-II-Part-I}

The aim of this section is to prove the following lemma:
\def\BTM{{\rm BTM}}
\begin{lemma}
\label{l:CDHM<beta}
Let $\Omega\subseteq \R^{d+1}$ be an interior corkscrew domain with LCR boundary. Let $\cT$ be a stopping-time region with top cube $Q_0$, and $\BTM$ (for ``bottom") be the (possibly empty) set of children of the minimal cubes for $\cT$. 
%With the assumptions of Theorem \ref{t:CDHM}, 
For $\lambda\geq 1$, and for $A,\tau^{-1}$ sufficiently large, we may find cubes $\Top$ contained in $Q_0$ and a partition of $\cT$ into trees $\{\Tree(R):R\in \Top\}$ so that for each $R\in \Top$, there is a corkscrew ball $B(x_R,c\ell(R))\subseteq B_{R}\cap \Omega$ so that for all $Q\in \Tree(R)$,
\[
\tau \Theta_{\omega^{x_{R}}}^{d}(\lambda B_R) \leq \Theta_{\omega^{x_{R}}}^{d}(\lambda B_{Q})\leq A\Theta_{\omega^{x_{R}}}^{d}(\lambda B_R).
\]
and for $M$ large enough,
\begin{equation}
\label{e:btm+top<b}
\sum_{R\in \BTM}\ell(R)^{d}+\sum_{R\in \Top}\ell(R)^{d}\lec \beta_{\d\Omega}(\cT):=\ell(Q_0)^{d}+ \sum_{Q\in \cT}\beta_{\d\Omega}^{d,2}(MB_{Q})^{2}\ell(Q)^{d}.
\end{equation}
%\[
%\CDHM(R,\lambda, A,\tau )\lec_{A,\tau,\lambda}\beta_{\d\Omega}(R) .
%\]
\end{lemma}

We will write $\omega=\omega_{\Omega}$ and $\beta=\beta_{\d\Omega}^{d,2}$ for short. 

%We define a refined stopping-time as follows. First let $\widetilde{\Top}(k_{0})$ be the cubes obtained from Lemma \ref{l:corona} with $C_1$ and $\vartheta$ to be decided later, and for $R\in \widetilde{Top}$, let $\widetilde{\Tree}(R)$ denote the trees obtained from that lemma. Let $\widetilde{\Bad}$ be those cubes that have a child in $\tilde{\Top}$.

Let $k_0\in \bN$ and $\cT(k_{0})$ be those cubes in $\cT$ with sidelength at least $5\rho^{k_{0}}$. Let $M>1$, $\ve>0$ and
\[\Bad=%\widetilde{\Bad}\cup 
\{Q\subseteq Q_{0}: b\beta_{\d\Omega} (MB_Q)\geq \ve\}
\cap \cT(k_{0}).
\]

Observe that by Theorem \ref{t:TST} and Theorem \ref{l:corona},
\begin{equation}
\label{e:sumbad}
\sum_{Q\in \Bad}\ell(Q)^{d}
\lec \beta_{\d\Omega}(Q_{0}).
\end{equation}

 For $R\in \Bad$, we define $\Stop(R)=\{\emptyset\}$ and $\Next(R)$ to be the children of $R$ that are in $\cT(k_{0})$ (so these could be empty). 

For $R\in \cT(k_0) \backslash \Bad$, there is $P_{R}$ so that $b\beta_{\d\Omega}(MB_{R},P_{R})<\ve$. We can assume without loss of generality that $P_{R}$ passes through $\zeta_{R}$, since we still have $b\beta_{\d\Omega}(MB_{R},P_{R})\leq  2\ve$ for $\ve>0$. 

Let $\nu_{R}$ be the normal vector to $P_{R}$, and let $x_{R}^{\pm} = \zeta_{R}\pm \frac{\ell(R)}{2}\nu_{R}$. Since $\Omega$ is a $c$-corkscrew domain, for $\ve>0$ small enough, $\Omega$ must contain either $B(x_{R}
^{\pm}, \ell(R)/8)$. This is because, since $b\beta_{\d\Omega}(MB_R,P_R)<\ve$,  every $z\in  \d\Omega\cap MB_R$ satisfies $\dist(z,P_R)\leq M\ve \ell(R)$, and so for $\ve>0$ small enough, $\dist(B(x_{R}
^{\pm}, \ell(R)/8),\d\Omega)>0$, and if we let $B_{R}^{\pm}$ be the two components of $\{x\in B_{R}: \dist(x,P_{R})\geq \ve \ell(R)\}$ that contain $x_{R}^{\pm}$ respectively, then one of these must be contained in $\Omega$ since  otherwise the largest corkscrew ball in $B_{R}$ must have radius at most $\ve \ell(R)$, which is a contradiction if $\ve<c$ (recall $\Omega$ is a $c$-corkscrew domain). By changing the $\pm$ if necessary, we will assume $x_{R}=x_{R}^{+}$ is always in $\Omega$.

Let $R\in \cT(k_0)$. We let $\Stop(R)$ be the maximal cubes $Q\in  \cT(k_0) \backslash \Bad$ which contain a child $Q'$ for which one the following occurs:
\begin{enumerate}
\item[\BTM:] $Q'\in \BTM$. We call these cubes $\BTM(R)$.
\item[\Bad :] $Q'\in \Bad$. We call these cubes $\Bad(R)$.
\item[HD:] $\Theta_{\omega^{x_{R}^{\pm}}}^{d}(\lambda B_{Q'})>A\Theta_{\omega^{x_{R}^{\pm}}}^{d}(\lambda B_{R})$, call these cubes $\HD^{\pm}(R)$ and let $\HD(R)=\HD^{+}(R)\cup \HD^{-}(R)$. If $x_{R}^{-}\not\in \Omega$, we simply let $\HD^{-}(R)=\emptyset$.
\item[LD:] $\Theta_{\omega^{x_{R}}}^{d}(Q')<\tau \Theta_{\omega^{x_{R}}}^{d}(R)$, call these cubes $\LD^{\pm}(R)$ and let  $\LD(R)=\LD^{+}(R)\cup \LD^{-}(R)$. If $x_{R}^{-}\not\in \Omega$, we simply let $\LD^{-}(R)=\emptyset$.
\item[$B\beta$:] $Q'\not\in \Bad$, but for some fixed $M>0$,
\begin{equation}
\label{e:beta-Jones>ve^2}
\sum_{Q'\subseteq T\subseteq R} \beta(MB_{Q})^2 \geq 2 \ve^2 
,\end{equation}
call these cubes $B\beta(R)$. Note that for such  a $Q'$, since $Q'\not\in \Bad$, we have $\beta(MB_{Q'})\leq b\beta(MB_{Q'})<\ve$, and so 
\begin{equation}
\label{e:beta-Jones>ve^2*}
2\ve ^2
>\sum_{Q\subseteq T\subseteq R} \beta(MB_{Q})^2 
=\sum_{Q'\subseteq T\subseteq R} \beta(MB_{Q})^2  - \beta(MB_{Q'})^2
> 2 \ve^2 -\ve^2 = \ve^2 .
,\end{equation}
%\item[Floor:] $Q\in \cD_{k_{0}}\backslash ( \LD(R)\cup \HD(R)\cup B\beta(R))$; call these cubes $\Floor(R)$. 
\end{enumerate}

We let $\Tree(R)$ denote the cubes in $\cT(k_0)$ contained in $R$ that are not properly contained in any cube from $\Stop(R)$ (so $\Stop(R)\subseteq \Tree(R)$) and let $\Next(R)$ denote the children of the cubes in $\Stop(R)$ that are in $\cT(k_0) $. so $\Next(R)$ could be empty, for example, if $R\in \cD_{k_0}$, or if $R$ is a minimal cube for $\cT$.

\def\wt{\widetilde}
%\begin{remark}
%\label{r:E_R}
%Observe that for $R\in \Top$, $\Tree(R)\subseteq \wt{\Tree}(\tilde{R})$ for some $\tilde{R}\in \wt{\Top}$, and so by Lemma \ref{l:corona}, there is an Ahlfors regular set $E_{R}=E(\Tree(R))$ satisfying the conclusions of the lemma. 
%\end{remark}

%Notice that in this way, if $Q\in\HD(R)$, then its parent $Q^{1}$ was not a stopped cube, thus
%\[
%A\Theta_{\omega^{x_{R}}}^{d}(R)<\Theta_{\omega^{x_{R}}}^{d}(Q)
%\lec \Theta_{\omega^{x_{R}}}^{d}(Q^1)\leq A\Theta_{\omega^{x_{R}}}^{d}(R).
%\]

For $R\in \Bad$, we let $\Stop(R)=\{R\}$ and $\Next(R)$ denote the children of $R$ in $\cT(k_0)$ and $\Tree(R)=\{R\}$. 

Let $Q_{0}\in \Top_{0}$, and inductively, if $R\in \Top_{k}$, set 
\[
\Top_{k+1}=\bigcup_{R\in \Top_{k}}\Next(R).
\]

Note that $\Top_{k}=\emptyset$ for all large $k$. Let $\Top = \bigcup \Top_{k}$.

\begin{remark}
\label{r:k_0}
The tops $\Top$ and trees $\Tree(R)$ should really be written $\Top^{k_{0}}$ and $\Tree^{k_0}(R)$ respectively since they depend on $k_0$, but we suppress the $k_0$ for ease of notation. Notice, however, that the trees are increasing in $k_0$, so the final tops we desire will be 
\[
\Top=\bigcap_{n>0}\bigcup_{k_{0}\geq n} \Top^{k_{0}}
\]
and for $R\in \Top$,
\[
\Tree(R) = \bigcup_{k_{0}\geq 1} \Tree^{k_{0}}(R).
\]
The purpose of cutting off our cubes at the scale $k_0$ is for simplicity, so that our trees are always finite.
\end{remark}

Without loss of generality, assume $P_{R}=\R^{d}$ and $\zeta_{R}=0$. Let $C_{1}= M/2 $ and $C_2$ be such that $M \gg C_{2}>1$. We can then apply Lemma \ref{l:DT-dyadic} with these constants to $S=\Tree(R)$. Let $g_{R}=g_{\Tree(R)}$ be the bi-Lipschitz map and $g_R(\R^{d})=\Sigma_{R}$ be the surface from the lemma. Since $\zeta_{R}=0$, by \eqref{e:gz-z}, $|g_R(0)|\lec \ve \ell(R)$.  Let 
 \[
 d_{R}(x) = \inf\{\ell(Q)+\dist(x,Q): Q\in \Tree(R)\}.
 \]

Let $\Omega_{R,\pm}$ be the component of $\Sigma_{R}$ containing the corkscrew $x_{R}^{\pm}$, we can assume $\Omega_{R,\pm}=g_{R}(\R^{d+1}_{\pm})$.  Let $\alpha>0$ be small, $e_{d+1}$ be the $(d+1)$st standard basis vector, and 
 \[
 U_{R}^{\pm} = \{x=x'\pm x_{d+1}e_{d+1}: x'\in \R^{d},\;  x_{d+1}>\alpha d_{R}(g(x')))\}\cap B(0,10\ell(R)).
 \]
If $x_{R}^{-}\not\in \Omega$, then we just set $U_{R}^{-}=\emptyset$. Since $d_{R}$ is Lipschitz, $U_{R}^{\pm }$ are disjoint Lipschitz domains, and since $g_{R}$ is bi-Lipschitz, the domain 
 \[
 \Omega_{R}^{\pm} = g_{R}(U_{R}^{\pm})
 \]
 is a CAD. Also note that, since $\Tree(R)\subseteq \cD(k_{0})$, we always have $d_{R}>0$.

\begin{lemma}
If $x_{R}^{\pm}\in \Omega$, then $\Omega_{R}^{\pm}\subseteq \Omega$. 
\end{lemma}

\begin{proof}
We will just show this for $\Omega_{R}=\Omega_{R}^{+}$ and $x_{R}=x_{R}^{+}$. We will show that if $y\in \d\Omega_{R}$, then $\dist(y,\d\Omega)>0$. 

Let $G_{R}= \d U_{R}^{+}\backslash \d B(0,10\ell(R))$. 

{\bf Case 1:} Suppose first that $y\in g_{R}(G_{R})$. Let $x=g_{R}^{-1}(y) \in \d U_{R}^{+}$. Then $x=x'+x_{d+1}e_{d+1} \in \R^{d}\oplus \R$ where $x'\in B(0,10\ell(R))\cap \R^{d}$. 

Let $y'=g_{R}(x')\in \Sigma_{R}$. Let $Q\in \Tree(R)$ be so that 
\[
\ell(Q)+\dist(y',Q)=d_{R}(y').
\]
Let $\hat{Q}\in \Tree(R)$ be the maximal ancestor of $Q$ so that $\ell(\hat{Q})\leq d_{R}(y')$. We claim that
\begin{equation}
\label{e:hatqsimdry'}
\ell(\hat{Q})\sim d_{R}(y').
\end{equation}

Indeed, if $\ell(\hat{Q})<\ell(R)$, then this is clear since the parent of $\hat{Q}$ (which has comparable size) will have size at least $d_{R}(y')$. If $\hat{Q}=R$, then because $g_{R}$ is bi-Lipschitz,
\[
\dist(y',R) 
\leq |y'-\zeta_{R}|
=|y'|
\leq |y'-g_{R}(0)|+|g_{R}(0)
\lec |x'-0|+\ve \ell(R) \lec \ell(R),\]
%since $U_{Q}^+\subseteq B(0,10\ell(R))$ and 
%\[
%\dist(g_{R}(0),R)\leq |g_{R}(0)-\zeta_{R}| = |g_{R}(0)| \lec  \ve \ell(R),
%\]
and so we have
\[
\ell(\hat{Q})\leq d_{R}(y')
\leq \ell(R)+\dist(y',R)\lec \ell(R) = \ell(\hat{Q}).
\]
This proves the claim. In particular, for $C_1$ large enough, $y'\in \frac{C_{1}}{2} B_{\hat{Q}}$. Since $g_{R}$ is bi-Lipschitz, we have 
\begin{equation}
\label{e:y-y'}
|y-y'|\sim |x-x'|
= x_{d+1} =\alpha d_{R}(g_{R}(x')) =\alpha  d_{R}(y')\sim \alpha \ell(\hat{Q}).
\end{equation}
Hence, for $\ve<\alpha$, and since $y'\in C_{1} B_{\hat{Q}}\cap \Sigma_{R}$, if $z\in \d\Omega$ is closest to $y$,
\begin{equation}
\label{e:distyOm}
|z-y|\leq \dist(y,\d\Omega) \leq |y-y'|+\dist(y',\d\Omega) \stackrel{ \eqref{e:close-to-E}\atop \eqref{e:y-y'}}{\lec} \alpha \ell(\hat{Q})
\end{equation}
In particular, for $C_1$ large enough and $\alpha$ small enough, $z\in C_{1} B_{\hat{Q}}$, and by \eqref{e:close-to-P_Q}, there is $y_0\in \Sigma_{R}$ with $|z-y_{0}|\lec \ve \ell(\hat{Q})$. Let $x_0=g_{R}^{-1}(y_{0})\in \R^{d}$. Since $g_{R}$ is $C$-bi-Lipschitz on $\R^{n}$, $x_0\in \R^{d}$, and $x\in U_{R}^{+}$, 
\[
|y-y_0|
\sim |x-x_0|
\geq \dist(x,\R^{d})
\geq  \alpha d_{R}(g_{R}(x'))
=\alpha d_{R}(y')
\sim \alpha \ell(\hat{Q})\]
and so  for $\ve\ll \alpha$,
\[
\dist(y,\d\Omega)
=|y-z|
\geq |y-y_0|-|y_0-z|
\gec \alpha \ell(\hat{Q})- \ve \ell(\hat{Q})
\gec \alpha d_{R}(y').
\]

This and \eqref{e:distyOm} imply
\begin{equation}
\label{e:GR}
\dist(y,\d\Omega) \sim \alpha d_{R}(y') \mbox{ for }y\in g_{R}(G_{R}).
\end{equation}
In particular, $\dist(y,\d\Omega)>0$ in this case.

{\bf Case 2:} Now suppose $y\in g_{R}(\d B(0,10\ell(R))\cap U_{R}^{+})$. Let $z\in \d\Omega$ be closest to $y$. Since $\dist(g_{R}(0),R)\lec \ve \ell(R)$, we have that 
\[
\dist(y,R)
\lec | y-g_{R}(0)|+\ve \ell(R)
\sim |x-0|+\ve \ell(R)\lec\ell(R),\]
and so $z\in CB_{R}$ for some $C>0$. By \eqref{e:close-to-P_Q}, for $C_{1}\gg C$, $\dist(z,\Sigma_{R})\lec \ve \ell(R)$. Let $z'\in \Sigma_{R}$ be closest to $z$, so $|z-z'|\lec \ve \ell(R)$. Hence, $g_{R}^{-1}(z')\in \R^{d}$, and so (using the fact that $g_R$ is bi-Lipschitz)

\begin{align*}
\dist(y,\d\Omega)
& =|y-z|
\geq |y-z'|-C\ve \ell(R)
\gec  |x-g_{R}^{-1}(z')|-C\ve \ell(R)\\
& \geq   \dist(x,\R^{d})-C\ve \ell(R)
\end{align*}
Thus, if $\dist(x,\R^{d})\geq \alpha \ell(R)$, then $\ve\ll \alpha$ implies $\dist(y,\d\Omega)\gec \alpha \ell(R)$. Otherwise, if $\dist(x,\R^{d})<\alpha\ell(R)$, then for $\alpha$ small enough, since $x\in \d B(0,10\ell(R))$, this implies $|x'|\geq 3\ell(R)$, and so the above inequality and \eqref{e:gz-z} imply
\begin{align*}
\dist(y,\d\Omega)
& \geq   \dist(x,\R^{d})-C\ve \ell(R)
\geq |x_{d+1}|-C\ve \ell(R)
=  \alpha d_{R}(g_{R}(x'))-C\ve \ell(R)\\
& \geq \alpha \dist\ps{g_{R}(x'),\frac{2}{1+C\ve}B_{R}}-C\ve \ell(R)\\
& \stackrel{\eqref{e:gz-z}}{\geq} \alpha \dist(x',B(0,2\ell(R)))-C\ve \ell(R)
\geq  \alpha\ell(R)-C\ve\ell(R)\gec \alpha\ell(R).
\end{align*}
 In either case, $\dist(y,\d\Omega)>0$. This finishes the proof.

\end{proof}

\begin{lemma}
\label{l:totalsum}
We have 
\begin{equation}
\label{e:totalsum}
\sum_{R\in \Top} \ell(R)^{d} \lec \beta_{\d\Omega}(\cT).
\end{equation}

\end{lemma}

\begin{proof}

We will first get some estimates on the stopped cubes. 

\begin{lemma}
\begin{equation}
\label{e:HD}
\sum_{Q\in \HD(R)} \ell(Q)^{d}
 \lec_{\lambda}  A^{-1} \ell(R)^{d}.
\end{equation}
\end{lemma}

\begin{proof}
Without loss of generality (and to simplify notation), we assume $\HD(R)=\HD^{+}(R)$, the general case is similar. If $Q\in \HD(R)$, then $Q$ has a child $Q'$ so that $\Theta_{\omega^{x_{R}}}(\lambda B_{Q'})>A\Theta_{\omega^{x_{R}}}(\lambda B_{R})$. Thus,
\begin{equation}
\label{e:Q>A}
\Theta_{\omega^{x_{R}}}(\lambda B_{Q})
\gec \Theta_{\omega^{x_{R}}}(\lambda B_{Q'})\geq A\Theta_{\omega^{x_{R}}}(\lambda B_{R}).
\end{equation}

By the Vitali covering lemma, we can find disjoint balls $\lambda B_{Q_{j}}$ so that 
\[
\bigcup_{Q\in \HD(R)} \lambda B_{Q} \subseteq \bigcup_{j} 5\lambda B_{Q_{j}}.
\]
Also note that for all $Q\in \HD(R)\subseteq \Tree(R)$, by \eqref{e:close-to-P_Q}, for $\ve$ small enough (recall we set $\delta=2\ve$)
\[
\dist(\zeta_{Q}, \Sigma_{R})<\frac{c_{0}}{2}\ell(Q)
\]
where $\zeta_{Q}$ is as in Definition \ref{t:Christ}, and so 
\begin{equation}
\label{e:Q<c0}
\ell(Q)^{d}\lec \cH^{d}(c_{0}B_{Q}\cap \Sigma_{R}).
\end{equation}

Thus, since the balls $c_{0}B_{Q}$ are disjoint and $\lambda B_{Q_{j}}\subseteq \lambda B_{R}$ for all $j$,
\begin{align*}
\sum_{Q\in HD(R)} \ell(Q)^{d}
& \lec \cH^{d}|_{\Sigma_{R}} \ps{\bigcup_{Q\in \HD(R)} c_{0}B_{Q}}
 \leq \cH^{d}|_{\Sigma_{R}} \ps{\bigcup_{j} 5\lambda B_{Q_{j}}}\\
& \leq \sum_{j} \cH^{d}|_{\Sigma_{R}} ( 5\lambda B_{Q_{j}})
\sim_{\lambda}  \sum_{j} \ell(Q_{j})^{d}\\
& \stackrel{\eqref{e:Q>A}}{\lec}A^{-1}  \Theta^{d}_{\omega^{x_{S}}}( \lambda B_R)^{-1} \sum_{j} \omega^{x_{R}}( \lambda B_{Q_{j}})\\
& \leq A^{-1}\Theta^{d}_{\omega^{x_{S}}}( \lambda B_R)^{-1} \omega^{x_{R}}\ps{\bigcup \lambda B_{Q_{j}}} \\
& \leq A^{-1}\Theta^{d}_{\omega^{x_{S}}}( \lambda B_R)^{-1} \omega^{x_{R}}( \lambda B_{R}) \\
& = A^{-1} (\diam \lambda R)^{d} \sim_{\lambda} A^{-1} \ell(R)^{d}.
\end{align*}

\end{proof}

\begin{lemma}
There is a universal constant $\theta>0$ so that 
\begin{equation}
\label{e:LD}
\sum_{Q\in \LD(R)} \ell(Q)^{d} \lec \tau^{\theta }\ell(R)^{d}.
\end{equation}

\end{lemma}

\begin{proof}
Again, to simplify notation, we just assume $\LD(R)=\LD^{+}(R)$ and $x_{R}=x_{R}^{+}$. Recall that for $Q\in \LD(R)$, there is a child $Q'$ of $Q$ so that $\Theta_{\omega^{x_{R}}}^{d}(Q)<\tau \Theta_{\omega^{x_{R}}}^{d}(R)$. Let $\LD'(R)$ be the set of these children. Then we clearly have 
\[
\sum_{Q\in \LD(R)} \ell(Q)^{d}\lec \sum_{Q\in \LD'(R)} \ell(Q)^{d},
\]
so we will estimate this latter sum. Note that for $Q\in \LD'(R)$, since $Q$ has a parent $\hat{Q}$ with $b\beta_{\d\Omega}(MB_{\hat{Q}},P_{\hat{Q}})<\ve$, we know by \eqref{e:close-to-P_Q} that
\[
 \dist(\zeta_{Q},\d\Sigma_{R})
\lec \ve \ell(Q)
\]
so there is $\zeta_{Q}'\in \d\Sigma_{R}$ with 
\begin{equation}
\label{e:z-z'}
|\zeta_{Q}-\zeta_{Q}'|\lec \ve \ell(Q).
\end{equation}
 Let 
\[
\xi_{Q} = g_{R}( g_{R}^{-1}(\zeta_{Q}')+\alpha d_{R}(\zeta_{Q}')e_{d+1} )\in \d\Omega_{R}.
\]
Note that since $\zeta_{Q}\in Q\subseteq R\subseteq  B_{R}$, $\zeta_{Q}'\in 2B_{R}$, and by \eqref{e:gz-z}, $g_{R}^{-1}(\zeta_{Q}')\in 3B_{R}$ for $\ve>0$ small, thus $g_{R}^{-1}(\zeta_{Q}')\in G_{R}$. Hence, \eqref{e:GR} implies
\[
\delta_{\Omega}(\xi_{Q})\sim \alpha d_{R}(\zeta_{Q}')
\]

Note that $d_{R}(\zeta_{Q})\leq \ell(Q)$ trivially, but also, if $T\in \Tree(R)$ is any other cube, then either $T\cap Q=\emptyset$ (in which case $\dist(\zeta_Q,T)\geq c_0 \ell(Q)$ by Theorem \ref{t:Christ}) or $T\supseteq Q$ (in which case $\ell(T)\geq \ell(Q)$), and thus in fact $d_{R}(\zeta_{Q})\geq c_{0}\ell(Q)$, so $d_{R}(\zeta_{Q})\sim \ell(Q)$. By \eqref{e:z-z'}, this also means that $d_{R}(\zeta_{Q}')\sim \ell(Q)$. Hence,
\[
\delta_{\Omega}(\xi_{Q})\sim \alpha d_{R}(\zeta_{Q}')\sim \alpha \ell(Q).
\]

Let $B^{Q}=B(\xi_{Q},\alpha^2 \ell(Q))$. Then for $\alpha\ll c_{0}$, the balls $\{B^{Q}:Q\in \LD(R)\}$ are disjoint. Moreover, using that $g_{R}$ is bi-Lipschitz,

\begin{align*}
|\xi_{Q}-\zeta_{Q}|
& \leq |\zeta_{Q}-\zeta_{Q}'|+
|\zeta_{Q}'-\xi_{Q}|\\
& \stackrel{\eqref{e:z-z'}}{ \lec} \ve \ell(Q)+|g_{R}^{-1}(\zeta_{Q}')-(g_{R}^{-1}(\zeta_{Q}')+\alpha d_{R}(\zeta_{Q}')e_{d+1} )|
\\
&=\ve \ell(Q)+\alpha d_{R}(\zeta_{Q}')\lec \alpha  \ell(Q).
\end{align*}
Thus,  for $\alpha$ small,
\begin{equation}
\label{e:B^Qinc0/4 B_Q}
B^{Q}\subseteq \frac{c_{0}}{4}B_{Q}\cap \Omega.
\end{equation}

Lemma \ref{l:bourgain} implies that 
\begin{equation}
\label{e:LDbourgain}
\omega_{\Omega}^{x}(c_{0}B_{Q})\gec 1 \;\; \mbox{ for }x\in B^{Q}.
\end{equation}
Hence, by the maximum principle,  since $\Omega_{R}\subseteq \Omega$ is a CAD (and hence $\d\Omega_{R}$ is AR),
\begin{align*}
\omega_{\Omega_{R}}^{x_{R}}\ps{\bigcup_{Q\in \LD'(R)}B^{Q}}
& \lec  \omega_{\Omega}^{x_{R}} \ps{\bigcup_{Q\in \LD'(R)} c_{0}B_{Q}}
 \stackrel{\eqref{e:containment}}{\leq} \sum_{Q\in \LD'(R)}\omega_{\Omega}^{x_{R}}(Q)\\
& <\tau  \Theta_{\omega^{x_{R}}}(R) \sum_{Q\in \LD'(R)} \ell(Q)^{d}\\
& \stackrel{\eqref{e:Q<c0}}{\lec} \tau  \Theta_{\omega^{x_{R}}}(R) \sum_{Q\in \LD'(R)} \cH^{d}|_{\d\Omega_{R}}(B^{Q})\\
& =\tau \Theta_{\omega^{x_{R}}}(R)  \cH^{d}|_{\d\Omega_{R}}\ps{\bigcup_{Q\in \LD'(R)} B^{Q}}\\
& \lec \tau  \Theta_{\omega^{x_{R}}}(R) \cH^{d}(\d\Omega_{R})\\
& \lec  \tau  \Theta_{\omega^{x_{R}}}(R)\ell(R)^{d}=\tau\omega_{\Omega}^{x_{R}}(R)\leq \tau. 
\end{align*}

By Theorem \ref{t:AMT-HM} or the main result of \cite{DJ90}, $\omega_{\Omega}^{x_{R}}$ is $A_{\infty}$-equivalent to $\cH^{d}|_{\d\Omega_{R}}$ (with constants only depending on the CAD constants of $\Omega_{R}$, which don't depend on any of our parameters apart from $d$). In particular, if $\cH^{d}|_{\d\Omega_{R}} = p_{R} \omega_{\Omega_{R}}^{x_{R}}$, then the function $p_{R}$ satisfies a reverse H\"older inequality (\cite[Section V.5]{Big-Stein}), that is, there is $1< p<\infty$ (depending on the CAD constants for $\Omega_{R}$), so that 

\[
\ps{\avint_{\d\Omega_{R}} p_{R}^{p}d\omega_{\Omega_{R}}^{x_{R}}}^{\frac{1}{p}} \lec \avint_{\d\Omega_{R}}p_{R}d\omega_{\Omega_{R}}^{x_{R}}=\frac{\cH^{d}(\d\Omega_{R})}{\omega_{\Omega_{R}}^{x_{R}}(\d\Omega_{R})}\sim \ell(R)^{d}.
\]
Hence, for any set $F\subseteq \d\Omega_{R}$, by H\"older's inequality, if $\frac{1}{p}+\frac{1}{p'}=1$,
\[
\cH^{d}(F)
\leq \ps{\avint_{\d\Omega_{R}} p_{R}^{p}d\omega_{\Omega_{R}}^{x_{R}}}^{\frac{1}{p}}
\omega_{\Omega_{R}}^{x_{R}}(F)^{\frac{1}{p'}}\lec \ell(R)^{d}\omega_{\Omega_{R}}^{x_{R}}(F)^{\frac{1}{p'}}.
\]
Letting $\theta=1/p'$ and $F=\bigcup_{Q\in \LD'(R)} B^{Q}$, our previous estimates now give
\begin{align*}
\tau^{\theta}\ell(R)^{d}
& 
\gec \cH^{d}|_{\d\Omega_{R}} \ps{\bigcup_{Q\in \LD'(R)} B^{Q}}
\sim \sum_{Q\in \LD'(R)} \ell(Q)^{d} \gec  \sum_{Q\in \LD(R)} \ell(Q)^{d} .
\end{align*}

\end{proof}

\begin{lemma}
For $R\in \Top$,
\begin{equation}
\label{e:Bbeta}
\sum_{Q\in B\beta(R)} \ell(Q)^{d}\lec \sum_{T\in \Tree(R)} \beta(MB_{T})^{2} \ell(T)^{d}.
\end{equation}
\end{lemma}

\begin{proof}

Using \eqref{e:beta-Jones>ve^2*}  and the fact that the cubes in $\Stop(R)$ partition $R$ by Lemma \ref{l:minimal-partition},

\begin{align*}
\sum_{Q\in B\beta(R)} \ell(Q)^{d}
& \leq \ve^{-2} \sum_{Q\in B\beta(R)} \sum_{Q\subseteq T\subseteq R} \beta(MB_{T})^{2}  \ell(Q)^{d}\\
& \leq \ve^{-2} \sum_{T\in \Tree(R)} \beta(MB_{T})^{2} \sum_{Q\in \Stop(R)\atop Q\subseteq T}  \ell(Q)^{d} \\
& \lec \ve^{-2} \sum_{T\in \Tree(R)} \beta(MB_{T})^{2} \sum_{Q\in \Stop(R)\atop Q\subseteq T}  \cH^{d}(\Sigma_{R}\cap c_{0}B_{Q})\\
& \leq  \ve^{-2} \sum_{T\in \Tree(R)} \beta(MB_{T})^{2} \cH^{d}(\Sigma_{R}\cap B_{T} )\\
&  \stackrel{\eqref{e:ETregular}}{\lec}\ve^{-2} \sum_{T\in \Tree(R)} \beta(MB_{T})^{2} \ell(T)^{d}.
\end{align*}

\end{proof}

\begin{lemma}
For $R\in \Top$, 
\begin{equation}
\label{e:bottomsum}
\sum_{Q\in \BTM(R)}\ell(Q)^{d}\lec \ell(R)^{d}.
\end{equation}
\end{lemma}

\begin{proof}
Note by \eqref{e:Q<c0}, and since the balls $c_0B_Q$ are disjoint for $Q\in  \BTM(R)$,
\[
\sum_{Q\in \BTM(R)} \ell(Q)^{d}
 \lec \sum_{Q\in \BTM(R)} \cH^{d}(\Sigma_{R}\cap c_{0}B_{Q})
  \leq \cH^{d}(\Sigma_{R}\cap 2B_{R} )
  \stackrel{\eqref{e:ETregular}}{\lec}\ell(R)^{d}.
\]

\end{proof}

Let $\Stop'(R)=\Stop(R)\backslash \BTM(R)$. Combining all these estimates, we see that 
\[
\sum_{Q\in\Stop'(R)}\ell(Q)^{d}
\lec (A^{-1}+\tau^{\theta})\ell(R)^{d} + \underbrace{\ve^{-2} \sum_{T\in \Tree(R)} \beta(MB_{T})^{2} \ell(T)^{d}+ \sum_{Q\in \Bad(R)}\ell(R)^{d}}_{=:I(R)}.
\]
Note that if $R\in \Top_k$ for $k\geq 1$, then $R$ is the child of a cube in $\Stop'(R')$ for some $R'\in \Top_{k-1}$. Thus,
\begin{align*}
\sum_{k=0}^{\infty} \sum_{R\in \Top_{k}}\ell(R)^{d}
& = \ell(Q_{0})^{d}+
\sum_{k=1}^{\infty} \sum_{R\in \Top_{k}}\ell(R)^{d}\\
& \lec \ell(Q_{0})^{d}+\sum_{k=1}^{\infty}\sum_{R\in \Top_{k-1}}\sum_{Q\in \Stop'(R)} \ell(Q)^{d}\\
& \lec \ell(Q_{0})^{d}+ (A^{-1}+\tau^{\theta})\sum_{k=0}^{\infty}\sum_{R\in \Top_{k}}\ell(R)^{d} +\sum_{k=0}^{\infty}\sum_{R\in \Top_{k}}I(R)
\end{align*}
Thus, for $A$ large and $\tau$ small enough, and because the sets $\Tree(R)$ partition $\cT(k_{0})$, 
\begin{align*}
\sum_{k=0}^{\infty} \sum_{R\in \Top_{k}}\ell(R)^{d}
\lec \sum_{k=0}^{\infty}\sum_{R\in \Top_{k}}I(R)
\stackrel{\eqref{e:sumbad}}{\lec}\ve^{-2} \beta_{\d\Omega}(\cT).
\end{align*}
Let $\BTM(k_0)$ be those cubes $Q$ from $\BTM$ that are a child of some cube $Q'\in \cT(k_0)$ (so each cube from $\BTM$ is in $\BTM(k_0)$ for some $k_0\geq 0$). Since the trees partition $\cT(k_0)$, we also see that 
\[
\sum_{Q\in \BTM(k_0)}\ell(Q)^{d}
\leq \sum_{Q\in \BTM(k_0)}\ell(Q')^{d}
\leq \sum_{R\in \Top} \sum_{Q'\in \BTM(R)}\ell(Q')^{d}
\lec \sum_{R\in \Top}\ell(R)^{d}
.
\]
Taking $k_0\rightarrow\infty$ (and recalling Remark \ref{r:k_0}), this completes the proof of \eqref{e:totalsum}.

%
%Hence, recalling that the parent of any cube in $\Top$ is a stopped cube or $Q_{0}$,
%\begin{align*}
%\sum_{R\in \Top}\ell(R)^{d}
%& \lec \ell(Q_0)^{d} + \sum_{R\in \Top}\sum_{Q\in \Stop(R)}\ell(Q)^{d}\\
%& \lec \ell(Q_{0})^{d} + \sum_{R\in \Top}
%

\end{proof}

We record the following Corollary of the proof for the case of semi-uniform domains. 

\begin{corollary}
\label{c:SU-cor}
With the same assumptions as Lemma \ref{l:CDHM<beta}, if $\Omega$ is also SU and $x\in \Omega\backslash MB_{Q_{0}}$, for $\lambda\geq 1$, and for $A,\tau^{-1}$ sufficiently large, we may find cubes $\Top$ contained in $Q_0$ and a partition of $\cT$ into trees $\{\Tree(R):R\in \Top\}$ so that for each $R\in \Top$, if $\omega=\omega^{x}$,
\begin{enumerate}
\item $\tau \Theta_{\omega}^{d}(\lambda B_R)\leq \Theta_{\omega}^{d}(\lambda B_Q) \leq A \Theta_{\omega}^{d}(\lambda B_R)$ for all $Q\in \Tree(R)$.
\item If $\Stop(R)$ denote the minimal cubes of $\Tree(R)$ and $Q\in \Stop(R)$, then either $Q$ as a child in $\BTM$, $\Theta_{\omega}^{d}(Q)\sim \tau \Theta_{\omega}^{d}(\lambda B_R)$ or $\Theta_{\omega}^{d}(\lambda B_Q)\sim A \Theta_{\omega}^{d}(\lambda B_R)$ (with implied constants independent of $\tau,\lambda$ and $A$).
\item \eqref{e:btm+top<b} holds.
\end{enumerate}
% there is a partition of the cubes in $\cT$ into stopping-time regions satisfying (1) and (2) and the forward inequality in (3) in Theorem \ref{t:SU-version}. 
\end{corollary}

\begin{proof}
Assume the same set-up as in the proof of Lemma \ref{l:CDHM<beta}. Let $\omega = \omega^{x}$. Let $R\in \Top$ and $Q\in \Tree(R)$. Let $x_{1}$ and $x_{2}$ be the corkscrew points for $\lambda B_{R}$ in Lemma \ref{l:w/w} (with $M\lambda$ in place of $M$ and $\lambda B_R$ in place of $B$). Notice that since $b\beta(MB_{R})<\ve$, we know that each of these corkscrew points is connected by a Harnack chain in $\Omega$ to either $x_{R}^{\pm}$, so without loss of generality, we can assume that $x_{1},x_{2}\in \{x_{R}^{\pm}\}$. Note that since $\omega$ is doubling, we have 
\begin{equation}
\label{e:qlambdaq}
 \Theta_{\omega}^{d}(\lambda B_{Q})\sim_{\lambda} \Theta_{\omega}^{d}(Q) \;\; \mbox{ for all }Q\subseteq Q_0.
 \end{equation}

 Then Lemma \ref{l:w/w} and the doubling property for harmonic measure implies that 
\begin{multline*}
 \frac{\Theta_{\omega}^{d}(\lambda B_{Q})}{\Theta_{\omega}^{d}(\lambda B_{R})}
 = \frac{\ell(R)^{d}}{\ell(Q)^{d}}  \frac{\omega(Q)}{\omega(R)} 
\stackrel{\eqref{e:qlambdaq}}{\sim} \frac{\ell(R)^{d}}{\ell(Q)^{d}}  \frac{\omega(\lambda B_Q)}{\omega(\lambda B_{R})} 
\lec \frac{\ell(R)^{d}}{\ell(Q)^{d}} \omega^{x_{2}}(\lambda B_{Q})
\sim \ell(R)^{d} \Theta_{ \omega^{x_{2}}}^{d}(\lambda B_{Q})\\
\leq A  \ell(R)^{d}  \Theta_{\omega^{x_{2}}}^{d}(\lambda B_{R}) \lec  A
\end{multline*}
and similarly, 
\begin{multline*}
 \frac{\Theta_{\omega}^{d}(\lambda B_{Q})}{\Theta_{\omega}^{d}(\lambda B_{R})} 
\stackrel{\eqref{e:qlambdaq}}{\sim}\frac{\ell(R)^{d}}{\ell(Q)^{d}}  \frac{\omega(\lambda B_Q)}{\omega(\lambda B_{R})} 
\gec \frac{\ell(R)^{d}}{\ell(Q)^{d}} \omega^{x_{1}}(\lambda B_{Q})
\sim \ell(R)^{d} \Theta_{ \omega^{x_{1}}}^{d}(\lambda B_{Q})\\
\geq \tau  \ell(R)^{d}  \Theta_{\omega^{x_{1}}}^{d}(\lambda B_{R}) \stackrel{\eqref{e:bourgain}}{\gec} \tau
\end{multline*}

Hence,
\begin{equation}
\label{e:intree}
\tau \Theta_{\omega}^{d}(\lambda B_{R})  \lec \Theta_{\omega}^{d}(\lambda B_{Q})\lec A\Theta_{\omega}^{d}(\lambda B_{R})
\;\;\; \mbox{ for all }Q\in \Tree(R).
\end{equation}

\def\wt{\widetilde}
Now we run a new stopping-time algorithm. Let $C>0$ be a large constant to be decided later. For $R\subseteq Q_0$, let $\wt{HD}(R)$ denote the maximal cubes $Q\subseteq R$ which contain a child $Q'$ so that $\Theta_{\omega}^d(Q')>CA\Theta_{\omega}^d(R)$ and $\wt{LD}(R)$ be those maximal cubes $Q\subseteq R$ which contain a child $Q'$ so that  $\Theta_{\omega}^d(Q')<C^{-1}\tau\Theta_{\omega}^d(R)$. Since $\omega$ is doubling, this and \eqref{e:intree} imply that 
\[
\Theta_{\omega}^d(Q')\sim CA\Theta_{\omega}^d(R) \;\; \mbox{ for }Q\in \wt{HD}(R)
\]
and 
\[
\Theta_{\omega}^d(Q')\sim C^{-1}\tau\Theta_{\omega}^d(R) \;\; \mbox{ for }Q\in \wt{LD}(R).
\]
Let $\wt{\Next}(R)$ denote the children of the cubes in $\wt \Stop(R):=\wt{HD}(R)\cup \wt{LD}(R)$ that are also in $\cT$.. Let $\wt\Top_{0}=\{Q_0\}$ and for $k\geq 0$ let 
\[
\wt{\Top}_{k+1} = \bigcup_{R\in \wt\Top_{k}}\wt{\Next}(R).
\]
Let $\wt\Top=\bigcup_{k\geq 0} \wt\Top_{k}$ and for $R\in \wt\Top$, let $\wt\Tree(R)$ be those cubes in $R$ in $\cT$ not properly contained in a cube from $\wt\Stop(R)$. Recalling \eqref{e:qlambdaq}, it is clear that these trees satisfy (1) and (2), since the minimal cubes in $\wt\Tree(R)$ are either in $\wt\Stop(R)$ or in $\BTM(R')$ for some $R'\in \Top$ and hence has a child in $\BTM$, so we just need to verify (3). 

Note that for each $R\in \wt\Top$, there is $R'\in \Top$ so that $R\in \Tree(R')$. Moreover, for $C$ large enough, $\wt\Stop(R)\subseteq \Tree(R)^{c}$. Thus, for fixed $R'\in \Top$, the cubes in $\{R\in \wt\Top: R\in \Tree(R')\}$ are disjoint. Thus,

\begin{align*}
\sum_{R\in \wt\Top}\ell(R)^{d}
& =\sum_{R'\in \Top}\sum_{R\in \wt\Top\atop R\in\Tree(R')}\ell(R)^{d}
\stackrel{\eqref{e:intree}\atop \eqref{e:qlambdaq}}{\sim}_{A,\tau}  \sum_{R'\in \Top} \Theta_{\omega}^{d}(\lambda B_{R'} )^{-1}\sum_{R\in \wt\Top\atop R\in\Tree(R')}\omega(R)\\
& \leq  \sum_{R'\in \Top}  \Theta_{\omega}^{d}(\lambda B_{R'})^{-1} \omega(R')
\leq  \sum_{R'\in \Top}   (\lambda \ell(R'))^{d}
\stackrel{\eqref{e:totalsum}}{ \lec} \beta_{\d\Omega}(\cT)
 \end{align*}

The estimate on $\sum_{Q\in \BTM} \ell(Q)^{d}$ is shown in the same way as before. This concludes the proof.

%
%
%These estimates imply (1) in Theorem \ref{t:SU-version}. 
%
%
%Let $Q\in \HD(R)\backslash\{R\}$, then there is $x\in \{x_{R}^{\pm}\}$ and a $Q$ has a child $Q'$ so that $\Theta_{\omega^{x}}^{d}(\lambda B_{Q'})>A\Theta_{\omega^{x}}^{d}(\lambda B_{R})$. By the above estimates, this also implies $\Theta_{\omega^{x}}^{d}(\lambda B_{Q'})\gec A\Theta_{\omega}^{d}(\lambda B_{R})$.
%\[
%A\Theta_{\omega}^{d}(\lambda B_{R})
%\lec \Theta_{\omega}^{d}(\lambda B_{Q'})
%\lec  \Theta_{\omega}^{d}(\lambda B_{Q})
%\leq A\Theta_{\omega}^{d}(\lambda B_{R}).
%\]
%If $Q\in \LD(R)\backslash \{R\}$, then there is $x\in \{x_{R}^{\pm}\}$ and a child $Q'$ so that  $\Theta_{\omega^{x}}^{d}(\lambda B_{Q'})<\tau \Theta_{\omega}^{d}(\lambda B_{R})$. By the above estimates again, we also have $\Theta_{\omega^{x}}^{d}(\lambda B_{Q'})<\tau \Theta_{\omega}^{d}(\lambda B_{R})$. By Lemma \ref{l:doubling},  
%\[
%\tau \Theta_{\omega^{x}}^{d}(\lambda B_{R})
%\lec  \Theta_{\omega}^{d}(\lambda B_{Q})
%\stackrel{\eqref{doubling}}{\sim}   \Theta_{\omega}^{d}(\lambda B_{Q'})
%\lec \tau \Theta_{\omega}^{d}(\lambda B_{R}).
%\]
%This completes the proof of the corollary.

\end{proof}

\section{Proof of Theorms \ref{t:CDHM} and \ref{t:SU-version}: Part II}
\label{s:I-II-Part-II}

The goal of this section is to complete the proofs of  Theorems \ref{t:CDHM} and \ref{t:SU-version}. 

Combined with Lemma \ref{l:CDHM<beta} and Corollary \ref{c:SU-cor}, Theorms \ref{t:CDHM} and \ref{t:SU-version} will follow from the following lemma, whose proof is the objective of this section. 

  \begin{lemma}
 \label{l:beta<CDHM}
 Let $\Omega\subseteq \R^{d+1}$ be a interior $c$-corkscrew domain with lower $d$-content regular boundary and let $\cD$ be the Christ-David cubes for $\d\Omega$. For $\lambda\geq 2$ and for $A,\tau^{-1}>0$ large enough (depending on $\lambda$), we have that for all $Q_0\in \cD$,
 \begin{equation}
 \label{e:b<CDHM}
\beta_{\d\Omega}(Q_0)\lec_{A,\tau,\lambda } \CDHM(Q_0,\lambda ,A,\tau).
 \end{equation}
 \end{lemma}

\begin{remark}
It suffices to prove the lemma when $\lambda =2$. 
%We first make some observations about $\CDHM(R,\lambda,A,\tau)$. Notice that for any $\tau'\leq \tau<1<A\leq A'$,
%\[
%\CDHM(R,\lambda,A,\tau)\geq  \CDHM(R,\lambda,A',\tau')
%\]
To see this, let $\Top$ and $\{\Tree(R):R\in \Top\}$ be as in Definition \ref{d:CDHM} with constants $\lambda\geq 2$ and $A,\tau>0$. Let $\Tree'(R)$ be those cubes in $\Tree(R)$ where the $N$th generation descendants of their children are in $\Tree(R)$ (we define this in terms of the children to ensure that $\Tree'(R)$ is a stopping-time region). Then for $Q\in \Tree'(R)$, 

\begin{align*}
\Theta_{\omega^{x_{R}}}(2B_{Q})\leq \ps{\lambda/2}^{d} & \Theta_{\omega^{x_{R}}}(\lambda B_{Q})
\leq \ps{\lambda/2}^{d} A\Theta_{\omega^{x_{R}}}(\lambda B_{R})
\leq \ps{\lambda/2}^{d} A (2\lambda \ell(R))^{-d}
\\
& \stackrel{\eqref{e:bourgain}}{\leq} \frac{CA}{2^{d}} \Theta_{\omega^{x_{R}}}(2B_{R})
\end{align*}
and for $N$ large enough (depending on $\lambda$), if $Q'\in \Tree(R)$ is an $N$th generation descendant of $Q$, then $\lambda B_{Q'}\subseteq 2B_{Q}$, and so
\[
\Theta_{\omega^{x_{R}}}(2B_{Q})
\geq \ps{\frac{\lambda \rho^{N}}{2}}^{d} \Theta_{\omega^{x_{R}}}(\lambda B_{Q'})
\geq  \ps{\frac{\lambda \rho^{N}}{2}}^{d} \tau \Theta_{\omega^{x_{R}}}(\lambda B_{R})
\leq \rho^{dN}\tau  \Theta_{\omega^{x_{R}}}(2 B_{R}).
\]

Hence, letting $A'=A\max\{C/2^{d},1\}$ and $\tau' = \rho^{dN}\tau$, we see that 
\[
\Theta_{\omega^{x_{R}}}(2 B_{R})/ \Theta_{\omega^{x_{R}}}(2 B_{Q})\in [\tau',A'] \;\; \mbox{ for all $Q\in \Tree'(R)$},
\]
and so 
\[
\CDHM(Q_0,2,A',\tau')\leq \sum_{R\in \Top} \ell(R)^{d}
\]
Now let 
\[
\Top'  =\Top \cup \bigcup_{R\in \Top}\Tree(R)\backslash \Tree'(R)
\]
and for $R\in \Top$ and $R'\in \Tree(R)\backslash \Tree'(R)$, we simply set $\Tree'(R')=\{R'\}$. Then by Definition \ref{d:CDHM}, we have 
\[
\CDHM(Q_0,2,A',\tau') 
\leq \sum_{R\in \Top'}\ell(R)^{d}
\lec_{N} \sum_{R\in \Top}\ell(R)^{d}
\]
and infimizing over all such collections $\{\Tree(R):R\in \Top\}$, we get 
\[
\CDHM(Q_0,2,A',\tau')\lec_{N}   \CDHM(Q_0,\lambda,A,\tau).
\]

In particular, for $A,\tau^{-1}$ large enough (so that $A',(\tau')^{-1}$ are large enough), \eqref{e:b<CDHM} with $\lambda =2$ implies \eqref{e:b<CDHM} for all $\lambda \geq 2$. 
 
\end{remark}

Before we prove Lemma \ref{l:beta<CDHM}, we make a few preliminary estimates comparing different multiscale geometric properties. 

The approach is essentially that of \cite{HM15,HLMN17}. There, they use the weak-$A_{\infty}$ property to control the behavior of the density of harmonic measure and thus show that Green's function is often affine, implying that the boundary (outside of a summable set of cubes) is flat from one side. We make ``flat from one side" more precise in the following definition.

\begin{definition}
Let $E\subseteq \R^{d+1}$ be a lower $d$-regular set and $\cD$ denote the Christ-David cubes for $E$. For $\ve>0$, we let $\cB^{\WHSA}_{K_0,\ve}$ denote those cubes $Q$ for which there is a half-plane $H_{Q}$ such that 
\[
H_{Q}\cap \ve^{-2}B_{Q}\subseteq E^{c},\]
\[
\dist(Q,\d H_{Q})\leq K_0 \ell(Q),\]
and
\[
\sup_{x\in \d H_{Q}\cap \ve^{-2}B_{Q}}\dist(x,E)<\ve \ell(Q).
\]
Let $\cB^{\WHSA}_{K_0,\ve}(R)$ be those cubes in $\cB^{\WHSA}_{K_0,\ve}$ contained in $R$. Let $\cG^{\WHSA}(R,K_0,\ve)$ be those cubes in $R$ that are not in $\cB_{K_0,\ve}^{\WHSA}(R)$. If $E$ is also Ahlfors regular, we say it satisfies the {\it Weak Half Space Approximation property} (or WHSA) if there are $\ve_0,K_0>0$ so that for $0<\ve<\ve_{0}$, and for any surface cube $R\in \cD$,
\[
\WHSA(R,K_0,\ve) :=\sum_{Q\in \cB_{K_0,\ve}^{\WHSA}(R)}\ell(Q)^{d}
\lec_{\ve,K_0} \ell(R)^{d}\]
\end{definition}

 In \cite[Section 5]{HM15}, it was shown that this is equivalent uniform rectifiability for Ahlfors regular sets.

 Lemma \ref{l:beta<CDHM} will follow from the earlier remark and the following three lemmas.

 \begin{lemma}
 \label{l:H<CDHM}
  Let $\Omega \subseteq \R^{d+1}$ be a corkscrew domain with lower $d$-content regular boundary and let $\cD$ be the Christ-David cubes for $\d\Omega$. Then for $\lambda \geq 1$, $R\in \cD$, and $A,\tau^{-1}>0$ large enough
  \[
  \cH^{d}(R)\lec_{A,\tau,\lambda }  \CDHM(R,\lambda , A,\tau).
  \]
  \end{lemma}

 \begin{lemma}
 \label{l:beta<WHSA}
  Let $\Omega \subseteq \R^{d+1}$ be a corkscrew domain with lower $d$-content regular boundary and let $\cD$ be the Christ-David cubes for $\d\Omega$. Then for all $K_0\geq 1$ and $\ve>0$ small enough, 
 \[
\beta_{\d\Omega}(R)\lec_{K_{0},\ve}  \cH^{d}(R)+ \WHSA(R,K_0,\ve)
 \]
 \end{lemma}

 \begin{lemma}
 \label{l:WHSA<CDHM}
  Let $\Omega \subseteq \R^{d+1}$ be a corkscrew domain with lower $d$-content regular boundary and let $\cD$ be the Christ-David cubes for $\d\Omega$. Then for all $\ve^{-1},A,\tau^{-1}>0$  large enough, we have that for all $R\in \cD$,
 \[
 \WHSA(R,5,\ve) \lec_{K_0,\ve,A,\tau} \CDHM(R,2,A,\tau).
 \]
 \end{lemma}

 We divide the proofs of these lemmas into subsections:
 
 \subsection{Proof of Lemma \ref{l:H<CDHM}}
 
Let $Q_{0}\in \cD$. Without loss of generality, $Q_{0}\in \cD_{0}$. Let $\Top$ and $\Tree(R)$ for $R\in \Top$ be as in Definition \ref{d:CDHM}. Let $k_0\in \bN$, $\cD(k_{0})=\{Q\subseteq Q_{0}: \ell(Q)\geq \rho^{k_{0}}\ell(Q_{0})\}$, and 
\[
\Top(k_{0})=\Top \cap \cD(k_{0}), \;\; \Tree_{k_{0}}(R) = \Tree(R)\cap \cD(k_{0}).
\]

For $R\in \Top(k_{0})$, let $\Stop_{k_{0}}(R)$ be the minimal cubes in $\Tree_{k_{0}}(R)$ that are in $\cD_{k_{0}}$. For $Q\in \Stop_{k_{0}}(R)$, let $\Delta$ denote the dyadic cubes in $\R^{n}$ and
\[
C(Q) = \{ I\in \Delta : I\cap 5\lambda B_Q\neq\emptyset, \ell(I)<\ell(Q)\leq 2\ell(I)\}
\]
and 
\[
E(R) = \bigcup_{Q\in \Stop_{k_{0}}(R)} \bigcup_{I\in C(Q)} I.
\]
Note that as $\cD(k_0)$ is finite and $\{\Tree_{k_{0}}(R):R\in \Top(k_0)\}$ partitions $\cD(k_0)$, 
\[
\cD_{k_{0}}\subseteq \bigcup_{R\in \Top(k_{0})}\Stop_{k_{0}}(R).
\]
Hence, since $Q\subseteq  \bigcup_{I\in C(Q)}I$, and every $Q\in \cD_{k_{0}}$ is in $\Stop_{k_{0}}(R)$ for some $R\in \Top(k_0)$,
\[
Q_0 \subseteq \bigcup_{Q\in \cD_{k_{0}}\atop Q\subseteq Q_{0}} Q \subseteq \bigcup_{R\in \Top(k_{0})} \bigcup_{Q\in \Stop_{k_{0}}(R)}   \hspace{-15pt} Q\subseteq \bigcup_{R\in \Top(k_{0})} \bigcup_{Q\in \Stop_{k_{0}}(R)}  \bigcup_{I\in C(Q)}I
=
\bigcup_{R\in \Top(k_0)} \hspace{-15pt} E(R).
\]

Let $\lambda B_{Q_j}$ be a Vitali subcover of $\{\lambda B_{Q}:Q\in \Stop_{k_{0}}(R)\}$, so they are disjoint and $\bigcup_{Q\in \Stop_{k_{0}}(R)}\lambda B_{Q}\subseteq \bigcup 5\lambda B_{Q_{j}}$. 

Also note that for $Q\in \cD_{k_{0}}$, $\ell(Q)=5\rho^{k_{0}}$, and so for $I\in C(Q)$, $\diam I\leq 5\rho^{k_{0}}\sqrt{n}$. Thus,

\begin{align*}
\cH^{d}_{5\rho^{k_0}\sqrt{n}}(E(R))
& \leq \sum_{j} \sum_{I\in C(Q_j)} (\diam I)^{d}
\lec_{\lambda} \sum_{j}\ell(Q_j)^{d}\\
& \lec_{\lambda} \tau^{-1} \Theta_{\omega^{x_{R}}}^{d}( \lambda B_{R})^{-1}  \sum_{j} \omega^{x_{R}}(\lambda B_{Q_{j}})\\
&  \leq\tau^{-1}\Theta_{\omega^{x_{R}}}^{d}( \lambda B_{R})^{-1}   \omega^{x_{R}}(5\lambda B_{R})
\lec_{\lambda} \tau^{-1} \ell(R)^{d}. 
\end{align*}

Thus,
\[
\cH^{d}_{5\rho^{k_0}\sqrt{n}}(Q_0)
\leq \sum_{R\in \Top(k_{0})}\cH^{d}_{5\rho^{k_{0}}}(E(R))
\lec_{\tau,\lambda} \sum_{R\in \Top(k_{0})}\ell(R)^{d}<\CDHM(R,\lambda , A,\tau).
\]

 Letting $k_0\rightarrow\infty$ gives the result and completes the proof of Lemma \ref{l:H<CDHM}.

 \subsection{Proof of Lemma \ref{l:beta<WHSA}}

The proof of this lemma will require a bit more theory from quantitative rectifiability. We recall this intermediate geometric property. 

\begin{definition}[Bilateral Approximation by a Union of Planes (BAUP)] 
\label{d:baup}
Let $E\subseteq \R^{n}$ be lower $d$-content regular. For $\ve,M>0$, and $R\in \cD$, let $\cG^{\BAUP}_{\ve,M}$ denote those cubes $Q$ for which there is a union $U$ of $d$-dimensional planes for which 
\[
d_{MB_{Q}}(E,U)<\ve.
\]
Let $\cB^{\BAUP}_{\ve,M}=\bD\backslash \cG^{\BAUP}_{\ve,M}$, $\cB^{\BAUP}_{\ve,M}(R)$ be those cubes in $\cB^{\BAUP}_{\ve,M}$ contained in $R$, and 
\[
\BAUP(R,M,\ve) = \sum_{Q\in \cB^{\BAUP}_{\ve,M}(R)} \ell(Q)^{d}.
\]
\end{definition}

These cubes were introduced by David and Semmes \cite{of-and-on} who showed that an Ahlfors regular set $E$ is UR if and only if for each $M>1$ there is $\ve_0$ so that for $0<\ve<\ve_{0}$, if 
\[
\BAUP(R,M,\ve)\lec \ell(R)^{d} \mbox{ for all }R\in \cD.
\]

In \cite{AV19}, a version of this result was shown for just lower regular sets:

\begin{theorem} \cite[Theorem 1.4]{AV19}
\label{t:BAUP}
Let $E\subseteq \R^{n}$ be $(c,d)$-lower content regular. Then for $M>1$ and $\ve>0$ small enough (depending on $M$ and the lower content regularity constant) and $R\in \cD$,
\[
\beta_{E}(R)
\sim_{\ve,M} \cH^{d}(R)+\BAUP(R,M,\ve).
\]
\end{theorem}

 In \cite[Section 5]{HM15}, the authors show the following (they assume Ahlfors regularity, but an inspection of the proof shows that it is not needed):
 
 \begin{lemma}
 \label{e:WHSA-BAUP}
 Let 
 \[
 D_{\ve}(Q) = \{Q'\in \cD: \ve^{3/2}\ell(Q)\leq \ell(Q')\leq \ell(Q), \;\; Q'\cap K_0^2 B_{Q}\neq\emptyset \}. 
 \]
Let $\tilde{\cB}^{\WHSA}_{\ve,K_0}$ be those cubes $Q$ for which $D_{\ve}\cap {\cB}^{\WHSA}_{\ve,K_0}\neq\emptyset$. Then
\[
 \cB^{\BAUP}_{\ve,10}\subseteq  \tilde{\cB}^{\WHSA}_{\ve,K_0} . 
\]

 \end{lemma}

Let 
\[
\widetilde{\WHSA}(R,\ve,K_0)= \sum_{Q\subseteq R \atop Q\in \tilde{\cB}_{\ve,K_0}^{\WHSA}(R)} \ell(Q)^{d}.
\]
Then the previous lemma and Theorem \ref{t:BAUP} imply

\begin{align*}
\beta_{E}(R)^{d}
& \sim_{\ve} \cH^{d}(R)+\BAUP(\ve,10)(R)
\leq \cH^{d}(R)+\wt{\WHSA}(R,\ve,K_0)\\
& \lec_{\ve}  \cH^{d}(R)+{\WHSA}(R,\ve,K_0)
\end{align*}

This completes the proof of  Lemma \ref{l:beta<WHSA}.

 \subsection{Proof of Lemma \ref{l:WHSA<CDHM}}
 
 This proof is modelled after that in \cite{HM15}, and Hofmann and Martell attribute this line of attack to \cite{LV07} (there are also some common aspects  to the proof of Theorem \ref{t:BJ94}). However, we have to take some care since our sets are not Ahlfors regular, but we can fix this by using Lemma \ref{l:corona}.

We first need a lemma that says if a harmonic function is approximately affine, in some ball, then the boundary is approximately flat near that ball. This has been proved elsewhere before, \cite{HM15, HLMN17} for example, but since those proofs require Ahlfors regularity, we give a different proof here.

\begin{lemma}
\label{l:halfspace}
Let $\Omega\subseteq \R^{d+1}$ be a corkscrew domain with lower $d$-content regular complement, $B$ a ball centered on $\d\Omega$ with $2r_{B}<\diam \d\Omega$. Let $\ve>0$ and let $u$ be a harmonic function on $\Omega$ that is positive on $\ve^{-2}B\cap \Omega$ and vanishing continuously on $\frac{2}{\ve^{2}}B\cap \d\Omega$. Suppose also that there is a corkscrew ball $B(y_{B},2cr_{B})\subseteq   \Omega\cap B$ so that 
\[
\sup_{\frac{2}{\ve^{2}}B\cap \Omega} u\lec u(y_{B}).
\]
Also suppose that 
\[
\avint_{B(y_{B},cr_{B})} |\grad^2 u|^{2}  <\delta \frac{u(y_{B})^2}{r_{B}^{4}}
\]
Then there is a half space $H_{B}$ so that 
\[
\d H_{B}\cap B\neq \emptyset, \;\;  
H_{B}\cap \ve^{-2}B\subseteq \Omega\]
and 
\[
\sup_{x\in \d H_{B}\cap \ve^{-2}B}\dist(x,\d\Omega)<\ve r_{B}.
\]
\end{lemma}

\begin{proof}
Without loss of generality, we can assume $B=\bB$ and $u(y_{B})=1$. Let $\bB'=\ve^{-2}\bB$. Suppose instead that for all $j$ there are domains $\Omega_j$ and harmonic functions $u_{j}$ positive on $\bB'\cap \Omega$, vanishing continuously on $\d\Omega\cap \bB'$, and there is $B_{j}=B(y_{j},c)\subseteq \Omega_{j}\cap \bB'$ such that 
\[
\sup_{2\bB'\cap \Omega_j} u_j\lec 1
\]
and 

\[
\avint_{B_j} |\grad^2 u_j|^{2}  < u(y_{j})^2/j=1/j
\]
but for every half space $H$, either
\[
H\cap \bB'\not\subseteq \Omega_j\]
or 
\[
\sup_{x\in \d H\cap \bB'}\dist(x,\d\Omega_j)\geq \ve.
\]
By Lemma \ref{l:holder}, the $u_{j}$ are uniformly H\"older on $\frac{3}{2}\bB'$, and since they are bounded on  $2\bB'$, we may pass to a subsequence so that $u_{j}\rightarrow u_{\infty}$ uniformly on $\frac{3}{2}\bB'$ and so that $y_{j}\rightarrow y$. Note that as $u_{j}\sim 1$ on $B_{j}$ by Harnack's inequality for all $j$, we have that $u_{j}\sim 1$ on $B(y,c/2)$, thus the set $P=\{x\in \frac{3}{2}\bB': u_{\infty}>0\}$ is nonempty. However, notice that by Cauchy estimates, the second derivatives of $u_{j}$ also converge uniformly after passing to a subsequence and they do so to the double derivative of $u_{\infty}$ on $P$. Thus, we must have $\grad^{2} u_{\infty}\equiv 0$ on $B(y,c/2)$. This implies $u_{\infty}$ is affine on the connected component $P'$ of $P$ containing $B(y,c/2)$, hence there is a half space $H$ so that $P'\cap \frac{3}{2}\bB'=H\cap \frac{3}{2}\bB'$. Let $H_{\ve/2}=\{x\in H: \dist(x,H^{c})=\ve/2\}$. 

Since $u_{j}\rightarrow u_{\infty}$ uniformly on $\bB'$ and $u_{\infty}\gec_{\ve}1$ by Harnack's inequality in $P'\cap \bB'$, we know that $u_{j}>0$ on $H_{\ve/2}\cap \bB'$ for $j$ large enough, that is, $\Omega_{j}\supseteq H_{\ve/2}\cap \bB'$.

Now suppose that for infinitely many $j$ there were $x_{j}\in \d H\cap \bB'$ so that $B(x_{j},\ve)\subseteq \Omega_{j}$. Passing to a subsequence, $x_j\rightarrow x\in \d H \cap \bB'$ and by Harnack's inequality, $B(x,\ve)\subseteq P'$, but this is impossible since $u_{\infty}(x)=0$ as $x\in \d H\cap \bB'$. Thus, for sufficiently large $j$, 
\[
\sup_{x\in \d H\cap \bB'}\dist(x,\d\Omega_{j})<\frac{\ve}{2}.
\]
Thus,
\[
\sup_{x\in \d H_{\ve}\cap\bB'}\dist(x,\d\Omega_{j})<\ve.
\]
The existence of $H_{\ve}$ now contradicts our assumptions. This proves the lemma. 

\end{proof}

Let $Q_{0}\in \cD$. Without loss of generality, $Q_{0}\in \cD_{0}$. Let $\Top_{1}$ and $\Tree_{1}(R)$ for $R\in \Top_{1}$ be as in Definition \ref{d:CDHM} with $\lambda = 2$ so that
\begin{equation}
\label{e:top1}
\sum_{R\in \Top_{1}} \ell(R)^{d} \leq 2 \cdot \CDHM(Q_0,2,A,\tau).
\end{equation}
Just as in the proof of Lemma \ref{l:H<CDHM}, let $k_0\in \bN$, $\cD(k_{0})=\{Q\subseteq Q_{0}: \ell(Q_{0})\geq \rho^{k_{0}}\ell(Q_{0})\}$, and let 
\[
\Top_{1}(k_{0})=\Top_{1} \cap \cD(k_{0}), \;\; \Tree_{1}^{k_{0}}(R) = \Tree_1(R)\cap \cD(k_{0}).
\]

%For $R\in \Top_{1}(k_{0})$, let $\Stop(R)$ be the minimal cubes in $\Tree(R)$. For $Q\in \Stop(R)$. 

By Definition \ref{d:CDHM}, for each $Q\in \Tree_{1}^{k_{0}}(R)$ and $R\in \Top_{1}(k_{0})$, there is a corkscrew point $x_{R}\in B_{R}$ so that 
\begin{equation}
\label{e:tree1}
\tau \Theta_{\omega^{x_{R}}}^{d}(2B_{R} )
\leq \Theta_{\omega^{x_{R}}}^{d}(2B_Q)
<A\Theta_{\omega^{x_{R}}}^{d}(2B_R).
\end{equation}
Let $\{\Tree_{2}(R):R\in \Top_{2}(k_{0})\}$ be the stopping-time regions from Lemma \ref{l:corona} for $C_{1},\vartheta^{-1}$ large enough constant we will pick later.  Now the sets

\[
\{\Tree_{1}^{k_{0}}(R_{1})\cap \Tree_{2}(R_{2}): R_{i}\in \Top_{i}(k_0)\}
\]
partition $\cD(k_0)$, and if $\Top(k_0)=\Top_{1}(k_{0})\cup \Top_2(k_{0})$, then each tree in the above collection has a top cube $R\in \Top(k_0)$, and we denote that tree $\Tree(R)$ (we drop the $k_0$ for convenience, but remember int also depends on $k_0$). In particular, \eqref{e:tree1} is still satisfied for $Q\in \Tree(R)$ and $R\in \Top(k_{0})$, and by Lemma \ref{l:corona}, there are Ahlfors regular sets $E_{R}$ satisfying the conclusion of Lemma \ref{l:corona} with respect to $\Tree(R)$.  By Lemma \ref{l:corona}, Lemma \ref{l:H<CDHM}, and Definition \ref{d:CDHM}, we have  for all $k_0\in\bN$ that
\begin{multline}
\label{e:sumtop12}
\sum_{R\in \Top(k_{0})}\ell(R)^{d}
 \leq \sum_{R\in \Top_{1}}\ell(R)^{d} + \sum_{R\in \Top_{2}(k_0)}\ell(R)^{d}   \\
\lec \CDHM(Q_0,2,A,\tau)+ \cH^{d}(Q_{0})
\lec \CDHM(Q_0,2,A,\tau).
\end{multline}

\begin{lemma}
For $Q\in \cD$, let $\cW$ be the Whitney cubes for $\Omega$ and
\[
U_{Q}^{K}=\{I\in \cW: I\cap KB_{Q}\neq\emptyset, \;\; \ell(I)\geq \ell(Q)/K\}
\]
and
\[
\Omega_{R}  = \bigcup_{Q\in \Tree(R)} U_{Q}^{K}.
\]
Then for $K$ large enough, $C_{2}\gg K$, and $\vartheta\ll K^{-1}$, $\d\Omega_{R}$ is Ahlfors regular. 
\end{lemma}

The proof is exactly the same as \cite[Proposition A.2]{HMM14}, as for $\vartheta$ small enough, $\Omega_{R}$ will be contained in 
\[
\ps{\bigcup_{I\in \cC_{R}}I}^{c},\]
which is a domain with Ahlfors regular boundary by Lemma \ref{l:corona}. We leave the details to the reader.

\begin{lemma}
\label{l:g<theta}
For $Q\in \Tree(R)$, there is $z_{Q}\in 4B_{Q}$ so that 
\begin{equation}
\label{e:g<theta}
\delta_{\Omega}(z_{Q})\sim_{A,\tau} \ell(Q) \;\; \mbox{ and }\frac{G_{\Omega}(x_{R},z_{Q})}{\ell(Q)}\sim_{A,\tau}\Theta_{\omega_{\Omega}^{x_{R}}}^{d}(2B_{R}) .
\end{equation}
\end{lemma}

\begin{proof}
Let $\phi_{Q}$ be a smooth bump function so that 
\[
\one_{2B_{Q}}\leq \phi_{Q}\leq \one_{4B_{Q}} \;\; \mbox{ and } |\grad^{2} \phi_{Q}|\lec \ell(Q)^{-2}.
\]
Then 
\[
\omega_{\Omega}^{x_{R}}(2B_{Q}) 
\leq \int \phi_{Q} d\omega_{\Omega}^{x_{R}}
\stackrel{\eqref{e:ibp}}{=} \int_{\Omega} G_{\Omega}(x_{R},x)\triangle \phi_{R} dx 
\lec \sup_{4B_{R}} G_{\Omega}(x_{R},\cdot) \ell(Q)^{d-1}.
\]
and so there is $z_{Q}\in 4B_{R}$ so that 
\begin{equation}
\label{e:tT<G/l}
\tau\Theta_{\omega_{\Omega}^{x_{R}}}^{d}(2B_{R})  \leq \Theta_{\omega_{\Omega}^{x_{R}}}^{d}(2B_{Q})
\lec \frac{G_{\Omega}(x_{R},z_Q) }{\ell(Q)}.
\end{equation}

\def\thomega{\Theta_{\omega_{\Omega}^{x_{R}}}^{d}}
If $\xi_Q\in\d\Omega $ is the closest point to $z_{Q}$, then $|\xi_Q-z_{Q}|\leq |\zeta_{Q}-z_{Q}|\leq 4\ell(Q)$, and so $z_{Q}\in B(\xi_{Q},4\ell(Q))\subseteq 8B_{Q}$. Thus,
\begin{align}
\frac{G_{\Omega}(x_{R},z_Q) }{\ell(Q)}
& \stackrel{\eqref{e:holder}}{\lec} \sup_{z\in 8B_{Q}} \ps{\frac{G_{\Omega}(x_{R},z) }{\ell(Q)}} \ps{\frac{|z_{Q}-\xi_{Q}|}{\ell(Q)}}^{\alpha} \notag \\ 
& \stackrel{\eqref{e:w>G}}{\lec} \ps{\frac{\delta_{\Omega}(z_{Q})}{\ell(Q)}}^{\alpha} \Theta_{\omega_{\Omega}^{x_{R}}}^{d}(32B_{Q})
\label{e:G/l<32}
\end{align}
We claim that for any $\lambda \geq 4$, 
\begin{equation}
\label{e:lambdaB<2B}
\thomega(\lambda B_{Q})\lec_{\lambda} A \thomega(2B_{R}).
\end{equation} 
Indeed, if $k\in \mathbb{N}$, $Q^{k}$ is the $k$th ancestor of $Q$ and $Q^{k}\in \Tree(R)$, then for $k$ large enough depending on $\lambda$,
\[
\thomega(\lambda B_{Q})
\lec_{\lambda} \thomega(2B_{Q^{k}})
\leq A\thomega(2B_{R}).
\]
Otherwise, if $R$ is the $j$th ancestor of $Q$ for some $j\leq k$, then 
\[
\thomega(\lambda B_{R})\lec_{\lambda}  \ell(R)^{-d}\stackrel{\eqref{e:bourgain}}{\sim} \thomega(2B_{R}).\]
%
%
%
%note that as $\rho<1/1000$ in Theorem \ref{t:Christ}, if $Q\neq R$, then its parent $Q^{1}$ is in $ \Tree(R)$ and $32B_Q\subseteq 2B_{Q^{1}}$, hence
%\[
%\thomega(32B_{Q})\lec \thomega(2B_{Q^{1}})\leq A\thomega(2B_{R}).\]
%Otherwise, if $Q=R$, then 
%\[
%\thomega(32B_{R})\lec  \ell(R)^{-d}\stackrel{\eqref{e:bourgain}}{\sim} \thomega(2B_{R})\]
and this proves the claim. Thus,
\begin{align*}
\tau\Theta_{\omega_{\Omega}^{x_{R}}}^{d}(2B_{R}) 
& \stackrel{\eqref{e:tT<G/l}}{\lec} \frac{G_{\Omega}(x_{R},z_Q) }{\ell(Q)}
\stackrel{\eqref{e:G/l<32}}{\lec} \ps{\frac{\delta_{\Omega}(z_{Q})}{\ell(Q)}}^{\alpha}  \thomega(32B_{Q})\\
& \stackrel{\eqref{e:lambdaB<2B}}{\lec} \ps{\frac{\delta_{\Omega}(z_{Q})}{\ell(Q)}}^{\alpha}  A \thomega(2B_{R}).
\end{align*}
Hence, we have 
\[
\ps{\frac{\tau}{A}}^{1/\alpha}\ell(Q) \lec \delta_{\Omega}(z_{Q}) \leq  4\ell(Q).\] 
Plugging this back into the above inequality also gives the rest of \eqref{e:g<theta}.

%Since $\Theta_{\omega_{\Omega}^{x_{R}}}^{d}(B_{Q})\sim \Theta_{\omega_{\Omega}^{x_{R}}}^{d}(4B_{Q})\sim \Theta_{\omega_{\Omega}^{x_{R}}}^{d}(4B_{R})$, we thus have 
%\begin{equation}
%\label{e:g<theta}
%\frac{G_{\Omega}(x_{R},z_Q) }{\ell(Q)}\sim \Theta_{\omega_{\Omega}^{x_{R}}}^{d}(4B_{R}).
%\end{equation}

\end{proof}

%Let $g_{R}(x) = G(x,x_{R})$. Since  $2|\grad^2 g_{R}|^2 = \triangle |\grad g_{R}|^{2}$, using Green's formula and the fact that $g_{R}$ is harmonic in $\Omega(R)\backslash B(x_{R},c\ell(R))$, and the Cauchy estimates
%\begin{equation}
%\label{e:cauchydelta}
%|\grad^2 g_{R}(x)|\delta_{\Omega}(x), |\grad g_{R}(x)|\lec \frac{g_{R}(x)}{\delta_{\Omega}(x)}
%\end{equation}
%we have 
%
%\begin{align*}
%\int_{\Omega_{R}} \av{\grad^2 g_{R}}^2 g_{R} dx 
%&  \sim   \int_{\Omega_{R}}  \triangle |\grad g_{R}|^{2} g_{R}dx \\
%& =   \int_{\d\Omega_{R}}\ps{g_{R}\frac{d|\grad g_{R}|^2}{d\nu}-|\grad g_{R}|^2 \frac{dg_{R}}{d\nu}}d\cH^{d}\\
%& \stackrel{\eqref{e:cauchydelta}}{\lec}
%  \int_{\d\Omega_R} \frac{g_{R}(x)^{3}}{\delta_{\Omega}(x)^{3}}dx\\
%  &  \stackrel{\eqref{e:g<theta}}{\lec} \Theta_{\omega^{x_{R}}}^{d}(R)^{3}  \ell(R)^{d}.
%\end{align*}

Let $g_{R}(x) = G_{\Omega}(x,x_{R})$. By \eqref{e:g<theta}, there is $c>0$ so that  if $B_{Q}'= B(z_{Q},c\ell(Q))$, then $2B_{Q}'\subseteq 5B_{Q}\cap \Omega$ and for $x\in B_{Q}'$, by Harnack's inequality
\begin{equation}
\label{e:g<r^-d}
\frac{g_{R}(x)}{\delta_{\Omega}(x)} \sim \frac{g_{R}(x)}{\ell(Q)}\sim  \Theta_{\omega^{x_{R}}}^{d}(2B_R) \sim \ell(R)^{-d}
\end{equation}
By adjusting the value of $c$ and the positions of the $z_Q$ and using Harnack's inequality, we can assume that $B_Q'\subseteq \Omega_{R}':=\Omega_{R}\backslash B(x_{R},c\ell(R))$ for all $Q\in \Tree(R)$ so that $R\neq Q$.

Moreover, for $\lambda$ large enough $\frac{\lambda}{4} B_{Q}\supseteq U_{Q}^{K}$, so for all $x\in U_{K}\backslash B(x_{R},c\ell(R))$,
\begin{equation}
\label{e:allxinU}
\frac{g_{R}(x)}{\delta_{\Omega}(x)} 
\stackrel{\eqref{e:w>G}}{\lec}
\Theta_{\omega^{x_{R}}}^{d}(\lambda B_{Q})
\stackrel{\eqref{e:lambdaB<2B}}{\lec}
\Theta_{\omega^{x_{R}}}^{d}(2 B_{R})
\stackrel{\eqref{e:bourgain}}{\lec} \ell(R)^{-d}.
\end{equation}

Since  $2|\grad^2 g_{R}|^2 = \triangle |\grad g_{R}|^{2}$, using Green's formula, the fact that $g_{R}$ is harmonic in $\Omega_{R}\backslash \{x_{R}\}$, and the Cauchy estimates
\begin{equation}
\label{e:cauchydelta}
|\grad^2 g_{R}(x)|\delta_{\Omega}(x), |\grad g_{R}(x)|\lec \frac{g_{R}(x)}{\delta_{\Omega}(x)} \;\; \mbox{ for all }x\in \Omega\backslash  B(x_{R},c\ell(R)).
\end{equation}
we have
\begin{align*}
\sum_{Q\in \Tree(R) \backslash \{R\} } \int_{B_{Q}'} \av{\frac{\grad^2 g_{R}}{g_{R}}}^2 & \delta_{\Omega}(x)^3dx 
\stackrel{\eqref{e:g<r^-d}}{\sim} \ell(R)^{3d}  \sum_{Q\in \Tree(R) \backslash \{R\}} \int_{B_{Q}'}  \av{\grad^2 g_{R}}^2 g_{R} dx \\
&  \lec  \ell(R)^{3d}  \int_{\Omega_{R}'}  \triangle |\grad g_{R}|^{2} g_{R}(x)dx  \\
& \sim  \ell(R)^{3d}  \int_{\d\Omega_{R}'}\ps{g_{R}\frac{d|\grad g_{R}|^2}{d\nu}-|\grad g_{R}|^2 \frac{dg_{R}}{d\nu}}d\cH^{d}\\
& \stackrel{\eqref{e:cauchydelta}}{\lec}
\ell(R)^{3d} \int_{\d\Omega_{R}'} \frac{g_{R}(x)^{3}}{\delta_{\Omega}(x)^{3}}d\cH^{d}(x)\\
  &  \stackrel{\eqref{e:allxinU}}{\lec}  \cH^{d}(\d\Omega_{R}' )
  \sim \ell(R)^{d}.
\end{align*}
\def\NA{\rm NA}
In particular,  If we let $\NA(R)$ (for "not affine") denote those cubes $Q\in\Tree(R)$ for which 
\[
{ \avint_{B_{Q}'} \av{\frac{\grad^2 g_{R}}{g_{R}}}^2  \delta_{\Omega}(x)^4dx }\geq \delta,
\]
Then
\begin{align*}
\sum_{Q\in \NA(R)}\ell(Q)^{d} 
& \leq  \ell(R)^{d}+\delta^{-1} \sum_{Q\in \Tree(R) \backslash \{R\}} \ps{ \avint_{B_{Q}'} \av{\frac{\grad^2 g_{R}}{g_{R}}}^2  \delta_{\Omega}(x)^4dx } \ell(Q)^{d}\\
&  \lec  \delta^{-1}\ell(R)^{d}.
\end{align*}

%Let $\cG$ denote those cubes $Q\in \cD(k_{0})$ so that there is a half space $H_{Q}$ so that $H\cap 4B_{Q}\subseteq \Omega$ and 
%\[
%\sup_{x\in \d H_{Q}\cap \frac{4}{\ve^{2}}B_{Q}}\dist(x,\d\Omega) <\ve 4\ell(Q)
%\]
%and let $\cB=\cD(k_{0})\backslash \cG$.  

By Lemma \ref{l:halfspace}, for $\delta$ small enough (and recalling that $2B_{Q}'\subseteq 5B_{Q}$)
\[
\cB^{\WHSA}_{\ve,5}(Q_0)  \cap \Tree(R)\subseteq \NA(R).
\]

Thus,
\begin{align*}
\sum_{Q\in \cB^{\WHSA}_{\ve,5}(Q_0) \cap \cD(k_0)}\ell(Q)^{d}
& \leq \sum_{R\in \Top(k_{0})} \sum_{Q\in \cB^{\WHSA}_{\ve,5}(Q_0)  \cap \Tree(R)} \ell(Q)^{d}\\
& \lec\sum_{R\in \Top(k_{0})} \sum_{Q\in \NA(R)}  \ell(Q)^{d} \\
&  \lec \delta^{-1} \sum_{R\in \Top(k_{0})} \ell(R)^{d} 
\stackrel{\eqref{e:sumtop12}}{\lec} \CDHM(Q_0,2,A,\tau)
\end{align*}
Letting $k_0\rightarrow\infty$ finally gives 
\[
\WHSA(Q_{0},\ve,5)
=\sum_{Q\in \cB^{\WHSA}_{\ve,5}(Q_0) }\ell(Q)^{d}
\lec \CDHM(Q_0,2,A,\tau).
\]

This finishes the proof of Lemma \ref{l:WHSA<CDHM}. 

\subsection{The semi-uniform case}

Toward showing the other direction of Theorem \ref{t:SU-version}, we will show the following.

  \begin{lemma}
 \label{l:beta<CDHM-SU}
Suppose $\Omega$ is SU and $\{\Tree_1(R):R\in \Top_1\}$ are trees as in the statement of Theorem \ref{t:SU-version} with respect to $\omega=\omega_{\Omega}^{x_0}$ with $x_0\in \Omega\backslash MB_{Q_{0}}$ and $Q_0\in \cD$.  For $\lambda\geq 2$ and for $A,\tau^{-1}>0$ large enough (depending on $\lambda$), we have that for all $Q_0\in \cD$,
 \[
\beta_{\d\Omega}(Q_0)\lec_{A,\tau,\lambda } \sum_{R\in \Top_{1}}\ell(R)^{d}.
 \]
 \end{lemma}
 
 We will require the following lemmas.
 
  \begin{lemma}
 \label{l:WHSA<CDHM-SU}
With the assumptions of the previous lemma, and for all $\ve^{-1},A,\tau^{-1}>0$  large enough, we have that for all $R\in \cD$,
 \[
 \WHSA(R,5,\ve) \lec_{K_0,\ve,A,\tau} \CDHM(R,2,A,\tau).
 \]
 \end{lemma}
 
 This has the same proof as Lemma \ref{l:WHSA<CDHM}, so we omit it.

\begin{lemma}
Suppose $\Omega$ is also SU and $\{\Tree_1(R):R\in \Top_1\}$ are trees as in the statement of Lemma \ref{t:SU-version} with respect to $\omega=\omega_{\Omega}^{x_0}$ with $x_0\in \Omega\backslash MB_{Q_{0}}$ and $Q_0\in \cD$. Then for $Q_0,\tau,A,\ve$ as above,
\[
\WHSA(Q_0,\ve,5)\lec \sum_{R\in \Top_1} \ell(R)^{d}.
\]
\end{lemma}

\begin{proof}
We sketch the changes needed in the above proof of Lemma \ref{l:WHSA<CDHM}. Note that if $\{\Tree_1(R):R\in \Top_1\}$ are as in Lemma \ref{t:SU-version}, since harmonic measure is doubling, we can actually assume $\lambda =2$ (this changes the constants $A$ and $\tau$ by a constant multiple depending on the doubling constant).

Now define $\Top_{1}(k_0)$, $\Tree_{1}^{k_0}$, $\Top_{2}$, $\Tree_{2}$, $\Tree$ and $\Top(k_0)$ as in the proof of Lemma \ref{l:WHSA<CDHM}. We now have 
\begin{multline*}
\sum_{R\in \Top(k_0)}\ell(R)^{d}
\leq \sum_{R\in \Top_1}\ell(R)^{d}
+\sum_{R\in \Top_2}\ell(R)^{d}\\
\lec \sum_{R\in \Top_1}\ell(R)^{d} + \cH^{d}(Q_0)
\lec \sum_{R\in \Top_1}\ell(R)^{d} 
\end{multline*}
where the last estimate follows from the proof of Lemma \ref{l:H<CDHM} (where we just replace $\omega_{\Omega}^{x_{R}}$ by $\omega$ everywhere).

We now need a version of Lemma \ref{l:g<theta}:

\begin{lemma}
\label{l:g<thetaSU}
For $R\in \Top$, let $\{x_{R}^{i}\}_{i=1}^{N}$ be reference points for $2B_{R}$ and $g = G_{\Omega}(x_0,\cdot)$. For $Q\in \Tree(R)$, there is $z_{Q}\in 4B_{Q}$ so that 
\begin{equation}
\label{e:g<thetaSU}
\delta_{\Omega}(z_{Q})\sim_{A,\tau} \ell(Q) \;\; \mbox{ and }\sum\frac{g(z_{Q})}{\ell(Q)}\sim_{A,\tau}\Theta_{\omega}^{d}(2B_{R}) .
\end{equation}
\end{lemma}

\begin{proof}
The proof is exactly the same as Lemma \ref{l:g<theta}, except that now to prove \eqref{e:lambdaB<2B}, we just use \eqref{doubling} since $\Omega$ is SU. 
\end{proof}

Again, there is $c>0$ so that  if $B_{Q}'= B(z_{Q},c\ell(Q))$, then $2B_{Q}'\subseteq 5B_{Q}\cap \Omega$ and for $x\in B_{Q}'$, by Harnack's inequality
\begin{equation}
\label{e:g<r^-d-SU}
\frac{g(x)}{\delta_{\Omega}(x)} \sim \frac{g(x)}{\ell(Q)}\sim  \Theta_{\omega}^{d}(2B_R)
\end{equation}
By adjusting the value of $c$ and the positions of the $z_Q$, we can assume that $B_Q'\subseteq \Omega_{R}':=\Omega_{R}\backslash B(x_{R},c\ell(R))$ for all $Q\in \Tree(R)$ so that $R\neq Q$.

Again, for $\lambda$ large enough $\frac{\lambda}{4} B_{Q}\supseteq U_{Q}^{K}$, so for all $x\in U_{K}\backslash B(x_{R},c\ell(R))$,
\begin{equation}
\label{e:allxinU-SU}
\frac{g(x)}{\delta_{\Omega}(x)} 
\stackrel{\eqref{e:w>G}}{\lec}
\Theta_{\omega}^{d}(\lambda B_{Q})
\stackrel{\eqref{doubling}}{\lec}
\Theta_{\omega}^{d}(2 B_{R}).
\end{equation}

Also note that \eqref{e:cauchydelta} still holds with $g$ in place of $g_{R}$. Repeating the same estimates as before, we get 
\begin{align*}
\sum_{Q\in \Tree(R) \backslash \{R\} } \int_{B_{Q}'} \av{\frac{\grad^2 g}{g}}^2 & \delta_{\Omega}(x)^3dx 
\stackrel{\eqref{e:g<r^-d-SU}}{\sim} \Theta_{\omega}^{d}(2B_R)^{-3}\sum_{Q\in \Tree(R) \backslash \{R\}} \int_{B_{Q}'}  \av{\grad^2 g}^2 g dx \\
& 
\lec \Theta_{\omega}^{d}(2B_R)^{-3} \int_{\d\Omega_{R}'} \frac{g(x)^{3}}{\delta_{\Omega}(x)^{3}}d\cH^{d}(x)\\
  &  \stackrel{\eqref{e:allxinU-SU}}{\lec} \Theta_{\omega}^{d}(2B_R)^{-3} \Theta_{\omega}^{d}(2B_R)^{3} \cH^{d}(\d\Omega_{R}' )
  \sim \ell(R)^{d}.
\end{align*}

 The remaining steps are just as in Lemma \ref{l:g<theta}.

\end{proof}

Now Lemma \ref{l:beta<CDHM-SU} follows from the previous two lemmas and Lemma \ref{l:beta<WHSA}.

\subsection{Conclusion of proofs of Theorems \ref{t:CDHM} and \ref{t:SU-version}}

We finally remark that Theorem \ref{t:CDHM} follows from Lemma \ref{l:CDHM<beta} (with $\cT$ equal to all cubes contained in $Q_0$) and \ref{l:beta<CDHM}. Theorem \ref{t:SU-version} similarly follows from Corollary \ref{c:SU-cor}  (with $\cT$ equal to all cubes contained in $Q_0$) and \ref{l:beta<CDHM-SU}.

\section{Proof of Theorem \ref{t:Afin}}
\label{s:III}

Theorem \ref{t:Afin} will follow from the following slightly more general result:

\begin{theorem}
\label{t:Afin2}
Let $\Omega\subseteq \R^{d+1}$ be a semi-uniform domain with Ahlfors regular boundary and let $\cD$ be the Christ-David cubes. There is $M>0$ depending on the semi-uniformity and Ahlfors regularity constants so that the following holds. For $Q_0\in \cD$, let $x_0\in  \Omega\backslash MB_{Q_{0}}$ and $\omega=\omega_{\Omega}^{x_0}$. Let $k$ be the Radon-Nikodym derivative of $\omega$ in $Q_0$ with respect to $\sigma = \cH^{d}|_{\d\Omega}$. Then there is $0<C\lec 1$ so that
\[
C+ \avint_{Q_{0}}  \log \frac{1}{k}  d\sigma+\log\frac{\omega(Q_0)}{|Q_0|} \sim \frac{\beta_{\d\Omega}(Q_0)}{\ell(Q_{0})^{d}}.
\]
Above, the constants only depend on the semi-uniformity and Ahlfors regularity.
\end{theorem}

To see how this implies Theorem \ref{t:Afin}, observe that if $\omega \ll \cH^{d}$ in $Q_0$, then $\omega|_{Q_{0}} = kd\sigma|_{Q_{0}}$. Hence, $\omega(Q_{0})/|Q_0|=\avint_{Q_{0}} kd\sigma$, and so Theorem \ref{t:Afin2} impleis 
\[
\frac{\beta_{\d\Omega}(Q_0)}{\ell(Q_{0})^{d}}
\sim 
C+ \avint_{Q_{0}}  \log \frac{1}{k}  d\sigma +\log\avint_{Q_{0}} kd\sigma.
\]
Note that by Jensen's inequality,
\[
\avint_{Q_{0}}  \log \frac{1}{k}  d\sigma +\log\avint_{Q_{0}} kd\sigma
\geq 
\avint_{Q_{0}}  \log \frac{1}{k}  d\sigma +\avint_{Q_{0}} \log kd\sigma = \avint_{Q_{0}} \log 1d\sigma 
= 0.\]
Thus,
\begin{align*}
C& \leq C+\avint_{Q_{0}}  \log \frac{1}{k}  d\sigma +\log\avint_{Q_{0}} kd\sigma\\
& \leq C\ps{1+\avint_{Q_{0}}  \log \frac{1}{k}  d\sigma +\log\avint_{Q_{0}} kd\sigma}\\
& \lec 1+\avint_{Q_{0}}  \log \frac{1}{k}  d\sigma +\log\avint_{Q_{0}} kd\sigma.
\end{align*}

This proves Theorem \ref{t:Afin}.\\

The rest of this section is dedicated to the proof of Theorem \ref{t:Afin2}, some of the ideas for which come from the martingale arguments used to study $A_{\infty}$-weights in \cite[Section 3.18]{FKP91}. The proof will follow from the main lemmas in the following two sections.

\subsection{The case assuming $\beta_{\d\Omega}(Q_0)<\infty$ a priori}

\begin{lemma}
The conclusions of Theorem \ref{t:Afin2} hold assuming  $\beta_{\d\Omega}(Q_0)<\infty$. 
\end{lemma}

\begin{proof}

Assume $\beta_{\d\Omega}(Q_0)<\infty$. Let $\Top$ and $\Tree(R)$ be the cubes and trees from Theorem \ref{t:SU-version} and $\omega=\omega_{\Omega}^{x_{0}}$ for some $x_{0}\in \Omega\backslash MB_{Q_{0}}$. 

For $R\in \Top$, let $\Next(R)$ be the children of the cubes in $\Stop(R)$. By Theorem \ref{t:SU-version} (3) and \eqref{doubling}, we know that  for $Q\in \Next(R)$, either 
\begin{enumerate}[(a)]
\item $\Theta_{\omega}^{d}(Q)\sim A \Theta_{\omega}^{d}(R)$, or 
\item $\Theta_{\omega}^{d}(Q)\sim \tau \Theta_{\omega}^{d}(R)$.
\end{enumerate}

Let $\HD(R)$ and $\LD(R)$ denote the cubes from the first and second alternatives respectively.

For $Q\subseteq \d\Omega$, let $|Q|=\cH^{d}(Q)$ and
\[
\theta^{d}_{\omega}(Q) = \frac{\omega(Q)}{|Q|}\sim \Theta_{\omega}^{d}(Q)
\]
where the last comparison follows from Ahlfors regularity of $\d\Omega$. Thus, by the doubling property for $\omega$, there is a constant $C_0$ so that 
\begin{equation}
\label{e:thetaHD}
C_0^{-1} A \theta^{d}_{\omega}(R) < \theta^{d}_{\omega}(Q) \leq C_0 A \theta^{d}_{\omega}(R) \;\; \mbox{ for all }Q\in \HD(R)
\end{equation}
and 
\begin{equation}
\label{e:thetaLD}
C_0^{-1} \tau\theta^{d}_{\omega}(R)  < \theta^{d}_{\omega}(Q) \leq C_0 \tau  \theta^{d}_{\omega}(R) \;\; \mbox{ for all }Q\in \LD(R)
\end{equation}

It can be shown using a similar proof to those of \cite[Lemmas 2.12 and 2.17]{Mattila} that there is a Borel function $k$ finite $\cH^{d}$-a.e. so that 
\[
\omega = k\cH^{d}+\omega_{s}\]
where $\omega_{s}\perp\omega$ and
\begin{equation}
\label{e:rd}
\lim_{Q\downarrow \{x\}} \theta^{d}_{\omega}(Q)  = k \;\; \mbox{ for $\cH^{d}$-a.e. $x\in Q_0$,}
\end{equation}

Let
\[
F(R) = R\backslash \bigcup_{Q\in  \Next(R)}Q
\]
and define

\[
f_{R} = \sum_{Q\in \Next(R)}\log \frac{\theta_{\omega}(Q)}{\theta_{\omega}(R)}\one_{Q} +\frac{k}{\Theta_{\omega}(R)}\one_{F_{R}}.
\]

\begin{lemma}
For all $x\in R$,
\begin{equation}
\label{e:f_R<}
-|\log \tau/C_0|= \log (\tau/C_0)\leq f_{R}(x) \leq \log (C_2A).
\end{equation}
\end{lemma}

\begin{proof}
Recall from Lemma \ref{l:ds-cover} that every $x\in R$ is either contained in infinitely many cubes from $\Tree(R)$ or is contained in a cube from $\Stop(R)$, and since $\Next(R)$ are the children of the cubes in $\Stop(R)$, this means $x$ is contained in a cube from $\HD(R)$ or $\LD(R)$. Thus, if $x\in F(R)$, it is contained in infinitely many cubes $Q\in \Tree(R)$, and so \eqref{e:rd} and (1) from Theorem \ref{t:SU-version} imply \eqref{e:f_R<}. 

Now if $x\in Q\in \Next(R)$ and $Q'\in \Stop(R)$ is the parent of $Q$, then \eqref{e:f_R<} follows from \eqref{e:thetaHD} and \eqref{e:thetaLD}.

\end{proof}

\begin{lemma}
The sets $\{F(R):R\in \Top\}$ are disjoint. In particular,
\begin{equation}
\label{e:frsum}
\sum_{R\in\Top}|F(R)|\leq |Q_{0}|.
\end{equation}
\end{lemma}

\begin{proof}
To see this, suppose there is $x\in F(R)\cap F(R')$ with $R\neq R'$. Then $x\in R\cap R'$, without loss of generality we can assume $R\supseteq R'$, so $R'\subseteq Q\in \Next(R)$, but since $x\in F(R)$, $x$ is not contained in any cube from  $ \Next(R)$, which is a contradiction. 
\end{proof}

Note that $\cH^{d}$-a.e. $x\in Q_0$ is contained in $F_{R}$ for some $R\in \Top$ by (3) of Theorem \ref{t:SU-version} and the Ahlfors regularity of $\d\Omega$. Thus,
\begin{equation}
\label{e:logk-logt = sum f}
\log k-\log \theta_{\omega}^{d}(Q_0)=\sum_{R\in \Top} f_{R} \;\;\;\; \mbox{$\cH^{d}$-a.e in $Q_{0}$}.
\end{equation}

\begin{lemma}
\[
\int_{Q_{0}} \ps{\log k-\log \theta_{\omega}^{d}(Q_0)}
=\sum_{R\in \Top}\int_{Q_{0}} f_{R}.
\]
\end{lemma}

\begin{proof}

We claim that $\log k-\log \theta_{\omega}(Q_0)$ is absolutely integrable. Indeed, by \eqref{e:logk-logt = sum f},
\[
|\log k-\log \theta_{\omega}^{d}(Q_0)|
\leq \sum_{R\in \Top} |f_{R}|\]
and so
\[
\int_{Q_{0}} |\log k-\log \theta_{\omega}^{d}(Q_0)|
\leq \sum_{R\in \Top}\int_{Q_{0}} |f_{R}|
\stackrel{\eqref{e:f_R<}}{\lec}_{\tau,A} \sum_{R\in \Top} |R|\lec \beta_{\d\Omega}(Q_{0}),\]
and this proves the claim.

Now we prove the lemma. Let $\Top_{j}$ be those cubes in $\Top$ that are properly contained in $j$ many cubes from $\Top$, $\Top^{N}=\bigcup_{j=0}^{N}\Top_{j}$  and let 
\[
f_{N} = \sum_{R\in \Top^N}f_{R}.
\]
Then
\[
|f_{N}|
\leq  \sum_{R\in \Top} |f_{R}|\lec_{\tau,A} \sum_{R\in \Top}\one_{R}
\]
and the last sum is a nonegative integrable function, thus the dominated convergence theorem implies
\begin{align*}
\int_{Q_{0}} \ps{\log k-\log \theta_{\omega}^{d}(Q_0)}
& =\int_{Q_{0}} \sum_{R\in \Top}f_{R}
 =\int_{Q_{0}} \lim_{N} f_{N} 
  =\lim_{N} \int_{Q_{0}} f_{N}\\
& =\lim_{N} \int_{Q_{0}} \sum_{R\in \Top^N}f_{R}
 =\lim_{N}\sum_{R\in \Top^{N}}\int_{Q_{0}} f_{R} 
  = \sum_{R\in \Top} \int_{Q_{0}} f_{R}.
\end{align*}

This proves the lemma.

\end{proof}

Thus,

%
%In particular, 
%\[
%\int_{Q_{0}} \av{\log k- \log \theta_{\omega}Q_{0}
%

\begin{equation}
\label{e:loglower}
\int_{Q_{0}} (\log k-\log \theta_{\omega}^{d}(Q_{0}))
=\sum_{R\in \Top} \int_{Q_{0}}  f_{R}
\stackrel{\eqref{e:f_R<}}{\geq}  -|\log(C_0^{-1}\tau)|\sum_{R\in \Top} |R|
\end{equation}

Now we will prove an opposite inequality. Note that
\begin{equation}
\label{e:A-1|R|}
\sum_{Q\in \HD(R)} |Q|
\stackrel{\eqref{e:thetaHD}}{\sim} A^{-1} \theta_{\omega}^{d}(R)^{-1} \sum_{Q\in \HD(R)} \omega(Q)
\leq A^{-1} \theta_{\omega}^{d}(R)^{-1} \omega(R)
=A^{-1} |R|.
\end{equation}
And so  for some constant $c$ depending on the Ahlfors regularity,
\begin{equation}
\label{e:HDint}
\int_{\bigcup_{Q\in \HD(R)}Q} f_{R} 
\stackrel{\eqref{e:f_R<}}{\leq} \sum_{Q\in \HD(R)} \log (C_0A) |Q|
\stackrel{\eqref{e:A-1|R|}}{\leq} c\frac{\log (C_0A)}{A} |R|
\end{equation}
We will abuse notation here and also write $\LD(R)=\bigcup_{Q\in \LD(R)}Q$. Then 
\begin{equation}
\label{e:LDint}
\int_{\bigcup_{Q\in \LD(R)}Q} f_{R} 
\stackrel{\eqref{e:thetaLD}}{\leq} \sum_{Q\in \LD(R)} \log (C_0\tau ) |Q|
=\log (C_0\tau) |\LD(R)|.
\end{equation}
We will pick $\tau>0$ small so that 
\begin{equation}
\label{e:log/8<1}
\frac{\log (C_0\tau)}{8}<-1.
\end{equation}
Let 
\[
\Top_{1}=\{R\in \Top:|\LD(R)|\geq |R|/4\}, \;\; \Top_{2}=\Top\backslash \Top_{1}.
\]

Also note that if $R\in \Top_{1}$, then (recalling $\log (C_0\tau)<0$ and $|\LD(R)|\geq |R|/4$)
\begin{equation}
\label{e:LDint-top1}
\int_{\bigcup_{Q\in \LD(R)}Q} f_{R} 
\stackrel{\eqref{e:LDint}}{\leq} \log (C_0\tau) |\LD(R)|
\leq  \frac{\log (C_0\tau)}{4}  |R|
\end{equation}
and so (recalling \eqref{e:log/8<1}) if we pick $A$ large enough so that $c\frac{\log (C_0A)}{A} < -\frac{ \log (C_0\tau)}{8}$,
\begin{align}
 \sum_{R\in \Top_{1}} &  \ps{\int_{\bigcup_{Q\in \HD(R)} Q}+\int_{\bigcup_{Q\in \LD(R)}Q}+\int_{F(R)}} f_{R}\notag \\
& \stackrel{\eqref{e:HDint}, \eqref{e:LDint-top1} \atop \eqref{e:f_R<},\eqref{e:frsum}}{\leq} \sum_{R\in \Top_{1}}\ps{c\frac{\log(C_0 A) }{A}+\frac{\log(C_0\tau)}{4}} |R| + \log(C_0A)|Q_{0}| \notag \\
& \leq \sum_{R\in \Top_{1}} \frac{\log(C_0\tau)}{8} |R| + \log(C_0A)|Q_{0}|
\stackrel{\eqref{e:log/8<1}}{<} -\sum_{R\in \Top_{1}}|R| + \log(C_0A)|Q_{0}|
\label{e:top1}
\end{align}

Now if $R\in \Top_{2}$,
\[
|R|=|F(R)\cup LD(R)\cup \HD(R)|
\stackrel{\eqref{e:A-1|R|}}{\leq} |F(R)|+\ps{CA^{-1}+\frac{1}{4}}|R|
\]
and so for $A>4C$,
\begin{equation}
\label{e:R<2FR}
|R|<2|F(R)|.
\end{equation}
Thus,
\begin{equation}
\label{e:rintop2}
\sum_{R\in \Top_{2}}|R|
\stackrel{\eqref{e:R<2FR}}{\leq} 2\sum_{R\in \Top_{2}} |F_{R}|
\stackrel{\eqref{e:frsum}}{\leq}2|Q_{0}|.
\end{equation}
Hence,
\begin{align}
 \sum_{R\in \Top_{2}} &  \ps{\int_{\bigcup_{Q\in \HD(R)} Q}+\int_{\bigcup_{Q\in \LD(R)}Q}+\int_{F(R)}} f_{R} \notag \\
& \stackrel{\eqref{e:HDint}, \eqref{e:LDint} \atop \eqref{e:f_R<},\eqref{e:frsum}}{\leq} \sum_{R\in \Top_{2}}\ps{c\frac{\log (C_0A)}{A} |R|+\underbrace{\log(C_0\tau) |\LD(R)|}_{\leq 0}} + \log(C_0A)|Q_{0}| \notag \\
& \leq \sum_{R\in \Top_{2}}{\frac{c\log (C_0A)}{A} |R|}+ \log(C_0A)|Q_{0}|  \notag \\
& \stackrel{\eqref{e:rintop2}}{ \leq} \ps{\frac{2c\log (C_0A)}{A}+  \log(C_0A)} |Q_{0}|
=: C_{A}|Q_{0}|.
\label{e:top2}
\end{align}

Thus, noting that $2+\log (C_0A)<2C_{A}$ for $A$ large,
\begin{align*}
\int_{Q_{0}} & \ps{\log k - \log \theta_{\omega}^{d}(Q_{0})}
  = \sum_{R\in \Top} \int_{Q_{0}} f_{R}
   = \ps{\sum_{R\in \Top_{1}}+\sum_{R\in \Top_{2}}}f_{R}\\
&   \stackrel{\eqref{e:top1} \atop\eqref{e:top2} }{ \leq} 
-\sum_{R\in \Top_{1}}  |R| + \log(C_0A)|Q_{0}|  + C_{A}|Q_{0}|\\
&   \leq - \sum_{R\in \Top}  |R|+\sum_{R\in \Top_{2}}|R|+ (\log (C_0A)+C_{A})|Q_{0}| + \\
  & 
 \stackrel{\eqref{e:rintop2}}{ \leq} - \sum_{R\in \Top}  |R|+(2+\log (C_0A)+C_{A})|Q_{0}| < - \sum_{R\in \Top}  |R|+3C_{A}|Q_{0}| 
  \end{align*}

%
%Thus, by \eqref{e:rd} and the definition of $F(R)$, 
%\[
%\sum_{R\in \Top(k_0)} \int_{F(R)} f_{R} \leq \sum_{R\in \Top(k_0)} \log (CA)|F(R)| \leq \log(CA)|Q_{0}|.
%\]
%Hence, if we pick $\tau>0$ small enough so that $\log C\tau<-1$ and $A$ large enough so that $\frac{\log (CA)}{A} < -\frac{\log (C\tau)}{2}$, then 
%\begin{align*}
%\int_{Q_{0}} & \ps{\log k_{k_{0}} - \log \Theta_{\omega}^{d}(Q_{0})}
%  = \sum_{R\in \Top} \int_{Q_{0}} f_{R}\\
%& \leq \sum_{R\in \Top} \ps{\int_{\bigcup_{Q\in \HD(R)} Q}+\int_{\bigcup_{Q\in \LD(R)}Q}+\int_{F(R)}} f_{R}\\
%& \lec \sum_{R\in \Top}\ps{\frac{\log(C A) }{A}+\log(C\tau)}|R| + \log(CA)|Q_{0}|\\
%& \lec \sum_{R\in \Top} \log(C\tau)|R| + \log(CA)|Q_{0}|
%<- \sum_{R\in \Top}  |R|+ |Q_{0}|
%\end{align*}
%
%

Combining this with \eqref{e:loglower}, we get 

\[
-\sum_{R\in \Top} |R|- |Q_{0}|\lec \int_{Q_{0}}  \ps{\log k- \log \theta_{\omega}^{d}(Q_{0})} \leq - \sum_{R\in \Top} |R| + 3C_{A}|Q_{0}|.
\]

Thus, if we subtract $C|Q_{0}|$ from both sides for large enough $C$, we get 
\[
\int_{Q_{0}}  \ps{\log k - \log \theta_{\omega}^{d}(Q_{0})-C}
\sim - \sum_{R\in \Top} |R|^{d}  -|Q_{0}|
.
\]
Now taking negatives of both sides (and recalling $|R|\sim \ell(R)^{d}$) gives 
\[
\avint_{Q_{0}}  \ps{\log \frac{1}{k} +  \log \theta_{\omega}^{d}(Q_{0}) +C}
\\
\sim  \sum_{R\in \Top} \frac{|R|}{|Q_{0}|}+ 1
\stackrel{\eqref{e:SU-top-sum}}{\sim} \frac{\beta_{\d\Omega}(Q_{0})}{|Q_{0}|}.
\]

This proves Theorem \ref{t:Afin2} under the assumption that $\beta_{\d\Omega}(Q_{0})<\infty$.

\end{proof}

\subsection{The case assuming the log integral is finite a priori}

\begin{lemma}
The conclusions of Theorem \ref{t:Afin2} hold assuming
\[
\avint_{Q_{0}}  \log \frac{1}{k}  d\cH^{d} +\log\frac{\omega(Q_0)}{|Q_{0}|}<\infty.
\]
\end{lemma}

\begin{proof}
Under these assumptions, we'll show that we still have $\beta_{\d\Omega}(Q_{0})<\infty$, and we can then employ the previous lemma. To do this, it suffices to find a new collection $\{\Tree(R):R\in \Top\}$ satisfying the conditions of Theorem \ref{t:SU-version}. 

Recall that $k<\infty$ $\cH^{d}$-a.e. so \eqref{e:rd} still holds. Run a stopping-time algorithm as follows. Let $C>1$. For $R\subseteq Q_{0}$, let $\HD'(R)$ be those maximal cubes in $R$ which have a child $Q'$ for which $\theta_{\omega}^{d}(Q')>CA\theta_{\omega}^{d}(R)$, and $\LD'(R)$ be those maximal $Q\subseteq R$ with a child $Q'$ so that $\theta_{\omega}^{d}(Q')<C^{-1}\tau \theta_{\omega}^{d}(R)$, let $\Stop(R)=\HD'(R)\cup \LD'(R)$. Let $\HD(R)$ be the children of the cubes in $\HD'(R)$ and $\LD(R)$ be the children of the cubes in $\LD'(R)$ and set $\Next(R)=\HD(R)\cup \LD(R)$. For $C$ large enough, by the doubling property, we have that $\theta_{\omega}(Q)>A\theta_{\omega}(R)$ for $Q\in \HD(R)$ and $\theta_{\omega}(Q)<\tau \theta_{\omega}(R)$ for $Q\in \LD(R)$. By maximality, we also have there is $C_0$ so that $\theta_{\omega}(Q)<C_0A\theta_{\omega}(R)$ for $Q\in \HD(R)$ and $\theta_{\omega}(Q)\geq C_0^{-1}\tau \theta_{\omega}(R)$ for $Q\in \LD(R)$.

Let $\Top_{0}=\{Q_{0}\}$ and for $k\geq 0$, let 
\[
\Top_{k+1}=\bigcup_{R\in \Top_{k}} \Next(R), \;\;\; \Top = \bigcup_{k\geq 0} \Top_{k}.
\]

By Lemma \ref{t:SU-version},
\[
\beta_{\d\Omega}(Q_0)
\lec \sum_{R\in \Top}|R|
\]
and so the lemma will follow once we show
\begin{equation}
\label{e:sum<log1/k}
\sum_{R\in \Top}|R|\lec 
\avint_{Q_{0}}  \log \frac{1}{k}  d\cH^{d} +\log\frac{\omega(Q_0)}{|Q_{0}|}<\infty.
\end{equation}

First observe that for $R\in \Top$, 
\begin{equation}
\label{e:HD<A-1-small-theta}
\sum_{Q\in \HD(R)}|Q|
< \sum_{Q\in \HD(R)} \theta_{\omega}^{d}(R)^{-1} A^{-1} \omega(Q)
\leq A^{-1}|R|.
\end{equation}
Let 
\[
G_{R} = R\backslash \bigcup_{Q\in \Next(R)} Q, \;\;\; 
G_{N} = \bigcup_{k=0}^{N} \bigcup_{R\in \Top_{k}}G_{R}
\]
and
\[
g_{R} = \sum_{Q\in \Next(R)} \log\frac{\theta_{\omega}(R)}{\theta_{\omega}(Q)} + \log \frac{\theta_{\omega}(R)}{k}\one_{G_{R}}.
\]
Note that if $x\in G_{R}$, then
\begin{equation}
\label{e:k<At}
k(x) \leq  CA \theta_{\omega}(R).
\end{equation}

Since $\theta_{\omega}(Q)<\tau \theta_{\omega}(R)$ for $Q\in \LD(R)$ and $\theta_{\omega}(Q)<C_0A\theta_{\omega}(R)$ for $Q\in \HD(R)$,
\begin{align*}
\log\tau^{-1}\sum_{Q\in \LD(R)}|Q|
& \leq \sum_{Q\in \LD(R)} |Q|\log\frac{\theta_{\omega}(R)}{\theta_{\omega}(Q)}\\
& \leq \sum_{Q\in \Next(R)}|Q|\log\frac{\theta_{\omega}(R)}{\theta_{\omega}(Q)}
+  \sum_{Q\in \HD(R)} |Q| \log\frac{\theta_{\omega}(Q)}{\theta_{\omega}(R)}\\
& \leq   \sum_{Q\in \Next(R)}|Q|\log\frac{\theta_{\omega}(R)}{\theta_{\omega}(Q)}+ \log C_0A\sum_{Q\in \HD(R)}|Q|\\
& \leq  \sum_{Q\in \Next(R)}|Q|\log\frac{\theta_{\omega}(R)}{\theta_{\omega}(Q)}+ \frac{\log C_0A}{A} |R|\\
& = \int_{R} g_{R} +\int_{G_{R}} \log  \frac{k}{\theta_{\omega}(R)} + \frac{\log A}{A} |R|\\ 
& \stackrel{\eqref{e:k<At}}{\leq} \int_{R} g_{R}  +\log CA |G_{R}| + \frac{\log A}{A}|R|.
\end{align*}

In particular, if $\kappa = (\log\tau^{-1})^{-1}$, for $k> 0$, if we pick $A$ large so that $A>4$ and $\kappa \frac{\log C_0A}{ A}<\frac{1}{4}$, 
\begin{align*}
\sum_{R\in \Top_{k}}|R|
& =\sum_{R\in \Top_{k-1}}\sum_{Q\in \Next(R)=\HD(R)\cup \LD(R)}|Q|\\
& \stackrel{\eqref{e:HD<A-1-small-theta}}{\leq} \sum_{R\in \Top_{k-1}} \ps{A^{-1}|R|+  \kappa \ps{\int_{R} g_{R}  +\log CA |G_{R}| + \frac{\log C_0A}{A}|R|}}\\
& <  \sum_{R\in \Top_{k-1}} \ps{\frac{|R|}{2}+  \kappa  \ps{\int_{R} g_{R}  +\log CA |G_{R}|}}\\
\end{align*}

Thus,

\begin{align*}
\sum_{k=0}^{N} \sum_{R\in \Top_{k}}|R|
& = |Q_{0}|+\sum_{k=1}^{N} \sum_{R\in \Top_{k}}|R|\\
& 
\leq |Q_{0}| +\sum_{k=1}^{N}\sum_{R\in \Top_{k-1}} \ps{\frac{|R|}{2}+  \kappa \ps{\int_{R} g_{R}  +\log CA |G_{R}|}}\\
\end{align*}
And so 

\[
\frac{1}{2} \sum_{k=0}^{N} \sum_{R\in \Top_{k}}|R|
\leq |Q_{0}| +\sum_{k=0}^{N-1}\sum_{R\in \Top_{k}}  \kappa \ps{\int_{R} g_{R}  +\log CA |G_{R}|}
\]
As before with the $F_{R}$, the $G_{R}$ are disjoint and so
\[
\sum_{k=0}^{N-1}\sum_{R\in \Top_{k}} \kappa  \log CA |G_{R}|\leq \kappa\log CA |Q_{0}|.\]
So to prove \eqref{e:sum<log1/k}, we just need to show
\[
\sum_{k=0}^{N-1}\sum_{R\in \Top_{k}} 
\int_{R} g_{R}\lec \avint_{Q_{0}}  \log \frac{1}{k}  d\cH^{d} +\log\frac{\omega(Q_0)}{|Q_{0}|}.
\]
Note that 
\[
Q_0 = G_{k-1}\sqcup \bigsqcup_{R\in \Top_{k}}R.
\]
Indeed, if $x\in Q_0$ and $x\not\in \bigcup_{R\in \Top_{k}}R$, then $x\in G_{R}$ for some $R\in \Top_{j}$ for some $j<k$, that is, $x\in G_{k-1}$. Let 
\[
u_{k} = \sum_{R\in \Top_{k}} \ps{\sum_{Q\in \Next(R)} \frac{\theta_{\omega}(Q)}{\theta_{\omega}(R)}\one_{Q} + \frac{k}{\theta_{\omega}(R)}\one_{G_{R}}}+\one_{G_{k-1}}.
\]

and
\[
h_{k}=\sum_{R\in \Top_{k}}\one_{R}\avint_{R}\log \frac{1}{u_{k}}
=\sum_{R\in \Top_{k}}\one_{R}\avint_{R}g_{R}.
\]
Note that by \cite[Theorem 2.12 (2)]{Mattila}, 
\[
\int_{E}k \leq \omega(E) \;\; \mbox{for $E$ Borel.}
\]
Hence, for $R\in \Top_{k}$,
\begin{align*}
\int_{R} u_{k}
& =\theta_{\omega}^{d}(R)^{-1}\ps{\sum_{Q\in \Next(R)} |Q| \theta_{\omega}(Q) + \int_{G_{R}} k}\\
& \leq  \theta_{\omega}^{d}(R)^{-1}\ps{\sum_{Q\in \Next(R)} \omega(Q) +\omega(G_{R})}\\
& =\theta_{\omega}^{d}(R)^{-1}\omega(R) = |R|.
\end{align*} 

Thus, for $x\in R\in \Top_{k}$, by Jensen's inequality, 
\[
h_{k}(x) =\avint_{R} g_{R}
=- \avint_{R} \log u_{k}
\geq -\log \avint_{R} u_{k} \geq  -\log 1 =0 .
\]
and for $x\in G_{k-1}$, $h_{k}(x)=0$. Thus, $ h_{k}\geq 0$ everywhere. 
Moreover, $\sum h_{k}=\log\frac{1}{k}+\log\theta_{\omega}(Q_0)$ by \eqref{e:rd}, so the monotone convergence theorem implies that 
\[
\sum_{k=0}^{N-1} \sum_{R\in \Top_{k}}\int_{R} g_{R} 
=\sum_{k=0}^{N-1} \sum_{R\in \Top_{k}}\int_{R} h_{k}
=\sum_{k=0}^{N-1} \int h_{k}
\leq \sum_{k=0}^{\infty}   \int  h_{k}
= \log\frac{1}{k}+\log\theta_{\omega}(Q_0).
\]
Combined with our earlier estimates, this proves \eqref{e:sum<log1/k}.

\end{proof}

%Letting $k_{0}\rightarrow \infty$, we get 
%\[
%\avint_{Q_{0}} (\log k-\log \Theta_{w}(Q_{0}))\gec -\beta_{\d\Omega}(Q_{0})/\ell(Q_{0})^{d}. \]
%Adding $\log \Theta_{w}(Q_{0})$ and multiplying both sides by $-1$ gives
%\[\avint_{Q_{0}} \log \frac{1}{k} \lec \log \Theta_{\omega}^{d}(Q_{0})^{-1}  + \beta_{\d\Omega}(Q_{0})/\ell(Q_{0})^{d} \lec 1+ \beta_{\d\Omega}(Q_{0})/\ell(Q_{0})^{d}.
%\]
%
%.
%%%

\section{Proof of Theorem \ref{t:hruscev}}

The proof is somewhat similar to that in the previous section. Assuming the conditions of Theorem \ref{t:hruscev}, let $k\in \mathbb{N}$. We can assume $E$ is a compact subset of $Q_0$ and $Q_0\in \cD_{0}$. 

%We define a dyadic Hausdorff measure of $E$ as 
%\[
%\cH^{d}_{\cD}(E) =\sup_{\delta>0}\inf\ck{\sum \ell(Q_{j})^{d}: E\subseteq \bigcup Q_{j},\;\; \ell(Q_j)<\delta}.
%\]
%It is not hard to check that $\cH^{d}(E)\sim \cH^{d}_{\cD}(E)$. 
%
%Pick a collection $\{Q_j\}$ a collection of cubes  that cover $E$ so that 
%\[
%\sum_{j} \ell(Q_{j})^{d} \gec \cH^{d}_{\cD}(E)\gec \cH^{d}(E)

We first claim that there is a finite collection of cubes $\cC$ covering $E$ so that 
\[
\cH_{\infty}^{d}(E)\sim \sum_{Q\in \cC}\ell(Q)^{d}.
\]
To see this, first by definition of Hausdorff content (or rather, dyadic Hausdorff content, since the two are comparable), we may find a collection of disjoint cubes $\cC'$ covering $E$ so that 
\[
\cH_{\infty}^{d}(E)\sim \sum_{Q\in \cC'}\ell(Q)^{d}.\]
If $\cC'$ is finite, then we can set $\cC=\cC'$. Otherwise, notice that $\{B_{Q}^{\circ}\}_{Q\in \cC'}$ is an open cover of $E$, and so we can take a finite subcover $\{B_{Q}^{\circ}\}_{Q\in \cC''}$ for some $\cC''\subseteq \cC$ finite. Now let $\cC$ be the maximal cubes $Q\subseteq Q_0$ for which there is a sibling $Q'$and there is $Q''\in \cC''$ so that $\ell(Q)=\ell(Q')=\ell(Q'')$ and $Q'\cap B_{Q''}\neq\emptyset$. Then $\cC$ is finite, covers $E$, and 
\[
\cH^{d}_{\infty}(E)\lec \sum_{Q\in \cC}\ell(Q)^{d}\sim \sum_{Q\in \cC''}\ell(Q)^{d}
\leq \sum_{Q\in \cC'}\ell(Q)^{d}\sim \cH^{d}_{\infty}(E).
\]
This proves the claim. 

Let $\cS$ be the maximal cubes that contain an element of $\cC$ as a child. Then $E\subseteq \bigcup_{Q\in \cS}$ and so 
\[
\cH_{\infty}^{d}(E)\lec \sum_{Q\in \cS}\ell(Q)^{d} 
\lec \sum_{Q\in \cC}\ell(Q)^{d} \sim \cH_{\infty}^{d}(E).
\]

Let $\cB$ be the maximal cubes that do not intersect a cube in $\cS$ and let $\cB'$ be their children. Then $\cB\cup \cS$ is finite. Let $\cT$ be those cubes that contain an element of $\cB\cup \cS$. By Lemma \ref{l:ds-cover}  (with $\cC$ replaced by $\cB'\cup \cC$ in our case), $\cT$ is a stopping-time region whose minimal cubes are $\cB\cup \cS$.
%
% We claim $\cT$ is a stopping-time region. Indeed, we just need to show that if $Q\in \cT$ (and $Q\neq Q_0$) then all its siblings are in $\cT$. We split into cases:
%\begin{enumerate}
%\item If $Q$ contains an element of $\cS$, then one of its children does as well. All children that contain an element of $\cS$ will be in $\cT$. If there is a child $Q'$ that does not contain an element of $\cS$, then $Q'\in \cB$, since $Q$ contains an element of $\cS$ so $Q'$ is a maximal cube disjoint from cubes in $\cS$. Thus, all children of $Q$ are in $\cT$ in this case. 
%\item If $Q$ does not contain an element of $\cS$, then it only contains cubes in $\cB$, but then $Q\in \cB$ since otherwise, $Q$ would properly contain a cube $Q'\in \cB$, in which case the parent $Q''$ of $Q'$ would have to contain an element of $\cS$ since $Q'$ was maximal, and so $Q$ would contain an element of $\cS$ as well, a contradiction.
%\end{enumerate}

Let $\BTM$ be the children of the minimal cubes. Let $\delta_{x}$ denote the Dirac mass at $x$ and let
\[
\mu=\sum_{Q\in \BTM}\ell(Q)^{d}\delta_{\zeta_{Q}}.
\]
 By Lemma \ref{l:CDHM<beta},
\[
\mu(Q_0)=\sum_{Q\in \BTM}\ell(Q)^{d}\lec \beta_{\d\Omega}(\cT).
\]
\def\FC{{\rm FC}}
Let $5\rho^{k_0}$ be the size of the smallest dyadic cube in $\BTM$ and set $\mu^{k_0}=\mu$. We define a collection of cubes $\FC$ (for ``Frostmann cube") and measures $\mu^{k}$ as follows. Suppose $k\leq k_0$ and $\mu^{k}$ has been defined. Let $Q\in \cD_{k-1}\cap \cT$. If 
\[
\mu^{k}(Q)>2\ell(Q)^{d},\]
then add $Q$ to $\FC$ and define 
\[
\mu^{k-1}|_{Q}=\frac{\ell(Q)^{d}}{\mu^{k}(Q)}\mu^{k}|_{Q}.
\]

Note that by induction,
\def\Child{{\rm Child}}
\[
2\ell(Q)^{d}<\mu^{k}(Q)\leq\sum_{R\in \Child(Q)}\mu^{k}(R)
\leq \sum_{R\in \Child(Q)} 2\ell(R)^{d}\lec \ell(Q)^{d},
\]
so we have 
\begin{equation}
\label{e:muksimmuk-1}
\mu^{k}|_{Q}\lec \mu^{k-1}|_{Q} \leq \mu^{k}|_{Q}.
\end{equation}
Otherwise, we just set $\mu^{k-1}|_{Q}=\mu^{k}|_{Q}$. Continue in this way and set $\nu=\mu^{0}$. 

For $Q\in \FC$, let $n(Q)$ denote the number of cubes in $\FC$ properly containing $Q$. If $k(Q)$ is so that $Q\in \cD_{k(Q)}$, and if $R\in \BTM$ and $R\subseteq Q$, and $R=Q^{0}\subseteq Q^1\subseteq \cdots \subseteq Q^{n(R)-n(Q)}=Q$ are in $\FC$, then
\[
\mu^{k(Q)}(R) 
=\mu^{k(R)}(R) \prod_{i=1}^{n(R)-n(Q)} \frac{\ell(Q^{i})^{d}}{\mu^{k(Q^{i})+1}(Q^{i})}
<\frac{\mu^{k(R)}(R) }{2^{n(R)-n(Q)}} = \frac{\ell(R)^{d}}{2^{n(R)-n(Q)} }.
\]

Thus,
\begin{align*}
\sum_{Q\in \FC} \ell(Q)^{d}
& =\sum_{Q\in \FC}\mu^{k(Q)}(Q)
\leq \sum_{Q\in \FC}\sum_{R\in \BTM\atop R\subseteq Q} \mu^{k(Q)}(R)\\
& \leq \sum_{R\in \BTM} \sum_{Q\in \FC\atop Q\supseteq R} \mu^{k(Q)}(R)
<  \sum_{R\in \BTM} \sum_{Q\in \FC\atop Q\supseteq R}\frac{\ell(R)^{d}}{2^{n(R)-n(Q)}}\\
& \lec \sum_{R\in \BTM} \ell(R)^{d}\lec \beta_{\d\Omega}(\cT).
\end{align*}

For $R\in \FC$, let $\Stop_{F}(R)$ be the maximal cubes in $\FC$ properly contained in $R$ and $\Tree_{F}(R)$ be those cubes contained in $R$ containing a cube from $\Stop_{F}(R)$. 

\begin{lemma}
\label{l:thetanu}
For $R\in \FC\cup \{Q_0\}$ and $Q\in \Tree_{F}(R)$,
\begin{equation}
\label{e:thetanu}
\Theta_{\nu}^{d}(Q) \sim\Theta_{\nu}^{d}(R). 
\end{equation}
\end{lemma}

\begin{proof}
Recall that $\nu|_{R}$ is a multiple of $\mu^{k(R)}|_{R}$, so it suffices to prove the same statement with $\mu^{R}:=\mu^{k(R)}$ in place of $\nu$. 
We can assume $Q\neq R$. If $Q\not\in \Stop_{F}(R)$, then $\mu^{R}(Q)\leq 2\ell(Q)^{d}$ and if $Q\in \Stop_{F}(R)$, then $\mu^{R}(Q)=\ell(Q)^{d}$. Thus, we just need to focus on showing a lower bound for $\mu^{R}(Q)$:
\[
\mu^{R}(Q) 
=\sum_{T\in \Stop_{F}(R)\atop T\subseteq Q}\mu^{R}(T)
=\sum_{T \in \Stop_{F}(R)\atop T\subseteq Q}\ell(T)^{d}
\gec \cH^{d}_{\infty}(Q)\sim \ell(Q)^{d}.
\]
This proves the lemma.
\end{proof}

%
%Note that of $R\in \FC$ and $Q\in \Stop_{F}(R)$, if $T\in \Tree(Q)$, then
%\[
%\mu^{R}(T) = \mu^{Q}(T) \frac{\ell(Q)^{d}}{\mu^{
%\[
%\mu^{R}(Q)\sim \ell(Q)^{d}=\mu^{Q}(Q)
%\]
%and so for any $T\in \Tree(Q)$, 
%\[
%\Theta_{\mu^{R}}(Q)\sim \Theta_{\mu^{Q}}(Q)\sim 

\begin{lemma}
\label{l:nu>exp}
\[
\nu(E)\gec \cH^{d}_{\infty}(E)\exp(-C\beta_{\d\Omega}(\cT)/\cH^{d}_{\infty}(E)).
\]
\end{lemma}

\begin{proof}
For $k\geq 0$, let $\cS(k)$ be those cubes in $\FC$ that are properly contained in $k$ many cubes in $\FC$. Let $\delta\in (0,1/2)$ and let $N\in \mathbb{N}$ be the maximal integer so that 
\[
\sum_{Q\in \cS(k)} (2\ell(Q))^{d}\geq \frac{ \cH^{d}_{\infty}(E)}{2} \mbox{ for all }1\leq k\leq N.
\]
Then
\[
N\leq   \frac{2}{ \cH^{d}_{\infty}(E)}\sum_{k=1}^{N}\sum_{Q\in \cS(k)} (2\ell(Q))^{d}
\leq \frac{2}{ \cH^{d}_{\infty}(E)} \sum_{Q\in \FC}(2\ell(Q))^{d} \lec\frac{ \beta_{\d\Omega}(\cT)}{ \cH^{d}_{\infty}(E)}
\]
Thus, 
\[
\cH^{d}_{\infty}\ps{E\cap \bigcup_{Q\in \cS(N+1)}Q}
\leq \sum_{Q\in \cS(N+1)}(\diam B_{Q})^{d}
= \sum_{Q\in \cS(N+1)}(2\ell(Q))^{d}<\frac{ \cH^{d}_{\infty}(E)}{2}.
\]
Hence, if we let $E'=E\backslash \bigcup_{Q\in \cS(N+1)}Q$, then $\cH^{d}_{\infty}(E')\geq \cH^{d}_{\infty}(E)/2$. By \eqref{e:muksimmuk-1}, if $Q\in \BTM\backslash \cS(N+1)$, then for some $c\in (0,1)$,
\begin{equation}
\label{e:nu>cellQ}
\nu(Q)\geq c^{N+1}\mu^{k_0}(Q) = c^{N+1}\ell(Q)^{d}.
\end{equation}
Thus,
\[
\nu(E)
\geq \nu(E')
=\sum_{Q\in \BTM\atop Q\subseteq E'} \nu(Q)
\geq c^{N+1} \sum_{Q\in \BTM\atop Q\subseteq E'} \ell(Q)^{d}
\gec c^{N+1}\cH^{d}_{\infty}(E')
\gec c^{N+1} \cH^{d}(E).
\]

\end{proof}
Let $\cT'$ be those cubes in $\cT$ that properly contain a cube from $\cS(N+1)\cup \BTM$. 
%Let $\cS'(N+1)$ be the set of maximal parents of cubes in $\cS(N+1)$ and let $\cT'$ be those cubes in $\cT$ that are not contained in any cube from $\cS'(N+1)$. 
Let $\BTM'$ be the children of the minimal cubes of $\cT'$, so they are each adjacent to a cube in $\cS(N+1)\cup \BTM$, and so to a cube in $\FC\cup \BTM$. %which are maximal cubes in $\BTM\cup \cS(N+1)$.
Now we partition $\cT'$ into subtrees as follows. For $R\in \cT$, let $\Stop'(R)$ be the cubes in $\cT'$ that contain cube from $\FC\cup \Top$ as a child, and let $\Next'(R)$ be the children of the cubes in $\Stop'(R)$. For $R\in \BTM'$, we let $\Next'(R)=\emptyset$. Let $\Top_{0}'=\{Q_0\}$ and for $k\geq 0$, let $\Top_{k+1}'=\bigcup_{R\in \Top_{k}'}\Next'(R)$ and $\Top'=\bigcup_{k\geq 0}\Top_{k}'$. Note that each $R\in \Top'$ either adjacent to a cube in $\BTM'\cup \Top\cup\FC$, and recalling that cubes in $\BTM'$ are each adjacent to a cube in $\FC\cup \BTM$, we have
\[
\sum_{R\in \Top'}\ell(R)^{d}
\lec \sum_{R\in \FC\cup \Top\cup \BTM}\ell(R)^{d}\lec \beta(\cT).
\]

Let $\theta_{\omega}^{\nu}(Q)= \omega(Q)/\nu(Q)$,
\[
k=\sum_{Q\in \BTM'}\Theta_{\omega}^{\nu}(Q)\one_{Q},
\]
and for $R\in \Top'$,
\[
f_{R} = \sum_{Q\in \Next'(R)} \log\frac{\Theta_{\omega}^{\mu}(Q)}{\Theta_{\omega}^{\nu}(R)}\one_{R}.
\]
Note that since each $\Tree'(R)$ is contained in some $\Tree_{F}(R')$ for some $R'\in \FC\cup \{Q_0\}$ and is also contained in $\Tree(R'')$ for some $R\in \Top$, \eqref{e:thetanu} and the conclusion of Theorem \ref{t:SU-version} imply that 
\[
|f_{R}|\lec_{\tau,A} 1.
\]
Thus,
%Moreover, since each $Q\in \cT'$ is not contained in a cube from $\cS(N+1)$, \eqref{e:nu>cellQ} holds, and so
\begin{align*}
\avint_{Q_{0}} \ps{\log k - \log \Theta_{\omega}^{\nu}(Q_0)}
& =\sum_{R\in \Top'} \int_{Q_0} \frac{f_{R}}{\nu(Q_{0})} d\nu 
\gec_{A,\tau}-\sum_{R\in \Top'} \frac{\nu(R)}{\nu(Q_{0})}\\
& \gec -\sum_{R\in \Top'} \frac{\ell(R)^{d}}{\ell(Q_{0})^{d}}
\gec -\frac{\beta(\cT)}{\ell(Q_{0})^{d}}.
\end{align*}
Following the proof of \cite[Theorem 1]{Hru84} up to \cite[Equation (7)]{Hru84} (with Lebesgue measure replaced with $\nu$ and $w$ replaced by $k$), and recalling Lemma \ref{l:nu>exp}, we obtain that there is $C>1$ so that 
\[
\log \frac{C\beta(\cT)}{\ell(Q_{0})^{d}}
\geq \frac{\nu(E')}{\nu(Q_0)}\log\frac{\omega(Q_{0})}{\omega(E)}
\geq \exp\ps{-C\frac{\beta(\cT)}{\cH^{d}_{\infty}(E)}} \frac{\cH^{d}(E)}{\ell(Q_{0})^{d}}\log\frac{\omega(Q_{0})}{\omega(E)}.
\]

Noting that $\cH^{d}_{\infty}(E)\lec \ell(Q_{0})^{d}$, the above implies (for a slightly larger $C_0>C$)
\[
\log \frac{\omega(Q_{0})}{\omega(E)}
\leq  \exp\ps{C\frac{\beta(\cT)}{\cH^{d}_{\infty}(E)}}\frac{\ell(Q_{0})^{d}}{\cH_{\infty}^{d}(E)}\log \frac{C\beta(\cT)}{\cH^{d}_{\infty}(E)}
\lec \frac{\ell(Q_{0})^{d}}{\cH_{\infty}^{d}(E)} \exp\ps{C_0\frac{\beta(\cT)}{\cH_{\infty}^{d}(E)}}.
\]
Thus, for some $C_1>0$,
\[
\frac{\omega(E)}{\omega(Q_{0})}
\geq \exp\ps{ -C_1\frac{\ell(Q_{0})^{d}}{\cH_{\infty}^{d}(E)} \exp\ps{C_0\frac{\beta(\cT)}{\cH_{\infty}^{d}(E)}}}.
\]

\section{Proof of Theorem \ref{t:example}}

\label{s:example}

In this section we give an example of a domain in $\bC$ whose boundary is Ahlfors regular with finite linear deviation but whose harmonic measure has a singular part. The proof is an adaptation of some of the details in \cite{Bat96}.

Let $f_{0},...,f_{3}$ be as in the introduction. Let $I_\emptyset=[-1/2,1/2]^{2}$ and for a word $\alpha=\alpha_{1}\cdots \alpha_{n}$ with $\alpha_i\in \{0,1,2,3\}$, let 
\[
f_{\alpha}=f_{\alpha_{n}}\circ \cdots \circ f_{1} \;\;\; \mbox{ and }\;\;
\;
I_{\alpha} = f_{\alpha}(I_0).
\]
See Figure \ref{f:cantor-label}

\begin{figure}
\includegraphics[width=200pt]{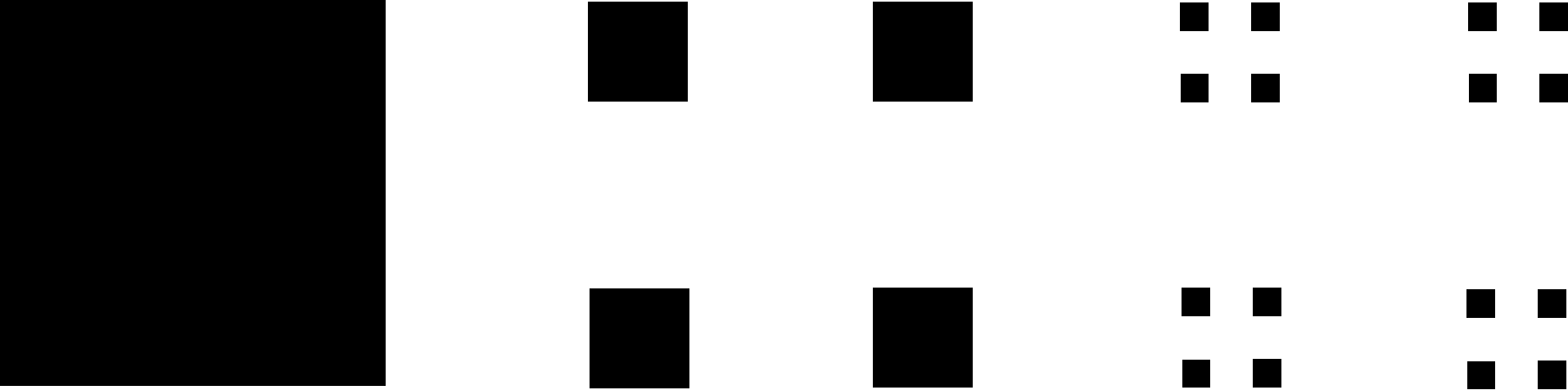}
\begin{picture}(200,0)(0,0)
\put(20,0){$I_{\alpha}$}
\put(75,0){$I_{\alpha 3}$}
\put(110,0){$I_{\alpha 4}$}
\put(110,40){$I_{\alpha 1}$}
\put(75,40){$I_{\alpha 2}$}
\put(200,0){$I_{\alpha 1 3}$}
\end{picture}
\caption{The cubes $I_{\alpha}$, its children, and grandchildren.}
\label{f:cantor-label}
\end{figure}

Let $K$ denote the 4-corner cantor set and $\omega=\omega_{K^{c}}^{\infty}$ denote the harmonic measure with pole at infinity for $\Omega = \bC\backslash K$. We let $|A|=\cH^{1}(A\cap K)$.

By \cite[Lemma 2.7]{Bat96}, there is $\rho>1$ so that for all $\alpha$, there is $i$ so that 
\[
\frac{\omega(I_{\alpha i})}{|I_{\alpha i}|} > \rho \frac{\omega(I_{\alpha})}{|I_{\alpha}|}.
\]
If we iterate this, this means there is $N$ and a word $\beta$ with $|\beta|=N$ so that 

\[
\frac{\omega(I_{\alpha \beta})}{|I_{\alpha \beta}|} > 4 \frac{\omega(I_{\alpha})}{|I_{\alpha}|}.
\]
\def\Ch{{\rm Child}}
For $k\geq 0$, let $D_{k} = \{I_{\alpha}: |\alpha|=kN\}$ and for $k\geq 0$, $I_{\alpha}\in D_{k}$, and $n\in \bN$, let 
\[
\Ch_{n}(I_{\alpha})=\{I_{\alpha \beta}:|\beta|=nN\}\subseteq D_{k+n}, \;\; \Ch(I_{\alpha}):=\Ch_{1}(I_{\alpha})\}.
\]
Let $D=\bigcup D_{k}$. Thus, the above says that for all $I\in D$ there is $I'\in \Ch(I)$ so that 
\[
\frac{\omega(I')}{|I'|} > 16 \frac{\omega(I)}{|I|}.
\]

Thus, by Cauchy-Schwartz, the fact that $(1-t)^{\frac{1}{2}}\leq 1-t/2$ for $t\geq 0$, and since $\omega(J)\geq c\omega(I)$ for all $J\in \Ch(I)$ (since $\omega$ is doubling as $\Omega$ is uniform),
\begin{align*}
\sum_{J\in \Ch(I)} \omega(J)^{\frac{1}{2}}|J|^{\frac{1}{2}}
& 
 = \sum_{J\in \Ch(I) \atop J\neq I'} \omega(J)^{\frac{1}{2}}|J|^{\frac{1}{2}} +  \omega(I')^{\frac{1}{2}}|I'|^{\frac{1}{2}} \\
 & \leq \ps{\sum_{J\in \Ch(I) \atop J\neq I'} \omega(J)}^{\frac{1}{2}}\ps{\sum_{J\in \Ch(I) \atop J\neq I'}  |J|}^{\frac{1}{2}} +\omega(I')^{\frac{1}{2}} \ps{ \frac{\omega(I')|I|}{16\omega(I)}}^{\frac{1}{2}}\\
 & \leq |I|^{\frac{1}{2}} \omega(I)^{\frac{1}{2}} \ps{\ps{1-\frac{\omega(I')}{\omega(I)}}^{\frac{1}{2}} + \frac{1}{4} \frac{\omega(I')}{\omega(I)}} \\
 & \leq  |I|^{\frac{1}{2}} \omega(I)^{\frac{1}{2}}\ps{1-\frac{1}{4}\frac{\omega(I')}{\omega(I)}}
 \leq |I|^{\frac{1}{2}} \omega(I)^{\frac{1}{2}}(1-c/4)=: \lambda |I|^{\frac{1}{2}} \omega(I)^{\frac{1}{2}}
\end{align*}

Let $\tau>0$ be small, $n\in \mathbb{N}$. Since $|J| = 4^{-N}|I|$ for $J\in \Child(I)$,
\begin{align*}
\sum_{J\in \Ch(I)} \omega(J)^{\frac{1}{2}}|J|^{\frac{1-\tau}{2}} 
& = 4^{\frac{N\tau}{2}}|I|^{-\frac{\tau}{2}}  \sum_{J\in \Ch(I)} \omega(J)^{\frac{1}{2}}|J|^{\frac{1}{2}} \\
& < 4^{\frac{N\tau}{2}}\lambda \omega(I)^{\frac{1}{2}}|I|^{\frac{1-\tau}{2}}
=: \gamma \omega(I)^{\frac{1}{2}}|I|^{\frac{1-\tau}{2}}
\end{align*}
where $\gamma <1$ for $\tau$ small enough depending on $N$ and $\lambda$. Let
\[
\Stop(I,n)=\ck{ J\in \Ch_{n}(I): \frac{\omega(J)}{\omega(I)}<\frac{|J|^{1-\tau}}{|I|^{1-\tau}}}.
\]
Then iterating the above estimate, we get 
\[
\sum_{J\in \Stop(I,n)}\omega(J)
=\sum_{J\in \Stop(I,n)}\omega(J)^{\frac{1}{2}}\omega(J)^{\frac{1}{2}}
\leq \frac{\omega(I)^{\frac{1}{2}}}{|I|^{\frac{1-\tau}{2}}} \sum_{J\in \Stop(I,n)}\omega(J)^{\frac{1}{2}}|J|^{\frac{1-\tau}{2}}
< \gamma^{n} \omega(I).
\]
Let 
\[
\Top(I,n)=\Ch_{n}(I)\backslash \Stop(I,n).
\]
Then
\begin{equation}
\label{e:h1decay}
\sum_{J\in \Top(I,n)}|J|
 = \sum_{J\in \Top(I,n)} |J|^{\tau}|J|^{1-\tau} 
 \leq  \sum_{J\in \Top(I,n)} \ps{\frac{|I|^{\tau}}{ 4^{nN\tau}}}\ps{ |I|^{1-\tau} \frac{\omega(J)}{\omega(I)}}   \\ 
 \leq \frac{|I|}{4^{nN\tau}} .
\end{equation}

Let $\Top(0) =I_{\emptyset}$. For $n>0$, let 
\[
\Stop(n) = \bigcup_{I\in \Top(n-1)} \Stop(I,n), \;\; \Top(n) = \bigcup_{I\in \Top(n-1)} \Top(I,n).
\]
Then for $n>0$,
\begin{align*}
\sum_{J\in \Top(n)}\omega(J)
& =\sum_{I\in \Top(n-1)} \sum_{J\in \Top(I,n)} \omega(J)\\
& = \sum_{I\in \Top(n-1)} \ps{\omega(I)-\sum_{J\in \Stop(I,n)}\omega(J)}\\
& >\sum_{I\in \Top(n-1)}\omega(I)(1-\gamma^{n})
 >\cdots > \omega(I_{\emptyset}) \prod_{k=1}^{n} (1-\gamma^{k})
\end{align*}
Thus, taking $n\rightarrow\infty$, if we set 
\[
G=\bigcap_{n\geq 0}\bigcup_{I\in \Top(n)} I,\]
we get that $\omega(G)>0$. Moreover, since $G\subseteq \Top(n)$ for all $n$, 
\[
|G|
\leq  \sum_{J\in \Top(n)} |J|
\leq \sum_{I\in \Top(n-1)}\sum_{J\in \Top(I,n)} |J|
< \frac{1}{4^{nM\tau}} \sum_{I\in \Top(n-1)}.
\]
Hence $|G|=0$, since otherwise, taking the intersection over $\bigcup_{J\in \Top(n)} J$, the above would give $|G|\leq \frac{1}{4^{nM\tau}}|G|<|G|$, a contradiction.

\def\Homega{\hat{\Omega}}
\def\homega{\hat{\omega}}

Now let $\eta>1$ be very close to $1$ and 
\[
F=\cnj{\bigcup_{n\geq 0} \bigcup_{I\in \Stop(n)} \eta I} =G\cup \bigcup_{n\geq 0} \bigcup_{I\in \Stop(n)} \eta I
\]

Let $\Homega=F^{c}$ and $\homega=\omega_{\Homega}^{\infty}$. See Figure \ref{f:homega}

\begin{figure}
\includegraphics[width=350pt]{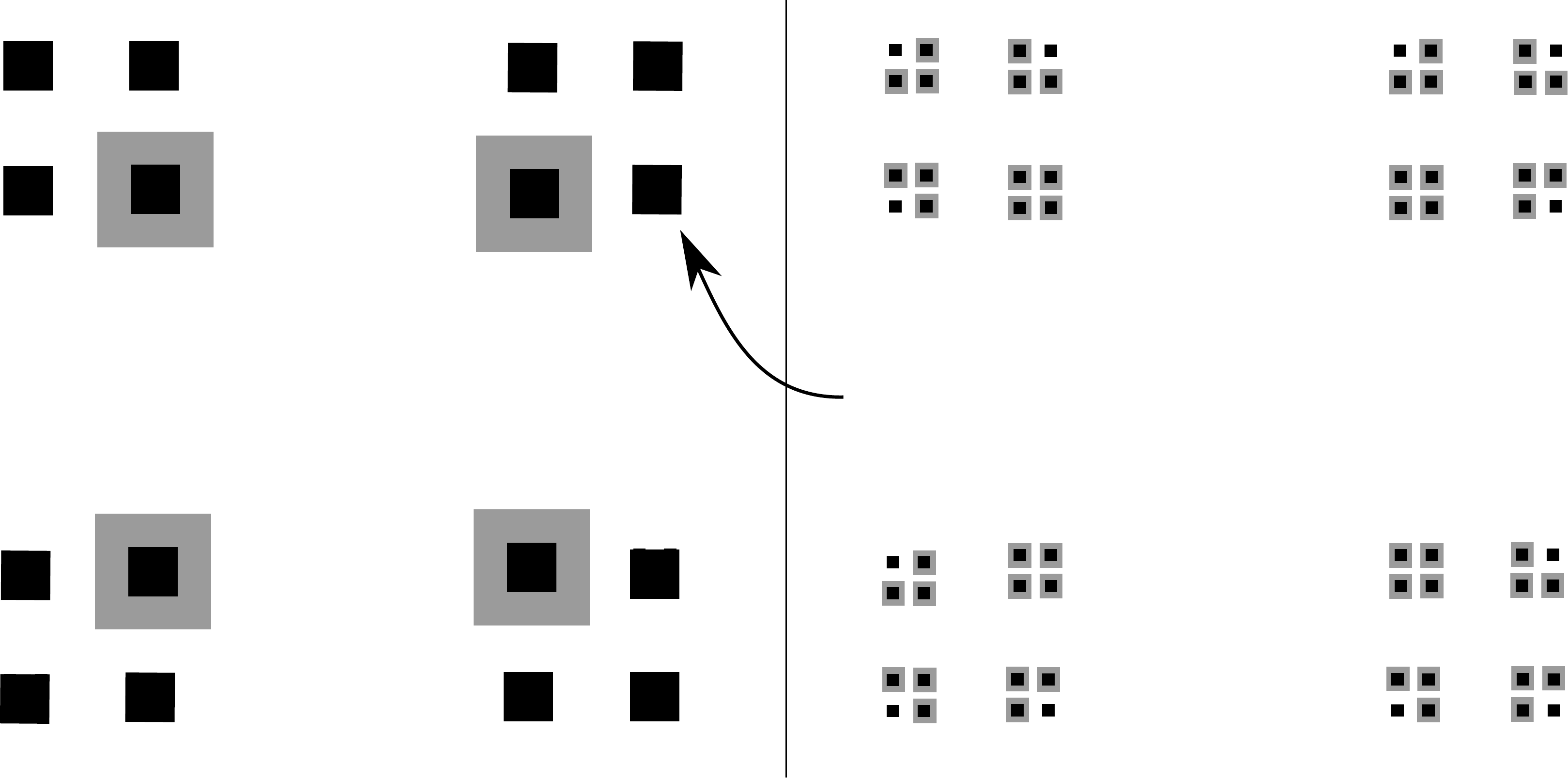}
\begin{picture}(350,0)(0,0)
\put(195,95){$J\in \Top(I,n)$}

\put(20,0){$\Child(I,n)$ for some $I$}
\put(180,0){$\Child(J,n+1)$ for some $J\in \Top(I,n)$}
\end{picture}
\caption{On the left, we have the cubes in $\Child_{n}(I)$ for some cube $I$. Around each cube $J\in \Stop(I,n)$, we place a dilated copy $\eta J$ where $\eta>1$ is a number very close to $1$ so that $\eta J$ contains but is very close to $J$. (In this figure, $n=2$). These cubes form part of the complement of $\Homega$. For the cubes $J\not\in \Stop(I,n)$, these form the top cubes $\Top(I,n)$, and inside each such cube (shown on the right), we consider the cubes that are $n+1$ generations below $J$ (so in this figure, $n+1=3$), and place dilated copies around those cubes in $\Stop(J,n+1)$, which again will be part of $\Homega^{c}$ . Notice that the cubes in $\Stop(J,n+1)$ take up a larger portion of the Hausdorff $1$-measure of $K$ in $J$. }
\label{f:homega}
\end{figure}

It is a bit annoying but rather straightforward to check that $\Homega$ is also uniform, so we omit the details

{\bf Claim:} $\omega(G)>0$. We recall a lemma about harmonic measure for uniform domains:

\begin{lemma} \cite[Theorem 1.3]{MT15} Let $n\geq 1$, $\Omega$ be a uniform domain in $R^{d+1}$ and let $B$ be a ball centered on $\d\Omega$ . Let $p_1,p_2\in\Omega$ be such that $\dist(p_i,B\cap \d\Omega)>c_0^{-1}r_B$ for$ i=1,2$. Then, for any Borel set $E\subseteq B\cap \d\Omega$ ,
\[
\frac{\omega^{p_1}(E)}{\omega_{p_1}(B)} \sim \frac{\omega^{p_2}(E)}{\omega_{p_2}(B)}
\]
with the implicit constant depending only on $c_0$ and the uniform behavior of $\Omega$.
\end{lemma}

This was originally shown for NTA domains in \cite{JK82}. In our setting, since $\d\Omega$ is $1$-Ahlfors regular, if $B=B(0,\diam \d\Omega)$, then $\omega^{0}_{\Omega}(\d\Omega)= 1$. Hence, if $x_{j}$ is a sequence of points going to infinity so that $\omega_{\Omega}^{x_{j}}/\omega_{\Omega}^{x_{j}}(\d\Omega)\warrow \omega$, 
\[
0<\omega(G) 
= \lim_{j\rightarrow\infty}  \frac{\omega_{\Omega}^{x_{j}}(G)}{\omega_{\Omega}^{x_{j}}(\d\Omega)}
\gec \frac{\omega_{\Omega}^{0}(G)}{\omega_{\Omega}^{0}(\d\Omega)}
= \omega_{\Omega}^{0}(G).
\]
The first equality is justified since $G$ is a compact totally disconnected and \cite[Exercise 1.9, p.22]{Mattila}. Thus, by Harnack's principle, since $\Omega$ is connected, $\omega_{\Omega}^{x}(G)>0$ for all $x\in G$. 

{\bf Claim:} $\homega(G)>0$. The proof here is modelled after that in the proof of \cite[Lemma 6.3]{JK82}.

First we need to show there is $c\in (0,1)$ so that 
\begin{equation}
\label{e:wG^c>1}
\omega_{\Omega}^{x}(G^{c})\gec 1 \;\; \mbox{ for all }x\in \d\Homega\backslash G. 
\end{equation}
Note that if $x\in \d\Homega\backslash G$, then $x\in \d \eta I$ for some square $I\in \Stop(n)$ for some $n$, and so $I\subseteq G^{c}$. By Lemma \ref{l:bourgain}, $\omega_{\Omega}^{x}(G^{c})\geq \omega_{\Omega}^{x}(I)\gec 1$. This proves \eqref{e:wG^c>1}.

Next, suppose that $\homega^{x}(G)=0$. Let $\phi$ be a lower function for $\one_{G}$ with respect to $\Omega$. Then by definition of harmonic measure, $\phi(x) \leq \omega_{\Omega}^{x}(G)$ for all $x\in \Omega$, and so by \eqref{e:wG^c>1}, for $x\in \d\Homega\backslash G$,  there is $c\in (0,1) $ so that 
\[
\phi(x) \leq \omega_{\Omega}^{x}(G)<c<1. 
\]
Thus, $\phi(x) - c$ is a lower function for $\one_{G}$ {\it with respect to $\Homega$}, since  for all $\xi\in G$
\[
\limsup_{x\rightarrow \xi}(\phi(x)-c)\leq 1-c<1
\]
 and  for $\xi\in \d\Homega\backslash G$,
 \[
 \limsup_{x\rightarrow \xi}\phi(x)\leq \omega_{\Omega}^{\xi}(G)-c<c-c=0.\]
 Hence, $\phi(x)-c\leq \homega^{x}(G)=0$ and so $\phi(x)\leq c<1$ for all $x\in \Homega$. Thus, taking the supremum over all such lower functions, we get $\omega_{\Omega}^{x}(G)\leq c$ for all $x\in \Homega$. By applying the maximum principle inside each $\lambda I$ for $I\in \bigcup \Stop(n)$ (and noting that $\omega_{\Omega}^{x}(G)$ vanishes on $\d\Omega\cap \lambda I$ for each $I\in \bigcup \Stop(n)$), this implies $\omega_{\Omega}^{x}(G)\leq c$ in $\Omega$. We now evoke the following lemma:

\begin{lemma}[{\cite[Lemma 2.24]{AAM19}}] 
Suppose that $\Omega\subseteq \R^{d+1}$, $\d\Omega$ is regular, $\omega_{\Omega}^{x}(\{\infty\})=0$ for all $x\in \Omega$, and that $E$ is a closed subset of $\d^\infty\Omega$. Then $\omega_{\Omega}^{x}(E)=0$ for all $x\in \Omega$ if and only if $$\sup_{x\in \Omega}\omega_{\Omega}^{x}(E)<1.$$
\end{lemma}

Here, $\d_{\infty}\Omega = \d\Omega\cup \{\infty\}$ if $\Omega$ is unbounded and equals $\d\Omega$ otherwise. By \cite[Example 6.5.6]{AG}, $\omega_{\Omega}(\{\infty\})=0$. Thus the previous lemma implies $\omega_{\Omega}^{x}(G)=0$, which is a contradiction. This proves the claim.

Finally, we need to show that, if $\cD$ are the surface cubes for $\d\Homega$, that
\[
\sum_{Q\in \cD} \beta_{\d\Omega}(3B_{Q})^{2} \ell(Q)<\infty.\]
Since $\d\Homega$ is $1$-Ahlfors regular, and using the fact that $\beta_{\d\Omega}(B')\lec \beta_{\d\Omega} (B)$ whenever $B'\subseteq B$ and $r_{B}\lec r_{B'}$, it suffices to show that 
\[
\int_{\d\Homega}\int_{0}^{1} \beta_{\d\Omega}(B(x,r))^{2} \frac{dr}{r}d\cH^{1}(x)<\infty.
\]

Let $I\in \Top(n-1)$ and $J\in \Stop(n,I)$. Note that since $\d\lambda J$ is just a square, we know that 
\[
\int_{\d \lambda I} \int_{0}^{\ell(J)}\beta_{\d\Omega}(B(x,r))^{2} \frac{dr}{r}d\cH^{1}(x)\lec \ell(I).\]
Then we trivially bound (using that $J\in \Ch_{n}(I)$ and $\beta_{\d\Omega}\leq 1$)
\[
\int_{\d \lambda I} \int_{\ell(J)}^{\ell(I)}\beta_{\d\Omega}(B(x,r))^{2} \frac{dr}{r}d\cH^{1}(x)
\lec \ell(J)\log\frac{\ell(I)}{\ell(J)}
\lec \ell(J) n \sim |J| n.
\]
Since $|G|=0$,
\begin{align*}
\int_{\d\Homega}\int_{0}^{1} &  \beta_{\d\Omega}(B(x,r))^{2} \frac{dr}{r}d\cH^{1}(x)\\
& =\sum_{n> 0} \sum_{I\in \Top(n-1)} \sum_{J\in \Stop(n,J)}\int_{\d \lambda J}\int_{\ell(J)}^{\ell(I)}\beta_{\d\Omega}(B(x,r))^{2} \frac{dr}{r}d\cH^{1}(x)\\
& \lec  \sum_{n> 0} \sum_{I\in \Top(n-1)} \sum_{J\in \Stop(n,J)} n |J|\\
& \leq  \sum_{n> 0} \sum_{I\in \Top(n-1)}  n |I|\\
& \lec 1+\sum_{n> 0} \sum_{J\in \Top(n-2) }\sum_{I\in \Top(n-1,J)} n|I|\\
& \stackrel{\eqref{e:h1decay}}{\lec} 
1+\sum_{n> 0}  4^{-nN\tau} \sum_{J\in \Top(n-2) } n|J|<1+\sum_{n> 0}  n4^{-nN\tau}<\infty.
\end{align*}

\section{Proof of Theorem \ref{t:green}}
\label{s:IV}

\subsection{Part I: Bounding the oscillation of Green's function using CDHM}

\label{s:IV-Part-I}

In this section, we prove the following:

\begin{lemma}
Let $\Omega\subseteq \R^{d+1}$ be a bounded uniform domain with lower $d$-content regular boundary. Let $Q_0\in \cD$ and $x_0$ be a $c$-corkscrew point in $B_{Q_{0}}\cap \Omega$ so that if $B_0= B(x_{0},c\ell(Q_{0}))$, then $2B_0\subseteq \Omega\cap B_{Q_{0}}$. Let $g=G_{\Omega}(x_{0},\cdot)$ and
\[
\Omega_{Q_{0}}= \bigcup_{Q\subseteq Q_0}U_{Q}^{K}
 \]
Then for $K$ large enough (chosen as in the proof of \ref{l:WHSA<CDHM})
\begin{equation}
\ell(Q_0)^{d} + \int_{\Omega_{Q_{0}}\backslash B_{0} } \av{\frac{\grad^2 g}{g}}^{2} \delta_{\Omega}(x)^{3} dx 
\lec \beta_{\d\Omega}(Q_0).
\end{equation}
\end{lemma}

We assume the same set up as in the proof of Lemma \ref{l:WHSA<CDHM}, but now assuming $\Omega$ is also uniform. Let $\omega=\omega_{\Omega}^{x_{0}}$. Then for $R\in \Top(k_0)$ and $Q\in \Tree(R)$, by the uniform case of Lemma \ref{l:w/w},
\begin{equation}
\label{e:theta/theta}
1\sim_{A,\tau} \frac{\Theta_{\hm^{x_{R}}}^{d}(Q)}{\Theta_{\hm^{x_{R}}}^{d}(R)}
= \frac{\hm^{x_{R}}(Q)\ell(R)^{d}}{\hm^{x_{R}}(R)\ell(Q)^{d}}
\stackrel{\eqref{e:bourgain}}{\sim}
 \frac{\hm^{x_{R}}(Q)\ell(R)^{d}}{\ell(Q)^{d}}
\sim \frac{\omega(Q)\ell(R)^{d}}{\omega(R)\ell(Q)^{d}}
\sim \frac{\Theta_{\omega}^{d}(Q)}{\Theta_{\omega}^{d}(R)}
\end{equation}

In particular, if we set $g=G_{\Omega}(x_{Q_{0}},\cdot)$, then for $x\in U_{Q}^{K}$ and $Q\in \Tree(R)$, since $\Omega$ is uniform,
\begin{equation}
\label{e:g/delta2}
\frac{g(x)}{\delta_{\Omega}(x)}
\sim_{K} \frac{g(z_Q)}{\delta_{\Omega}(z_Q)}
\sim \Theta_{\omega}^{d}(Q)
\stackrel{\eqref{e:theta/theta}}{\sim} 
\sim_{A,\tau} \Theta_{\omega}^{d}(R).
\end{equation}

Thus, with the same arguments as before, if we set 
\[
\Omega_{k_0} = \bigcup_{Q\in \cD(k_0)}U_{Q}^{K},\]
then

%
%\begin{align*}
%\int_{\Omega_{k_0}} \av{\frac{\grad^2 g}{g}}^2 & \delta_{\Omega}(x)^3dx  
% \lec \sum_{R\in \Top(k_0)}
%\sum_{Q\in \Tree(R)} \int_{U_{Q}^{K}} \av{\frac{\grad^2 g}{g}}^2  \delta_{\Omega}(x)^3dx \\
%& \sim \sum_{R\in \Top(k_0)}\sum_{Q\in \Tree(R)} \int_{U_{Q}^{K}}  \av{\frac{\grad^2 g_{R}}{g_{R}}}^2  \delta_{\Omega}(x)^3dx \\
%&  \lec  \ell(R)^{-3}  \int_{\Omega_{R}}  \triangle |\grad g|^{2} g(x)dx  \\
%& \sim  \ell(R)^{-3}  \sum_{R\in \Top} \int_{\d\Omega_{R}}\ps{g\frac{d|\grad g|^2}{d\nu}-|\grad g|^2 \frac{dg}{d\nu}}d\cH^{d}\\
%& \stackrel{\eqref{e:cauchydelta}}{\lec}
%\ell(R)^{-3}   \sum_{R\in \Top} \int_{\d\Omega_R} \frac{g(x)^{3}}{\delta_{\Omega}(x)^{3}}dx\\
%  &  \stackrel{\eqref{e:g<theta}}{\lec} \cH^{d}(\d\Omega_{R} )
%  \sim \ell(R)^{d}.
%\end{align*}

\begin{align*}
\int_{\Omega_{k_0}\backslash B_{0}} \av{\frac{\grad^2 g(x)}{g(x)}}^2 & \delta_{\Omega}(x)^3dx   
\leq \sum_{R\in \Top(k_0)}
\sum_{Q\in \Tree(R)} \int_{U_{Q}^{K}\backslash B_{0}} \av{\frac{\grad^2 g}{g}}^2  \delta_{\Omega}(x)^3dx \\
&
\stackrel{\eqref{e:g/delta2}}{\sim}  \sum_{R\in \Top(k_0)}\Theta_{\hm^{x_{Q_{0}}}}(R)^{-3}  \sum_{Q\in \Tree(R)} \int_{U_{Q}^{K}\backslash B_{0}}  \av{\grad^2 g}^2 g dx \\
&
\sim  \sum_{R\in \Top(k_0)}\Theta_{\hm^{x_{Q_{0}}}}(R)^{-3} \int_{\Omega_{R}\backslash B_{0}}  \av{\grad^2 g}^2 g dx \\
&  \lec\sum_{R\in \Top(k_0)} \Theta_{\hm^{x_{Q_{0}}}}(R)^{-3}  \int_{\Omega_{R}\backslash B_{0}}  \triangle |\grad g|^{2} g(x)dx  \\
& \sim \sum_{R\in \Top(k_0)}  \Theta_{\hm^{x_{Q_{0}}}}(R)^{-3} \sum_{R\in \Top} \int_{\d(\Omega_{R}\backslash B_{0})}\ps{g\frac{d|\grad g|^2}{d\nu}-|\grad g|^2 \frac{dg}{d\nu}}d\cH^{d}\\
& \stackrel{\eqref{e:cauchydelta}}{\lec}
 \sum_{R\in \Top(k_0)} \Theta_{\hm^{x_{Q_{0}}}}(R)^{-3} \int_{\d(\Omega_{R}\backslash B_{0})} \frac{g(x)^{3}}{\delta_{\Omega}(x)^{3}}d\cH^{d}(x)\\
  &  \stackrel{\eqref{e:g<theta}}{\lec} \sum_{R\in \Top(k_0)}  \cH^{d}(\d(\Omega_{R}\backslash B_{0}))
  \sim \sum_{R\in \Top(k_0)} \ell(R)^{d} \\
  & \lec  \CDHM(Q_0,2,A,\tau) 
   \stackrel{\eqref{e:CDHM}}{\lec} \beta_{\d\Omega}(Q_0).
\end{align*}

\subsection{Part II: A Theorem on affine deviation of functions in uniform domains}
\label{s:IV-Part-II}

In this section we prove one half of Theorem \ref{t:green}. In fact, the result holds for a more general class of functions than Green's function.

\begin{theorem}
\label{t:affinedeviation}
Let $\Omega\subseteq \R^{d+1}$ be a uniform domain with lower content regular boundary, and $f\in W^{2,2}_{\loc}(\Omega)$ such that $f>0$ on $\Omega$,
\begin{equation}
\label{e:fharnack}
f(x)\sim_{\Lambda} f(y) \mbox{ whenever }\frac{|x-y|}{\min\{\delta_{\Omega}(x),\delta_{\Omega}(y)\}}\leq \Lambda.
\end{equation}
and
\begin{equation}
\label{e:fholder}
|f(x)| \one_{B\cap \Omega}(x) \lec ||f||_{L^{\infty}(B\cap \Omega)} \av{\frac{x-x_{B}}{r_{B}}}^{\alpha}.
\end{equation}
For $Q_{0}\in \cD$, let 
\[
\Omega(Q_{0})=\bigcup_{Q\subseteq Q_{0}} U_{Q}^{2\rho^{-1}K}.
\]
Then
\[
\int_{\Omega(Q_{0})}\av{\frac{\grad^{2} f(x)}{f(x)}}^{2} \delta_{\Omega}(x)^{3} dx 
\gec \beta_{\d\Omega}(Q_0).\]

\end{theorem}

The rest of this section is dedicated to proving the above theorem.

\begin{lemma}
\label{l:fcarleson}
Let $\Omega\subseteq \R^{d+1}$ be a domain, $B$ a ball centered on $\d\Omega$, and $f$ a positive function on $\Omega$ vanishing continuously on $\d\Omega\cap 2B$ satisfying \eqref{e:fharnack} and \eqref{e:fholder}. Then for any corkscrew point $y\in B\cap \Omega$,
\[
\sup_{B}f \lec f(y).
\]
\end{lemma}

\begin{proof}
The proof of this is exactly as in \cite{JK82}. 
\end{proof}

For a ball $B\subseteq \Omega$, we let $f_{B}=\avint_{B}f$. Let
\[
A_{B}=(\grad f)_{B}\cdot (x_{B}-x)+f_{B}.
\]
Observe that by Poincar\'{e}'s inequality,
\begin{equation}
\label{e:poincare}
\avint_{B} \av{ \frac{f-A_{B}}{r_{B}}}^{2} 
\lec \avint_{B}|\grad f-\underbrace{(\grad f)_{B}}_{=\grad A_{B}}|^2
\lec \avint_{B}|\grad^{2} f r_{B}|^{2}
\end{equation}

We also let 
\[
\gamma(B)= \ps{\avint_{B} \av{\frac{\grad^{2} f}{f}}^{2} r_{B}^{4}}^{\frac{1}{2}}
\]

For a ball $B$, we will write
\[
||g||_{B}=\ps{\avint_{B} g^{2} }^{\frac{1}{2}}.
\]

Notice that  by \eqref{e:fharnack} and \eqref{e:poincare}, if $\dist(B,\d\Omega)\geq c r_{B}$, then

\begin{equation}
\label{e:f-a<gamma}
||f-A_{B}||_{B}\lec_{c} \gamma(B) f_{B}
\end{equation}

\begin{lemma}
\label{l:f/d}
Let $B_{x}=B(x,(1-\varrho)\delta_{\Omega}(x))$. Then for $\varrho>0$ small enough and $\ve \in (0,1)$ small enough depending on $\varrho$, $\gamma(B_x)<\ve$ implies
\begin{equation}
\label{e:f/d}
|\grad A_{B_x}|=|(\grad f)_{B_x}|\sim \frac{f(x)}{r_{B_x}}\sim \frac{f(x)}{\delta_{\Omega}(x)}\sim \frac{A_{B_x}(x)}{\delta_{\Omega}(x)}.
\end{equation}
\end{lemma}

\begin{proof}
Let $B=B_{x}$.  Since $(A_{B})_{B}=f_{B}$, we see that 
\begin{align*}
|\grad A_{B}|
& \sim \nrm{\frac{A_{B}-(A_{B})_{B}}{r_{B}}}_{B}
=\nrm{\frac{A_{B}-f_{B}}{r_{B}}}_{B}
 \leq  \nrm{\frac{A_{B}-f}{r_{B}}}_{B}
+ \nrm{\frac{f-f_{B}}{r_{B}}}_{B}\\
& \leq \nrm{\grad^{2} f r_{B}}+ \nrm{\frac{f}{r_{B}}}_{B}+ \nrm{\frac{f_{B}}{r_{B}}}_{B}
\stackrel{\eqref{e:fharnack}}{\lec} \gamma(B)\frac{f(x)}{\delta_{\Omega}(x)}+\frac{f(x)}{r_{B}}
\lec \frac{f(x)}{r_{B}}.
\end{align*}
For the opposite inequality, let $\xi\in \d\Omega$ be so that $|x-\xi|=\delta_{\Omega}(x)=r_{B}/(1-\varrho)$. Then by Lemma \ref{l:fcarleson}, $f\lec \varrho^{\alpha} f(x)$ on $B(\xi,2\varrho r_{B})\cap B_{x}$. Thus, for $\varrho$ small enough,
\begin{align*}
\frac{f(x)}{r_{B}} 
& \lec_{\rho} \nrm{\frac{f-f_{B}}{r_{B}}}_{B}\lec ||\grad f||_{B}
 \leq ||\grad f-(\grad f)_{B}||_{B}+|(\grad f)_{B}|\\
& \lec ||\grad^{2}f r_{B}||_{B}+|\grad A_{B}|
\stackrel{\eqref{e:fharnack}}{\lec} \gamma(B)\frac{f(x_{B})}{r_{B}}+|\grad A_{B}|
<\ve \frac{f(x_{B})}{r_{B}}+|\grad A_{B}|
\end{align*}
and so for $\ve$ small enough depending on $\varrho$, this implies
\[
\frac{f(x)}{r_{B}} \lec |\grad A_{B}|.
\]

The last comparison in \eqref{e:f/d} follows from the definition of $A_{B}$ and Harnack's inequality.

\end{proof}

Let $\cW$ be the set of maximal dyadic cubes $I\subseteq \Omega$ so that $NI\subseteq \Omega$. We will call these the {\it Whitney cubes} for $\Omega$. For a Whitney cube $I\in \cW$, let $B_{I}$ be the smallest ball containing $I$ and set $A_{I}=A_{B_{I}}$. For $N$ chosen sufficiently large depending on $d$, $2B_{I}\subseteq \Omega$. 

For a ball $B\subseteq \Omega$, let $B^{*} = B_{x_B}$.

\def\dom{\d_{\Omega}}

\begin{lemma}
If $I$ is a Whitney cube and $x_{I}$ is its center, and $\gamma(B_{I}^{*})$ is small enough, then
\begin{equation}
\label{e:gradasimf/I}
|\grad A_{I}|\sim \frac{f(x_{I})}{\ell(I)}\sim A_{I}(x_{I})
\end{equation}
\end{lemma}

\begin{proof}
By the previous lemma,
\begin{align*}
|\grad A_{B_{I}^{*}}-\grad A_{I}|
& =\av{\avint_{B_{I}}\ps{\grad A_{B_{I}^{*}}-\grad f}}
\leq ||\grad A_{B_{I}^{*}}-\grad f||_{B_{I}}\\
& \lec ||\grad^{2}f||_{B_{I}^{*}}\ell(I)
\sim \gamma(B_{I}^{*}) \frac{f(x_I)}{r_{B_{I}^{*}}}
\stackrel{\eqref{e:f/d}}{\sim} \gamma(B_{I}^{*}) |\grad A_{B_{I}^{*}}|.
\end{align*}

By picking $\gamma(B_{I}^{*}) $ small enough, this implies 
\[
 |\grad A_{I}|  \sim |\grad A_{B_{I}^{*}}|\stackrel{\eqref{e:f/d}}{\sim} \frac{f(x_{I})}{\delta_{\Omega}(x_{I})}\sim  \frac{f(x_{I})}{\ell(I)}\]
 
 For the last inequality, we just observe that 
 \[
 A_{I}(x_{I})
 =\avint_{B_{I}}f\stackrel{\eqref{e:fharnack}}{\sim} f(x_{I}).
 \]
 
\end{proof}

\begin{lemma}
If $I$ and $J$ are adjacent Whitney cubes and $B_{I,J}=B_{I}^{*}\cup B_{J}^{*}$. Then  for $M\geq 1$,
\begin{equation}
\label{e:A-A-Linfinity}
||A_{I}-A_{J}||_{L^{\infty}(MB_{I})}\lec \gamma(B_{I,J}) |\grad A_{I}|M \ell(I)
\end{equation}
and 
\begin{equation}
\label{e:gradAI-gradAJ/gradAI}
\frac{|\grad A_{I}-\grad A_{J}|}{|\grad A_{I}|}\lec \gamma(B_{I,J})
\end{equation}
\end{lemma}

\begin{proof}
Let $B=B_{I,J}$. We estimate
\begin{align}
|\grad A_{I}-\grad A_{J}|
& \leq \av{\grad A_{I}-( \grad f)_{B}|+|(\grad f)_{B}-\grad A_{J}}  \notag \\
& \lec \avint_{B} |\grad f-(\grad f)_{B}|
 \leq ||\grad f-(\grad f)_{B}||_{B} \notag \\
& \lec (\diam B) ||\grad^{2} f||_{B} \lec \gamma(B)\frac{f_{B}}{\diam B}
\stackrel{\eqref{e:fharnack}}{\sim} \gamma(B)\frac{f(x_I)}{\ell(I)} \stackrel{\eqref{e:gradasimf/I}}{\sim}  \gamma(B) |\grad A_{I}|.
\label{e:ga-ga}
\end{align}

Since $I$ and $J$ are adjacent, $|B_{I}\cap B_{J}|\sim |I|$. Thus,
\begin{align*}
\avint_{B_{I}\cap B_{J}}|A_{I}-A_{J}|
& \lec \avint_{B_{I}}|A_{I}-f| + \avint_{B_{J}}|f-A_{J}|
\leq ||A_{I}-f||_{B_{I}}+||A_{J}-f||_{B_{J}}\\
& \stackrel{\eqref{e:f-a<gamma}}{\lec} \gamma(B_{I})f_{B_{I}}
 \stackrel{\eqref{e:fharnack} \atop \eqref{e:gradasimf/I}}{\sim}  \gamma(B_{I}) |\grad A_{I}|\ell(I).
\end{align*}

Thus, there is $x_0\in B_{I}\cap B_{J}$ so that $|A_{I}(x_0)-A_{J}(x_0)|\lec |\grad A_{I}|\ell(I)$. Hence, for $x\in MB_{I}$,

\begin{align*}
|A_{I}(x)-A_{J}(x)|
&  = |\grad A_{I}(x-x_0)+A_{I}(x_{0}) -(\grad A_{J}(x-x_0)+A_{J}(x_{0}) )\\
& \leq  |(\grad A_{I}-\grad A_{J})(x-x_0)| + |A_{I}(x_{0}) -A_{J}(x_{0})|\\
& \stackrel{\eqref{e:ga-ga}}{\lec}   \gamma(B)|\grad A_{I}|\cdot|x-x_0|+\gamma(B)|\grad A_{I}|\ell(I) \lec \gamma(B) |\grad A_{I}|M \ell(I).
\end{align*}

\end{proof}

\begin{lemma}
\label{l:Pangles}
For $I\in \cW$, let $P_{I}=\{x: A_{I}(x)=0\}$. If $I,J\in \cW$ are adjacent, then for $M\gg N$ and $\gamma(B_{I,J})$ small enough
\[
d_{MB_{I}}(P_{I},P_{J})\lec \gamma(B_{I,J}).
\]
\end{lemma}

\begin{proof}

First we claim that for $M'\gg N$, 
\[
P_{I}\cap M'B_{I}\neq\emptyset \mbox{ and  }P_{J}\cap M'B_{I}\neq\emptyset.\]
To see this, recall from the proof of Lemma \ref{l:f/d} that $f\lec \varrho^{\alpha}f(x_I)$ on $B_{I}^{*}\cap B(\xi,2\varrho r_{B_{I}})$ and the measure of this set is comparable to $\ell(I)^{d}$ (depending on $\varrho$), thus there is $y\in B_{I}^{*}\cap B(\xi,2\varrho r_{B_{I}})$ so that 
\begin{align*}
A_{I}(y)
& \leq \avint_{B_{I}^{*}\cap B(\xi,2\varrho r_{B_{I}})} |A_{I}|
\lec \avint_{B_{I}^{*}\cap B(\xi,2\varrho r_{B_{I}})} |A_{I}-f|+\varrho^{\alpha} f(x_{I})\\
& \stackrel{\eqref{e:f-a<gamma}}{\leq} C_\varrho\gamma(B_{I}^{*})f_{B_{I}} +\varrho^{\alpha} f(x_{I})
\stackrel{\eqref{e:gradasimf/I}}{\sim} \ps{ C_\varrho\gamma(B_{I}^{*})+\varrho^{\alpha} }A_{I}(x_I).
\end{align*}
So for $\varrho$ small enough and $ \gamma(B_{I}^{*})$ small enough depending on $C_{\varrho}$ (and hence on $\varrho$), $A_{I}(y)<\frac{1}{2}A_{I}(x_I)$. Also, as $y\in B_{I}^{*}$, $|y-x_I|\lec N\ell(I)$. This implies that there is $z\in \{A_I=0\}$ so that $|x_I-z|\lec N\ell(I)$ as well, and this implies the claim for $M'$ large enough.

Now let $M\geq M'$. Let $x\in P_{I}\cap MB_{Q}$ and $x'=\pi_{P_{J}}(x)$. Then
\begin{align*}
|\grad A_{J}|\cdot |x-x'|
& =|\grad A_{J}\cdot (x-x')|
=|A_{J}(x)-A_{J}(x')|
= |A_{J}(x)|\\
& =|A_{J}(x)-A_{I}(x)|
\stackrel{\eqref{e:A-A-Linfinity}}{\lec} \gamma(B_{I,J})|\grad A_{I}|M\ell(I)
\stackrel{\eqref{e:gradAI-gradAJ/gradAI}}{\lec}  \gamma(B_{I,J})|\grad A_{J}|M\ell(I)
\end{align*}
Dividing both sides by $|\grad A_{J}|$ gives
\[
\sup_{x\in P_{I}\cap MB_{I}}\dist(x,P_{J})
\lec \gamma(B_{I,J})M\ell(I).
\]
A similar argument with  $I$ and $J$ switched finishes the proof.

\end{proof}

Now we switch to the dyadic setting by assigning to each surface cube a plane. 
%For $Q\in \cD$, let 
%\[
%U_{Q}^{K}=\{I\in \cW: I\cap KB_{Q}\neq\emptyset, \;\; \ell(I)\geq \ell(Q)/K\}
%\]
Let
\[
\gamma(Q)=\ps{\avint_{U_{Q}^{K}}\av{\frac{\grad^{2} f}{f}}^{2} \ell(Q)^{4}}^{\frac{1}{2}},
\;\;\;
\gamma'(Q) = \ps{\avint_{U_{Q}^{2\rho^{-1}K}}\av{\frac{\grad^{2} f}{f}}^{2} \ell(Q)^{4}}^{\frac{1}{2}}.
\]

Let $\delta>0$. To each cube $Q\in \cD$, let $I_Q$ be a Whitney cube of maximal size so that (assuming $\delta \ll c_{0}$)
\begin{equation}
\label{e:I_Qinc0Q}
MB_{I_{Q}}
 \subseteq \frac{\delta}{2} B_{Q} 
 \subseteq \frac{c_{0}}{4} B_{Q}.
\end{equation}
In particular, $\ell(I_{Q})\sim_{M,\delta} \ell(Q)$. 

Let $C_{2}>1$, and suppose $Q,R\in \cD$ are such that 
\[
C_{2}B_{Q}\cap C_{2}B_{R}\neq\emptyset \mbox{ and } \dist(Q,R)\leq C_{2}\min \{\ell(Q),\ell(R)\}.\]

Since $\Omega$ is uniform, there is a Harnack chain of adjacent Whitney cubes $I_{Q}=I_{1},...,I_{n}=I_{R}$ so that $n\lec_{C_{2}} 1$. If we choose $K$ large enough, then 
\[
\bigcup_{i=1}^{n-1} B_{I_{i},I_{i+1}} \subseteq U_{Q}^{K/2}\cap U_{R}^{K/2}.
\]

In particular, applying Lemma \ref{l:Pangles} to each pair $(I_{j},I_{j+1})$ gives that, for $C_1$ large enough

\begin{equation}
\label{e:almost_epsilon}
\ell(Q)^{-1}\ps{\sup_{x\in P_{Q}\cap C_{1} B_{Q}}\dist(x,P_{R}) + \sup_{x\in P_{R}\cap C_{1} B_{Q}}\dist(x,P_{Q})}
\lec_{C_{1},C_{2},K} \gamma(Q).
\end{equation}

%
%Let $U_{Q,i}^{M}$ denote the connected components of $U_{Q}^{M}$ that contain a cube $I_{i}$ for which $\ell(I_{j})\leq \ell(Q)<2\ell(I_{j})$. Let $B_{Q,j}$ denote the ball centered at $x_{I_{j}}$ with radius 
%

%
%Let $U(Q_0)=\bigcup_{Q\subseteq Q_{0}} U_{Q}$. Since the $U_{Q}$ have bounded overlap, it follows that 
%\[
%\sum_{Q\subseteq Q_{0}}\gamma(Q)^2\ell(Q)^{d}
%\lec \int_{U(Q_{0})} \av{\frac{\grad^{2} f(x)}{f(x)}}^{2}\delta_{\Omega}(x)^{3}dx.
%\]
\def\Tree{{\rm Tree}}
\def\Stop{{\rm Stop}}
\def\BD{{\rm BD}}
\def\BG{\rm B\gamma}

\def\wt{\widetilde}

\begin{lemma}
\label{l:fml}
Let $C_{1}>1$, $\delta>0$. For $K,\ve^{-1}$ large enough depending on $\delta$ if $\gamma(Q)<\ve$, then   \eqref{e:C1beta} holds with $E=\d\Omega$ and $\delta$ in place of $\ve$. 
\end{lemma}

\begin{proof}
Let $x\in 2C_{1} B_{Q}\cap \d\Omega$. Then there is $y\in \d U_{Q}^{K}\cap 3C_{1}B_{Q}$ so that $|x-y|<\frac{1}{K} \ell(Q)$. Let $I\subseteq  U_{Q}^{K/2}$ be the Whitney cube containing $y$. Then for $K$ large, $B_{I}^{*}\subseteq U_{Q}^{K}$, and so
\begin{align*}
|A_{Q}(y)|
& =\av{\avint_{B_I} A_{Q}}
\leq \avint_{B_{I}}( |A_{Q}-A_{B_{I}}|  +|A_{B_{I}}-f|)  |+f)\\
& 
\stackrel{\eqref{e:f-a<gamma} \atop \eqref{e:fholder}}{\lec} \gamma(U_{Q}^{K}) K |\grad A_{Q}|\ell(Q) + \gamma(B_{I}) f_{B_{I}} + f(x_{I})K^{-\alpha} \\
& \lec  (K\ve+K^{-\alpha})|\grad A_{Q}|\ell(Q).
\end{align*}

Thus,
\begin{align*}
\dist(x,P_{Q})
& \leq |x-y|+\dist(y,P_{Q})
<\frac{1}{K}\ell(Q) + \frac{|\grad A_{Q}(y-\pi_{P_{Q}}(y))|}{|\grad A_{Q}|}\\
& = \frac{1}{K}\ell(Q) + \frac{| A_{Q}(y)-A_{Q}(\pi_{P_{Q}}(y))|}{|\grad A_{Q}|}\\
& \lec  \frac{1}{K}\ell(Q) + \frac{(K\ve+K^{-\alpha})|\grad A_{Q}|\ell(Q)+0}{|\grad A_{Q}|}\\
& = \frac{1}{K}\ell(Q) + (K\ve +K^{-\alpha})\ell(Q).
\end{align*}
For $\ve,K^{-1}$ small enough depending on $\delta$ (and $\ve$ depending on $K$), this proves the lemma. 
\end{proof}

Let $k_0\in \mathbb{N}$. 
%Let $\wt\Top(k_0)$, $\wt\Tree(R)$ and $\wt\Stop(R)$ be the cubes obtained from Lemma \ref{l:corona}. 
Let $\BG$ (``big $\gamma$") denote those cubes $Q\in \cD(k_0)$ for which $\gamma(Q)\geq \ve$. For $R\in \BG$, let $\Stop(R)=\{R\}$ and $\Next(R)$ be the children of $R$. 
\def\BJ{{\rm BJ}}
For $R\in \cD(k_0)\backslash \BG$, let $\Stop(R)$ be the maximal cubes $Q\subseteq R$ which have a child $Q'$ such that 
\[
\sum_{Q'\subseteq T\subseteq R} \gamma(T)^{2} \geq \ve^2.
\]
Since $U_{Q'}^{K}\subseteq U_{Q}^{2\rho^{-1}K}$, we have that
\begin{equation}
\label{e:gamma<gamma'}
\gamma(Q')\lec \gamma'(Q).
\end{equation}

We claim that this implies 
\begin{equation}
\label{e:gammasumsimepsilon}
\sum_{Q\subseteq T\subseteq R} \gamma'(T)^{2} \gec \ve^2.
\end{equation}
Indeed, if $\gamma'(Q)\geq \ve/C$, then this follows since $Q\in \Tree(R)$ implies
\[
\ve^{2} /C^2 \leq \gamma'(Q)^2 \leq  \sum_{Q\subseteq T\subseteq R} \gamma'(T)^{2}
\]
Otherwise, if $\gamma'(Q)< \ve/C$, then by \eqref{e:gamma<gamma'}, $\gamma(Q')\lec \ve/C$, and so 
\begin{align*}
\ve^{2}
& \leq \sum_{Q'\subseteq T\subseteq R} \gamma(T)^{2}
\lec \gamma(Q')+ \sum_{Q\subseteq T\subseteq R} \gamma(T)^{2} 
\lec \ve^{2}/C^2 + \sum_{Q\subseteq T\subseteq R} \gamma'(T)^{2} 
%=\gamma(Q')^2 + \sum_{Q\subseteq T\subseteq R} \gamma(T)^{2}
%\lec \gamma'(Q)^2 + \sum_{Q\subseteq T\subseteq R} \gamma'(T)^{2}\\
%& \leq \ve^2/C+ \sum_{Q\subseteq T\subseteq R} \gamma'(T)^{2}
\end{align*}
and this implies \eqref{e:gammasumsimepsilon} for $C>$ large enough.

\def\NR{{\rm NR}}
Let $\Tree(R)$ denote those cubes in $\cD(k_0)$ be those cubes in $R$ not properly contained in a cube from $\Stop(R)$. Let $\Top_{0}=\{Q_{0}\}$, $\Top_{k+1} =\bigcup_{R\in \Top_{k}}\Next(R)$, and $\Top(k_0)=\bigcup_{k\geq 0} \Top_{k}$. %Observe that by this construction, for each $R\in \Top(k_0)$ there is $R'\in \wt\Top(k_0)$ so that $\Tree(R)\subseteq \Tree(R')$. Let $E_{R}=E(\Tree(R))$ be the Ahlfors regular set given for this subtree by Lemma \ref{l:corona}.
Observe that if we define $\ve(Q)$ as in Lemma \ref{l:DT-dyadic}, then for $R\in \Top$ and $Q\in \Tree(R)$, 
\[
\ve(Q)\stackrel{\eqref{e:almost_epsilon} \atop \eqref{e:gamma<gamma'}}{\lec} \gamma'(Q).
\]
In particular, these observations and \eqref{e:I_Qinc0Q} imply we can apply Lemma \ref{l:DT-dyadic} (with $\delta$ in place of $\ve$ if we pick $\ve\ll \delta$) to each stopping-time region $\Tree(R)$ to obtain a bi-Lipschitz surface $\Sigma_{R}$ which, by Lemma \ref{l:Pangles} and \eqref{e:I_Qinc0Q} satisfies \eqref{e:dtjones}. In particular, by \eqref{e:close-to-P_Q}, for $\delta>0$ small,
\[
\frac{c_{0}}{2} B_{Q}\cap \Sigma_{R}\neq\emptyset.
\]

Since $\Sigma_{R}$ is Ahlfors regular, and because the balls $\{c_0 B_{Q}:Q\in \Stop(R)\}$ are disjoint by Theorem \ref{t:Christ}, we have that for all $R\in \Top(k_0)$,
\begin{align*}
\sum_{Q\in \Stop(R)}\ell(Q)^{d}
& \sim \sum_{Q\in \Stop(R)}  \cH^{d}(\Sigma_{R}\cap c_{0}B_{Q})\\
& \sim \ve^{-2}\sum_{Q\in \Stop(R)} \sum_{Q\subseteq T\subseteq R} \gamma'(T)^{2}  \cH^{d}(\Sigma_{R}\cap c_{0}B_{Q})\\
& = \ve^{-2}\sum_{ T\in\Tree(R)} \gamma'(T)^{2} \sum_{Q\in \Stop(R) \atop Q\subseteq T}   \cH^{d}(\Sigma_{R}\cap c_{0}B_{Q})\\ 
& \leq \ve^{-2} \sum_{ T\in\Tree(R)} \gamma'(T)^{2}  \cH^{d}(\Sigma_{R}\cap B_{T}) \\
& \lec\ve^{-2} \sum_{ T\in\Tree(R)} \gamma'(T)^{2}\ell(T)^{d}.
\end{align*}

Thus, for $k>0$,
\begin{align*}
\sum_{R\in \Top_{k}\backslash \BG}\ell(R)^{d}
& =\sum_{R\in \Top_{k-1}} \sum_{Q\in \Next(R)}\ell(Q)^{d} 
\lec \sum_{R\in \Top_{k-1}} \sum_{Q\in \Stop(R)}\ell(Q)^{d} \\
& =\sum_{R\in \Top_{k-1}} \sum_{ T\in\Tree(R)} \gamma'(T)^{2}\ell(T)^{d}
\end{align*}

Therefore, since the $U_{Q}^{2\rho^{-1} K}$ have bounded overlap, 
\begin{multline*}
\sum_{R\in \Top(k_{0})}\ell(R)^{d} 
 =\ell(Q_{0})^{d}+\sum_{k=1}^{\infty} \sum_{R\in \Top_{k}} \ell(R)^{d} 
\\
 \lec \ell(Q_{0})^{d} + \sum_{k=0}^{\infty} \sum_{R\in \Top_{k-1}}\sum_{ T\in\Tree(R)} \gamma(T)^{2}\ell(T)^{d} +  \sum_{k=1}^{\infty} \sum_{R\in \Top_{k}\cap \BG}\ell(R)^{d}\\
 \lec \ell(Q_{0})^{d} + \ve^{-2} \sum_{T\subseteq Q_{0}}  \gamma(T)^{2}\ell(T)^{d}  + \ve^{-2} \sum_{R\in \BG} \gamma'(Q)^{2} \ell(Q)^{d}\\
\lec  \ve^{-2} \sum_{T\subseteq Q_{0}}  \gamma'(T)^{2}\ell(T)^{d}
\lec_{\ve} \int_{\Omega(Q_{0})}\av{\frac{\grad^{2} f(x)}{f(x)}}^{2} \delta_{\Omega}(x)^{3} dx  .
\end{multline*}

Theorem \ref{t:affinedeviation} now follows from the following lemma.

%
%\begin{lemma}[{\cite{AS18}, Lemma}] \label{l:AS-beta}
%Let $1 \leq p < \infty$ and $E_1$, $E_2$ lower content $d$-regular subsets of $\R^n$; let moreover $x \in E_1$ and choose a radius $r>0$. Then if $y \in E_2$ is so that $B(x,r) \subset B(y, 2r)$, we have
%\begin{align}
%    \beta_{E_1}^{p,d} \lesssim \beta_{E_2}^{p,d}(y, 2r) + \ps{\frac{1}{r^d} \int_{E_1 \cap B(x,2r)} \ps{\frac{\dist(y, E_2)}{r}}^p \, d \cH^{d}(y)}^{\frac{1}{p}}. 
%\end{align}
%\end{lemma}
%
%Let $Q\in \Tree(R)$ for $R\in \Top(k_0)$. Then
%\[
%\beta_{\d\Omega}(2B_{Q})^2
%\lec \beta_{\d\Omega}(4B_{Q})^2 + 
%\frac{1}{\ell(Q)^{d}} \int_{\Sigma_{R} \cap 4B_{Q}} \ps{\frac{\dist(y, \Sigma_{R})}{r}}^2 \, d \cH^{d}(y).
%\]

\begin{lemma}
Suppose that $E$ is a lower $d$-regular set, $\cD$ the Christ-David cubes for $E$, and $Q_0\in \cD$, and for each $k_0\in\mathbb{N}$, let $\Top(k_0)$ be a collection of cubes such that, for each $R\in \Top(k_0)$, there is a stopping-time region $\Tree(R)$ whose top cube is $R$ and a uniformly rectifiable set $\Sigma_{R}$ so that for $x\in E\cap C_{1} B_{R}$, 
\begin{equation}
\label{e:AV-d<d}
\dist(x,\Sigma_{R})\leq \delta  d_{R}(x):=\delta\inf_{Q\in \Tree(R)} (\ell(Q)+\dist(x,Q)),
\end{equation}
For $C_{1}$ large enough and $\delta$ small enough,
\[
 \ell(Q_{0})^{d} + \beta_{E}(Q_{0})\lec  \limsup_{k_{0}\rightarrow\infty}\sum_{R\in \Top(k_{0})}\ell(R)^{d}.
 \]
\end{lemma}

The proof is exactly as that of  \cite[Lemma 4.6]{AV19}. \\

Thus, to finish the proof of Theorem \ref{t:affinedeviation}, we just need to verify \eqref{e:AV-d<d}. Let $R\in \Top(k_0)$ and $x\in C_{1} B_{R}\cap \d\Omega$. Let $Q\in \Tree(R)$ be so that 
\[
d_{R}(x) = \ell(Q)+\dist(x,Q). 
\]
Let $\hat{Q}$ be the largest ancestor of $x$ so that $\ell(\hat{Q})<d_{R}(x)$. Then $\ell(\hat{Q})\sim d_{R}(x)>\dist(x,\hat{Q})$, and so for $C_{1}$ large enough, $x\in C_{1} B_{\hat{Q}}$. By \eqref{e:close-to-P_Q} (which holds by Lemma \ref{l:DT-dyadic} and Lemma \ref{l:fml}),
\[
\dist(x,\Sigma_{R})\lec \delta \ell(\hat{Q})\lec \ve d_{R}(x) .
\]

This finishes the proof.

\appendix

\section{David-Reifenberg-Toro Parametrizations}

The rest of this section is dedicated to the proof of Theorem \ref{l:DT-dyadic}. We will frequently cite equations and terminology from \cite{DT12} rather than stating their result here, since it is quite long to state. 

Without loss of generality, we can assume $Q(S)\in \cD_{0}$, $\ell(Q(S))=1$, and $P_{Q(S)}=\R^{d}$. 

Let $r_{k}=10^{-k}$,  and let $X_{k}$ be a maximal $\frac{3}{2}r_{k}$-separated set in 
\[
\Xi_{k}=\{\zeta_{Q}:Q\in S_{k}:=\cD_{s(k)}\cap S\}
\]
where $s(k)$ is such that
\begin{equation}
\label{e:rbetween}
5\rho^{s(k)} \leq r_{k}/4 < 5\rho^{s(k)-1}. 
\end{equation}
Note that this integer is unique and $s(0)=0$. Recall that $\ell(Q)=5\rho^{k}$ if $Q\in \cD_{k}$. 

For $x\in X_{k}$, let $Q_k(x)\in S_{k}$ be so that $\zeta_{Q_{k}(x)}=x$. Also let $x'=\pi_{P_{Q_{k}(x)}}(x)$, $P_{k}(x'):= P_{Q_{k}(x)}$, and 
\[
X_{k}'=\{x':x\in X_{k}\}. \]

Note that $X_{0}$ is a singleton (just the center of $Q(S)$), and thus so is $X_{0}'$. Without loss of generality, we can assume $X_{0}'=\{0\}$.

Enumerate $X_{k}'=\{x_{j,k}\}_{j\in J_{k}}$ and if $x_{j,k}=x'$, set $P_{j,k}=P_{k}(x')$ and $B_{j,k} = B(x_{j,k},r_{k})$. Since $X_0'$ is a singleton, $X_{0}'=\{x_{0,0}\}$ and $P_{0,0}=P_{Q(S)}$.

We will show that this collection of points $X_{k}'$ and planes $P_{j,k}$ satisfy the conditions for David and Toro's theorem \cite[Theorem 2.5]{DT12}. First, we show that they form a coherent collection of balls and planes, which requires verifying equations (2.1) and (2.3-2.10) in \cite{DT12} (see \cite[Definition 2.1]{DT12}).

By assumption \eqref{e:distQtoP_Q},
\[
|x-x'|<\ve \ell(Q_{k}(x))=5\ve \rho^{s(k)}.
\]
Then for all $x',y'\in X_{k}'$ and $\ve$ small,
\begin{equation}
\label{e:DT-separated}
|x'-y'|\geq |x-y|-10\ve \rho^{s(k)}
\stackrel{\eqref{e:rbetween}}{>}\frac{3}{2}r_{k}-\frac{\ve r_{k}}{2}> r_{k}. 
\end{equation}
Thus, $X_{k}'$ is an $r_{k}$ separated set for each $k\geq 0$ (which is \cite[Equation (2.1)]{DT12}). 

Next, we want to show \cite[Equation (2.3)]{DT12}, i.e. that for all $k>0$,

\begin{equation}
\label{e:Vk}
X_{k}'\subseteq V_{k-1}^{2} \mbox{ where } V_{k}^{2}:= \{x:\dist(x,X_{k}')\leq 2r_{k})\}. 
\end{equation}

Let $x'\in X_{k}'$, recall that there is  a corresponding $x\in X_{k}$ and $|x-x'|\leq \ve \ell(Q)<\rho r_{k}$ for $\ve$ small enough. 

{\bf Case 1:} If $r_{k-1}/4 < 5\rho^{s(k)-1}$, then $s(k-1)=s(k)$, so $\Xi_{k-1}=\Xi_{k}$. Since $X_{k-1}$ is a maximally $\frac{3}{2}r_{k-1}$-separated set in $\Xi_{k-1}$ (and hence in $\Xi_{k}$), there is $y\in X_{k-1}$ so that $|x-y|<\frac{3}{2} r_{k-1}$. Thus,
\[
|x'-y'|
\leq |x-y| + |x'-x|+|y'-y|
\leq \frac{3}{2} r_{k-1} + 2\rho r_{k-1} < 2r_{k-1}
\]
and so $x'\in V_{k}^{2}$. 

{\bf Case 2:} If $r_{k-1}/2\geq 5\rho^{s(k)-1}$, then $s(k-1)=s(k)-1$. Let $R\in S_{k-1}$ be the parent of $Q_{k}(x)\in S_{k}$ (so $\ell(R)= 5\rho^{s(k-1)} \leq r_{k-1}/4$). Then there is $y\in X_{k-1}$ so that $|\zeta_{R}-y|<\frac{3}{2} r_{k-1}$. Thus, for $\ve$ small enough,
\begin{align*}
|x'-y'|
& \leq |x'-x|+|x-\zeta_{R}|+|\zeta_{R}-y|+|y-y'|\\
& \leq \ve\ell(Q)+\ell(R)+\frac{3}{2}r_{k-1}+\ve \ell(R)\\
& \leq (2\ve \rho+1)\ell(R)+\frac{3}{2}r_{k-1}\\
& \leq  \ps{2\ve \rho+1}r_{k-1}/4+\frac{3}{2}r_{k-1} <2r_{k-1}. 
\end{align*}
These two cases prove \eqref{e:Vk}. 

Set $\Sigma_0=P_{Q(S)}$. Then equations (2.4-2.7) in \cite{DT12} are trivially satisfied (note in particular that as $s(0)=0$, $X_0=\{\zeta_{Q(S)}\}$, and so $X_{0}'=\{\zeta_{Q(S)}'\}\subseteq P_{Q(S)}=\Sigma_0$).

For $y\in V_{k}^{11}$, let 
\[
\ve_{k}''(y) = \sup \{d_{x_{i,\ell},100 r_{\ell}}(P_{j,k},P_{i,\ell}) : j\in J_{k},\ell\in \{k-1,k\}, i\in J_{\ell}, \mbox{ and }y\in 11 B_{j,k} 12 B_{i,\ell}\}
\]
and let $\ve_{k}''(y) =0$ otherwise. 

Now $\ve_{k}''(x_{j,k}) \lec \ve$ by \eqref{e:dtjones} for $C_{1}$ and $C_{2}$ large enough, which implies \cite[Equations (2.8-2.10)]{DT12}, and so $(\Sigma_0,X_k',\{B_{j,k}\},\{P_{j,k}\})$ are a coherent collection of balls and planes. Hence, \cite[Theorem 2.4]{DT12} implies there is $g_{S}:\R^{n}\rightarrow \R^{n}$ so that $\Sigma_{S}=g_{S}(\R^{d})$ is a $(1+C\ve)$-Reifenberg flat surface and so that \eqref{e:gz-z} holds (recall $B_{Q(S)}=\bB$) by \cite[Equation 2.13]{DT12}.

Now we wish to show \cite[Equation (2.19)]{DT12}, that is,
\begin{equation}
\label{e:jones-claim}
\sum_{k\geq 0} \ve_{k}(g_{S}(z))^{2} \lec \ve^2 \;\; \mbox{ for all }z\in \R^{d}
\end{equation}
since then by \cite[Theorem 2.5]{DT12}, $g_{S}$ is bi-Lipschitz and (after an examination of the details in \cite[Chapter 8]{DT12}), $g_{R}|_{\bR^{d}}$ is $(1+C\ve^2)$-bi-Lipschitz. 

Let $y=g_S(z)$ where $z\in \R^{d}$. Note that if $y\in  (V_{0}^{11})^{c}$, then there are only finitely many $k$ for which $\ve_{k}''(y)\neq 0$, so \eqref{e:jones-claim} holds trivially in this case since $\ve_{k}''\lec C\ve$ whenever it is nonzero. Hence, we can assume $y\in V_{0}^{11}$ by \eqref{e:gz-z}.

Let $k(y)$ be the largest integer for which $y\in V_{k}^{11}$ (so $k(y)\geq 0$ by the previous paragraph), and if such an integer does not exist, let $k(y)=\infty$. Then $\dist(y,X_{k}')\leq 11r_{k}$ for all $0\leq k\leq k(y)$. 

Let $x'\in X_{k}'$ be closest to $y$, $Q=Q_{k}(x')$. Then it is not hard to show
\[
\ve_{j}(y)\lec \ve(R) \;\; \mbox{ for all } j\leq k \mbox{ and }Q\subseteq R\in S_{j}.
\]
Thus, 
\[
\sum_{j=0}^{\infty}\ve_{j}(y)^{2} 
= \sum_{j=0}^{k(y)} \ve_{j}(z)
\lec \sum_{Q\subseteq R\subseteq Q(S)} \ve(R)^{2} <\ve^{2}
\]
and this proves \eqref{e:jones-claim}.

Suppose \eqref{e:C1beta} holds, we will show this implies \eqref{e:close-to-P_Q}. Let $z\in E\cap C_{1}B_{Q}$ for some $Q\in S$. Assume $Q\in S_{k}$ with $k\geq k_0$ where $k_0\geq 0$ we will decide shortly. 

If $\pi_{Q}$ is the projection into $P_{Q}$ and $z'=\pi_{Q}(z)$, then $|z-z'|\leq \ve C_{1}\ell(Q)$ by \eqref{e:C1beta}, so $z'\in 2C_{1} B_{Q}$. 

Let $k=k-k_0\geq 0$, $Q'\in S_{k'}$ contain $Q$, let $x\in \Xi_{k'}$ be so that $|x-\zeta_{Q'}|<\frac{3}{2}r_{k'}$, and let $x_{j,k'}=x'\in X_{k'}'$. Then for $C_{2}$ large enough depending on $k_0$, $Q_{k'}(x)\sim Q$ since $\ell(Q)\sim_{k_{0}}\ell(Q(x))$ and 
\[
\dist(Q,Q_{k'}(x))
\leq \ell(Q')+|\zeta_{Q'}-x|
\leq r_{k-k_{0}}/4 + \frac{3}{2}r_{k'}
\lec_{k_{0}} r_{k}\lec \ell(Q).
\]

Since $z'\in 2C_{1}B_{Q}\cap P_{Q}\subseteq C_{1} B_{Q'}$ and $\ve(Q)<\ve$, this means there is $z''\in P_{Q_{k'}(x)}=P_{j,k'}$ so that $|z'-z''|< C_{1}\ve \ell(Q_{k'}(x))$. Moreover,
\[
|z-z''|\leq |z-z'|+|z'-z''|\leq 2C_{1}\ve \ell(Q_{k'}(x)).
\]
Also note that since $z\in C_{1} B_{Q}\subseteq \frac{4}{3}B_{Q_{k'}(x)}$ for $k_0$ large enough, and since $Q_{k'}(x),Q'\in S_{k'}$,
\begin{align*}
|z''-x_{j,k}|
& = |z''-x'|
\leq |z''-z|+|z-\zeta_{Q'}|+|\zeta_{Q'}-x|+|x-x'|\\
& \leq 2C_{1}\ve \ell(Q_{k'}(x))+\frac{4}{3}\ell(Q') + \frac{3}{2}r_{k'}+\ve \ell(Q')\\
& \leq \ps{2C_{1}\ve + \frac{1}{3} +\frac{3}{2} + \frac{\ve}{4}}r_{k'}
<2r_{k'}
\end{align*}
if $\ve\ll C_{1}^{-1}$. Thus, $z''\in 2B_{j,k'}\cap P_{j,k'}$. 

%
%\leq C_{1}\ell(Q) + \ell(Q') + |z-x|
%\leq |z-\zeta_{Q}|+|\zeta_{Q}-x|
%\leq C_{0}\ell(Q) + \frac{3}{2}r_{k}
%< C_{1}\ell(Q) +\frac{6}{\rho}\ell(Q)<C_{2}\ell(Q)
%\]
%if $C_{2}>C_{1}+6/\rho$. Thus, since $\ve(Q)<\ve$, there is $z''\in P_{Q(x)}=P_{j,k}$ with $|z''-z'|\leq \ve C_{1}\ell(Q)$. Thus, for $\ve'$ small enough, 
%%\[
%%|x'-z|
%%\leq |x'-x|+|x-z|
%%< \ve \ell(Q)+ C_{0}\ell(Q) + \frac{3}{2}r_{k}
%
%
%%<\frac{4}{3}r_{k} \leq 2C\ell(Q)\]
%
%
%Then for $C\gg C_{0},C_{1}$,%$|x-x'|<\epsilon C\ell(Q)$, and so 
%\begin{align*}
%|z-x'| 
%& \leq |z-\zeta_{Q}|+|\zeta_{Q}-x|+|x-x'|
%\leq C_1\ell(Q)+\frac{3}{2}r_{k}+C_0\ell(Q)\\
%& <\ps{\frac{C_1}{C\rho}+\frac{3}{2}+\frac{C_{0}}{C\rho}}r_{k}<\frac{4}{3} r_{k}. 
%\end{align*}
%Let $j\in J_{k}$ be so that $x'=x_{j,k}$. Then the above implies $z\in 2B_{j,k}$. 
%
%Now observe that 
%
%and $\ell(Q)=\ell(Q_{k})$. Thus, $Q\sim Q_{k}(x)$ for $C_{2}\gg C$, also, if $z'=\pi_{Q}(z)$, then by \eqref{e:C1beta}, $|z-z'|\leq 2C_{1}\delta  \ell(Q)$, so for $\delta$ small enough, $z'\in $
%
%, then since $P_{j,k}$ passes through the center of $2B_{j,k}$, this means $z'\in 2B_{j,k}$ as well, 
%
%Since $\ve(Q)<\ve$ and $z\in E\cap C_{1}B_{Q}$,
%\[
%\dist(z,P_{j,k})
%=\dist(z,P_{Q_{k}(x)})
%\lec_{C_{1}} \ve \ell(Q) + \dist(z,P_{Q})\lec (\ve+\delta) \ell(Q).
%\]
%For $\ve>0$ small enough, $\pi_{j,k}(z)\in 3B_{j,k}$. 

By \cite[Equation (5.3)]{DT12}, $\dist(z'',\Sigma_{k})\lec \ve r_{k'}\sim \ve \ell(Q)$. Let $w\in \Sigma_{k}$ be closest to $z''$. By \cite[Equation (5.11)]{DT12}, for any $\ell\geq 0$,
\[
|\sigma_{\ell}(y)-y|\lec \ve r_{\ell} \;\; \mbox{ for all }y\in \Sigma_{\ell}
\]
where $\sigma_{\ell}$ is defined in \cite[Equation (4.2)]{DT12} and $\Sigma_{\ell}$ in  \cite[Equation (5.1)]{DT12}. Note that $\Sigma_{S}$ is defined in \cite[Equations (6.1-6.2)]{DT12}. Iterating this from $\ell=k'$ with $y=w$ and taking the limit as $\ell\rightarrow \infty$,  we get that 
\[
\dist(w,\Sigma_{S})\lec \ve r_{k'}\sim_{k_{0}} \ve \ell(Q).
\]
Thus,
\[
\dist(z,\Sigma_{S})
\leq |z-z''| + |z'' - w| + \dist(w,\Sigma_{S})
\lec (C_1\ve + \ve)\ell(Q) + \ve \ell(Q) + \ve\ell(Q)\lec \ve \ell(Q).
\]
This proves \eqref{e:close-to-P_Q} when $k\geq k_0$. If $k=k_0$, then $Q\sim Q(S)$, so if $z\in E\cap C_{1}B_{Q}$, then $z'\in 2C_{1}B_{Q}\cap P_{Q}$ just as before, and \eqref{e:C1beta} implies $|z-z'|<2C_{1}\ve \ell(Q)$. Since $\ve(Q)<\ve$, $|z'-\pi_{P(Q)}(z')|\lec \ve \ell(Q(S))\sim \ve\ell(Q)$, and \eqref{e:gz-z} implies 
\[
\dist(\pi_{P(Q)}(z'),\Sigma_{R})\leq |\pi_{P(Q)}(z')-g_{S}(\pi_{P(Q)}(z'))|\lec \ve \sim \ve\ell(Q).
\]
Thus,
\begin{align*}
\dist(z,\Sigma_{S})
& \leq |z-z'|+|z'-\pi_{P(Q)}(z')|+\dist(\pi_{P(Q)}(z'),\Sigma_{R})\\
& \lec \ve \ell(Q).
\end{align*}

This proves \eqref{e:close-to-P_Q} in every case.\\

Now suppose \eqref{e:C2beta} holds, we will prove \eqref{e:close-to-E}. Let $z\in C_{1}B_{Q}\cap \Sigma_{R}$ where $Q\in S_k$. Then there is $z_0\in \R^{d}$ so that if $z_{k} = \sigma_k(z_{k-1})$, then $z_{k}\in \Sigma_k$ and $z_{k}\rightarrow z$. 

Let $k_0$, $k'$, $Q'$, $x$, and $Q_{k'}(x)\in S_{k'}$ be chosen just as before (assume $k\geq k_0$), so again $Q\sim Q_{k'}(x)$ for $C_2$ large enough. Again,  $z\in C_{1} B_{Q}\subseteq \frac{4}{3}B_{Q_{k'}(x)}$ for $k_0$ large enough, so 
\begin{align*}
|z_{k'}-x'|
& \leq |z_{k'}-z|+|z-\zeta_{Q'}|+|\zeta_{Q'}-x|+|x-x'|\\
& \ve r_{k'}+ \frac{4}{3}\ell(Q') + \frac{3}{2}r_{k'} + \ve \ell(Q')\\
& <2r_{k'}
\end{align*}
Hence, $z_{k'}\in B_{j,k'}$ where $j$ is so that $x_{j,k'}=x'$. 
%Again, by \cite[Equation (5.11)]{DT12}, $|z_{k'}-z|\lec \ve r_{k'}$ for all $k\geq 0$. The same arguments above imply that, if $Q\in S_{k}$, then $z\in 2B_{j,k}$, and so $z_{k}\in 3B_{j,k}$ where $x_{j,k}=x'$ for some $x\in \Xi_k$ with $Q_{k}(x')\sim Q$.
By  \cite[Equation (5.6)]{DT12}, 
\[
|z_{k'}-\pi_{j,k'}(z_{k'})|=\dist(z_{k'},P_{j,k'})\lec \ve r_{k'} \sim_{k_{0}} \ve \ell(Q),\] 
and by \cite[Equation (5.11)]{DT12}, 
\[
|z_{k'}-z|\lec \ve r_{k'}\sim_{k_{0}} \ve\ell(Q).
\]
Since $P_{j,k'}$ passes through the center of $2B_{j,k'}$ and $z_{k'}\in B_{j,k'}$, we have 
\begin{align*}
\pi_{j,k'}(z_{k'})\in 2B_{j,k'}
& =B(x',2r_{k'})
 \subseteq B(x,2r_{k'}+\ve \ell(Q_{k'}(x)))\\
& \stackrel{\eqref{e:rbetween}}{ \subseteq} B(x,(8\rho+\ve)\ell(Q_{k'}(x))
\subseteq C_{1} B_{Q_{k'}(x)}
\end{align*}
for $C_{1}$ large enough. Thus, by \eqref{e:C2beta},
\[
\dist(\pi_{j,k'}(z_{k'}),E)\lec \ve \ell(Q_{k'}(x))\lec_{k_{0}} \ve \ell(Q).\]
Thus, 
\[
\dist(z,E)
\leq |z-z_{k'}|+|z_{k'}-\pi_{j,k'}(z_{k'})|+\dist(\pi_{j,k'}(z_{k'}),E)
\lec_{k_{0}} \ve \ell(Q).
\]

Now assume  $0\leq k\leq k_0$.  Note that by \eqref{e:gz-z}, if $y=g_{R}^{-1}(z)\in P_{Q(S)}$, since $z\in C_{1}B_{Q}\subseteq C_{1}B_{Q(S)}$, $y\in 2C_{1} B_{Q(S)}$ for $\ve$ small, so \eqref{e:C2beta} implies $\dist(y,E)\lec \ve \ell(Q(S))\sim \ve\ell(Q)$. Thus, by \eqref{e:gz-z} again,
\[
\dist(z,E)
\leq |z-y|+\dist(y,E)
\lec \ve + \ve \ell(Q(S))\sim \ve\ell(Q).
\]

% 
% Since $z\in C_{1}B_{Q}\cap \Sigma_{R}$, for $\ve>0$ small enough, 
%
%
%Thus, for $\ve>0$ small, $\pi_{j,k'}(z)\in P_{j,k'}\cap 4B_{j,k'}$. In particular, since $Q_{k'}(x)\sim Q$ and $\ve(Q)<\ve$, $\dist(\pi_{j,k'}(z),P_{Q})< C_1\ve \ell(Q)$. Since $z\in C_{1}B_{Q}\cap \Sigma_{R}$, for $\ve>0$ small enough, $\pi_{P_{Q}}(\pi_{j,k}(z)) \in 2C_{1} B_{Q}\subseteq\frac{4}{3}B_{Q_{k'}(x)} $ for $k_0$ large, 
%
%
%, so \eqref{e:C2beta} implies $\dist(\pi_{P_{Q}}(\pi_{j,k}(z)),E)\lec \ve \ell(Q)$. Thus,
%\begin{align*}
%\dist(z,E)
%& \leq |z-\pi_{j,k}(z)|+|\pi_{j,k}(z)-\pi_{P_{Q}}(\pi_{j,k}(z))|+\dist(\pi_{P_{Q}}(\pi_{j,k}(z)),E)\\
%& \lec \ve r_{k}+ \ve \ell(Q)\sim \ve\ell(Q).
%\end{align*}

This finishes the proof of Lemma \ref{l:DT-dyadic}.

\newcommand{\etalchar}[1]{$^{#1}$}
\def\cprime{$'$}


\begin{thebibliography}{AHM{\etalchar{+}}17}

\bibitem[AAM19]{AAM19}
M.~Akman, J.~Azzam, and M.~Mourgoglou.
\newblock Absolute continuity of harmonic measure for domains with lower
  regular boundaries.
\newblock {\em Adv. Math.}, 345:1206--1252, 2019.

\bibitem[ABHM16]{ABHM16}
M.~Akman, S.~Bortz, S.~Hofmann, and J.M. Martell.
\newblock Rectifiability, interior approximation and harmonic measure.
\newblock {\em arXiv preprint arXiv:1601.08251, to appear in Ark. Mat}, 2016.


\bibitem[Aik06]{Aik06}
H.~Aikawa.
\newblock Characterization of a uniform domain by the boundary {H}arnack
  principle.
\newblock In {\em Harmonic analysis and its applications}, pages 1--17.
  Yokohama Publ., Yokohama, 2006.

\bibitem[Aik08]{Aik08}
H.~Aikawa.
\newblock Equivalence between the boundary {H}arnack principle and the
  {C}arleson estimate.
\newblock {\em Math. Scand.}, 103(1):61--76, 2008.

\bibitem[AH08]{AH08}
H.~Aikawa and K.~Hirata.
\newblock Doubling conditions for harmonic measure in {J}ohn domains.
\newblock {\em Ann. Inst. Fourier (Grenoble)}, 58(2):429--445, 2008.

\bibitem[AG01]{AG}
D.~H. Armitage and S.~J. Gardiner.
\newblock {\em Classical potential theory}.
\newblock Springer Monographs in Mathematics. Springer-Verlag London, Ltd.,
  London, 2001.

\bibitem[Azz17]{Azz17}
J.~Azzam.
\newblock Semi-uniform domains and the $a_{\infty}$ property for harmonic
  measure.
\newblock {\em arXiv preprint arXiv:1711.03088, to appear in IMRN}, 2017.


\bibitem[AHM{\etalchar{+}}16]{AHMMMTV16}
J.~Azzam, S.~Hofmann, J.~M. Martell, S.~Mayboroda, M.~Mourgoglou, X.~Tolsa, and
  A.~Volberg.
\newblock Rectifiability of harmonic measure.
\newblock {\em Geom. Funct. Anal.}, 26(3):703--728, 2016.

\bibitem[AHM{\etalchar{+}}17]{AHMNT17}
J.~Azzam, S.~Hofmann, J.M. Martell, K.~Nystr\"om, and T.~Toro.
\newblock A new characterization of chord-arc domains.
\newblock {\em J. Eur. Math. Soc. (JEMS)}, 19(4):967--981, 2017.


\bibitem[AM18]{AM18}
J.~Azzam and M.~Mourgoglou.
\newblock Tangent measures and absolute continuity of harmonic measure.
\newblock {\em Rev. Mat. Iberoam.}, 34(1):305--330, 2018.

\bibitem[AMT17]{AMT15}
J.~Azzam, M.~Mourgoglou, and X.~Tolsa.
\newblock Singular sets for harmonic measure on locally flat domains with
  locally finite surface measure.
\newblock {\em Int. Math. Res. Not}, (12):3751--3773, 2017.

\bibitem[AHMMT19]{AHMMT19}
J.~Azzam, S. Hofmann, J.M. Martell, M.~Mourgoglou, and X.~Tolsa.
\newblock Harmonic measure and quantitative connectivity: geometric characterization of the $L^p$ solvability of the Dirichlet problem.
\newblock {\em to appear}, 2019.

\bibitem[AS18]{AS18}
J.~Azzam and R.~Schul.
\newblock An analyst's traveling salesman theorem for sets of dimension larger
  than one.
\newblock {\em Math. Ann.}, 370(3-4):1389--1476, 2018.

\bibitem[AT15]{AT15}
J.~Azzam and X.~Tolsa.
\newblock Characterization of {$n$}-rectifiability in terms of {J}ones' square
  function: {P}art {II}.
\newblock {\em Geom. Funct. Anal.}, 25(5):1371--1412, 2015.

\bibitem[ATT18]{ATT18}
J.~Azzam, X.~Tolsa, and T.~Toro.
\newblock Characterization of rectifiable measures in terms of $alpha
  $-numbers.
\newblock {\em arXiv preprint arXiv:1808.07661}, 2018.

\bibitem[AV19]{AV19}
J. Azzam and M. Villa.
\newblock Quantitative comparisons of multiscale geometric properties.
\newblock {\em arXiv preprint arXiv:1905.00101}, 2019.


\bibitem[Bat96]{Bat96}
A.~Batakis.
\newblock Harmonic measure of some {C}antor type sets.
\newblock {\em Ann. Acad. Sci. Fenn. Math.}, 21(2):255--270, 1996.

\bibitem[BJ90]{BJ90}
C.~J. Bishop and P.~W. Jones.
\newblock Harmonic measure and arclength.
\newblock {\em Ann. of Math. (2)}, 132(3):511--547, 1990.

\bibitem[BJ94]{BJ94}
C.~J. Bishop and P.~W. Jones.
\newblock Harmonic measure, {$L^2$} estimates and the {S}chwarzian derivative.
\newblock {\em J. Anal. Math.}, 62:77--113, 1994.

\bibitem[BE17]{BE17}
Simon Bortz and Max Engelstein.
\newblock Reifenberg flatness and oscillation of the unit normal vector.
\newblock {\em arXiv preprint arXiv:1708.05331}, 2017.


\bibitem[Dah77]{Dah77}
B.~E.~J. Dahlberg.
\newblock Estimates of harmonic measure.
\newblock {\em Arch. Rational Mech. Anal.}, 65(3):275--288, 1977.

\bibitem[DEM18]{GEM18}
G.~David, M.~Engelstein, and S.~Mayboroda.
\newblock Square functions, non-tangential limits and harmonic measure in
  co-dimensions larger than one.
\newblock {\em arXiv preprint arXiv:1808.08882}, 2018.

\bibitem[DJ90]{DJ90}
G.~David and D.~S. Jerison.
\newblock Lipschitz approximation to hypersurfaces, harmonic measure, and
  singular integrals.
\newblock {\em Indiana Univ. Math. J.}, 39(3):831--845, 1990.

\bibitem[DS91]{DS}
G.~David and S.~W. Semmes.
\newblock Singular integrals and rectifiable sets in {${\bf R}^n$}: {B}eyond
  {L}ipschitz graphs.
\newblock {\em Ast\'erisque}, (193):152, 1991.

\bibitem[DS93]{of-and-on}
G.~David and S.~W. Semmes.
\newblock {\em Analysis of and on uniformly rectifiable sets}, volume~38 of
  {\em Mathematical Surveys and Monographs}.
\newblock American Mathematical Society, Providence, RI, 1993.

\bibitem[DT12]{DT12}
G.~David and T.~Toro.
\newblock Reifenberg parameterizations for sets with holes.
\newblock {\em Mem. Amer. Math. Soc.}, 215(1012):vi+102, 2012.

\bibitem[Eng16]{Eng16}
M.~Engelstein.
\newblock A two-phase free boundary problem for harmonic measure.
\newblock {\em Ann. Sci. \'Ec. Norm. Sup\'er. (4)}, 49(4):859--905, 2016.


\bibitem[FKP91]{FKP91}
R.~A. Fefferman, C.~E. Kenig, and J.~Pipher.
\newblock The theory of weights and the {D}irichlet problem for elliptic
  equations.
\newblock {\em Ann. of Math. (2)}, 134(1):65--124, 1991.

\bibitem[GKS10]{GKS10}
J.~Garnett, R.~Killip, and R.~Schul.
\newblock A doubling measure on {$\Bbb R^d$} can charge a rectifiable curve.
\newblock {\em Proc. Amer. Math. Soc.}, 138(5):1673--1679, 2010.

\bibitem[GM08]{Harmonic-Measure}
J.~B. Garnett and D.~E. Marshall.
\newblock {\em Harmonic measure}, volume~2 of {\em New Mathematical
  Monographs}.
\newblock Cambridge University Press, Cambridge, 2008.
\newblock Reprint of the 2005 original.

\bibitem[GMT18]{GMT18}
J.~Garnett, M.~Mourgoglou, and X.~Tolsa.
\newblock Uniform rectifiability from {C}arleson measure estimates and
  {$\varepsilon$}-approximability of bounded harmonic functions.
\newblock {\em Duke Math. J.}, 167(8):1473--1524, 2018.

\bibitem[HKM06]{HKM}
J.~Heinonen, T.~Kilpel{\"a}inen, and O.~Martio.
\newblock {\em Nonlinear potential theory of degenerate elliptic equations}.
\newblock Dover Publications, Inc., Mineola, NY, 2006.
\newblock Unabridged republication of the 1993 original.

\bibitem[Hof19]{Hof19}
Steve Hofmann.
\newblock Quantitative {A}bsolute {C}ontinuity of {H}armonic {M}easure and the
  {D}irichlet {P}roblem: {A} {S}urvey of {R}ecent {P}rogress.
\newblock {\em Acta Math. Sin. (Engl. Ser.)}, 35(6):1011--1026, 2019.


\bibitem[HLMN17]{HLMN17}
S.~Hofmann, P.~Le, J.M. Martell, and K.~Nystr\"om.
\newblock The weak-{$A_\infty$} property of harmonic and {$p$}-harmonic
  measures implies uniform rectifiability.
\newblock {\em Anal. PDE}, 10(3):513--558, 2017.

\bibitem[HM14]{HM14}
S.~Hofmann and J.~M. Martell.
\newblock Uniform rectifiability and harmonic measure {I}: {U}niform
  rectifiability implies {P}oisson kernels in {$L^p$}.
\newblock {\em Ann. Sci. \'Ec. Norm. Sup\'er. (4)}, 47(3):577--654, 2014.

\bibitem[HM15]{HM15}
S.~Hofmann and J.~M. Martell.
\newblock Uniform rectifiability and harmonic measure iv: Ahlfors regularity
  plus poisson kernels in $ l^{p}$ implies uniform rectifiability.
\newblock {\em arXiv preprint arXiv:1505.06499}, 2015.

%\bibitem[HM18]{HM18}
%S.~Hofmann and J.M. Martell.
%\newblock On quantitative absolute continuity of harmonic measure and big piece
%  approximation by chord-arc domains.
%\newblock {\em arXiv preprint arXiv:1712.03696}, 2018.

\bibitem[HMM14]{HMM14}
S.~Hofmann, J.~M. Martell, and S.~Mayboroda.
\newblock Uniform rectifiability and harmonic measure {III}: {R}iesz transform
  bounds imply uniform rectifiability of boundaries of 1-sided {NTA} domains.
\newblock {\em Int. Math. Res. Not. IMRN}, (10):2702--2729, 2014.


\bibitem[HMT10]{HMT10}
S.~Hofmann, M.~Mitrea, and M.~Taylor.
\newblock Singular integrals and elliptic boundary problems on regular
  {S}emmes-{K}enig-{T}oro domains.
\newblock {\em Int. Math. Res. Not. IMRN}, (14):2567--2865, 2010.



\bibitem[HMUT14]{HMU14}
S.~Hofmann, J.~M. Martell, and I.~Uriarte-Tuero.
\newblock Uniform rectifiability and harmonic measure, {II}: {P}oisson kernels
  in {$L^p$} imply uniform rectifiability.
\newblock {\em Duke Math. J.}, 163(8):1601--1654, 2014.

\bibitem[Hru84]{Hru84}
S.~V. Hru{\v{s}}{\v{c}}ev.
\newblock A description of weights satisfying the {$A_{\infty }$} condition of
  {M}uckenhoupt.
\newblock {\em Proc. Amer. Math. Soc.}, 90(2):253--257, 1984.

\bibitem[JK82]{JK82}
D.~S. Jerison and C.~E. Kenig.
\newblock Boundary behavior of harmonic functions in nontangentially accessible
  domains.
\newblock {\em Adv. in Math.}, 46(1):80--147, 1982.

\bibitem[Jon90]{Jon90}
P.~W. Jones.
\newblock Rectifiable sets and the traveling salesman problem.
\newblock {\em Invent. Math.}, 102(1):1--15, 1990.

\bibitem[KT97]{KT97}
C.~E. Kenig and T.~Toro.
\newblock Harmonic measure on locally flat domains.
\newblock {\em Duke Math. J.}, 87(3):509--551, 1997.


\bibitem[KT03]{KT03}
C.~E. Kenig and T.~Toro.
\newblock Poisson kernel characterization of {R}eifenberg flat chord arc
  domains.
\newblock {\em Ann. Sci. \'Ecole Norm. Sup. (4)}, 36(3):323--401, 2003.


\bibitem[LV07]{LV07}
J.~L. Lewis and A.~L. Vogel.
\newblock Symmetry theorems and uniform rectifiability.
\newblock {\em Bound. Value Probl.}, pages Art. ID 30190, 59, 2007.

\bibitem[Mat95]{Mattila}
P.~Mattila.
\newblock {\em Geometry of sets and measures in {E}uclidean spaces}, volume~44
  of {\em Cambridge Studies in Advanced Mathematics}.
\newblock Cambridge University Press, Cambridge, 1995.
\newblock Fractals and rectifiability.

\bibitem[MT15]{MT15}
M.~Mourgoglou and X.~Tolsa.
\newblock Harmonic measure and riesz transform in uniform and general domains.
\newblock {\em arXiv preprint arXiv:1509.08386, to appear in J. Reine Ang.
  Math.}, 2015.

\bibitem[Oki92]{Oki92}
K.~Okikiolu.
\newblock Characterization of subsets of rectifiable curves in {${\bf R}^n$}.
\newblock {\em J. London Math. Soc. (2)}, 46(2):336--348, 1992.

\bibitem[Pom86]{Pom86}
Ch. Pommerenke.
\newblock On conformal mapping and linear measure.
\newblock {\em J. Analyse Math.}, 46:231--238, 1986.

\bibitem[Sch07]{Sch07-TST}
R.~Schul.
\newblock Subsets of rectifiable curves in {H}ilbert space---the analyst's
  {TSP}.
\newblock {\em J. Anal. Math.}, 103:331--375, 2007.

\bibitem[Sem90]{Sem90}
S.~Semmes.
\newblock Analysis vs. geometry on a class of rectifiable hypersurfaces in
  {${\bf R}^n$}.
\newblock {\em Indiana Univ. Math. J.}, 39(4):1005--1035, 1990.

\bibitem[Ste93]{Big-Stein}
E.~M. Stein.
\newblock {\em Harmonic analysis: real-variable methods, orthogonality, and
  oscillatory integrals}, volume~43 of {\em Princeton Mathematical Series}.
\newblock Princeton University Press, Princeton, NJ, 1993.
\newblock With the assistance of Timothy S. Murphy, Monographs in Harmonic
  Analysis, III.

\bibitem[Vil]{Vil19}
M.~Villa.
\newblock A topological condition guaranteeing the analyst's traveling salesman
  theorem.
\newblock {\em preprint}.

\bibitem[Wu86]{Wu86}
J-M. Wu.
\newblock On singularity of harmonic measure in space.
\newblock {\em Pacific J. Math.}, 121(2):485--496, 1986.

\end{thebibliography}
\end{document}